\documentclass{amsart}
\usepackage{amsmath,amsthm,amssymb,amsfonts}
\usepackage{mathdots}
\usepackage{latexsym}
\usepackage[all]{xy}
\usepackage{CJK}
\usepackage{bm}
\usepackage{graphicx}
\usepackage{subfigure}
\usepackage{mathrsfs}
\usepackage{float}
\usepackage{makecell}
\usepackage{multirow}
\usepackage{xcolor}
\usepackage{cite}
\usepackage{hyperref}
\numberwithin{equation}{section}
\newtheorem{theorem}{\bf{Theorem}}[section]
\newtheorem{proposition}[theorem]{\bf{Proposition}}

\newtheorem{corollary}[theorem]{\bf{Corollary}}
\newtheorem{lemma}[theorem]{\bf{Lemma}}
\newtheorem{remark}[theorem]{\bf{Remark}}
\newcommand{\mrgl}{\mathrm{GL}}
\newcommand{\mrm}{\mathrm{M}}
\newcommand{\mrn}{\mathrm{N}}

\newcommand{\mrdet}{\mathrm{det}}
\newcommand{\mrdim}{\mathrm{dim}}
\newcommand{\mrhom}{\mathrm{Hom}}

\newcommand{\mbz}{\mathbb{Z}}
\newcommand{\mbc}{\mathbb{C}}

\newcommand{\mfp}{\mathfrak{p}}
\newcommand{\mfo}{\mathfrak{o}}
\newcommand{\mfa}{\mathfrak{a}}
\newcommand{\mfb}{\mathfrak{b}}
\newcommand{\mfc}{\mathfrak{c}}
\newcommand{\mcs}{\mathcal{S}}
\newcommand{\mcc}{\mathcal{C}}

\newcommand{\bs}[1]{\boldsymbol{#1}}
\newcommand{\rn}{\romannumeral}
\begin{document}
\title{Supercuspidal representations of $\mathrm{GL}_{n}(F)$ distinguished by an orthogonal involution}
\author{Jiandi Zou}
\address{Universit\'e Paris-Saclay, UVSQ, CNRS, Laboratoire de Math\'ematiques de Versailles, 78000, Versailles, France.}	
\email{jiandi.zou@ens.uvsq.fr}
\maketitle
	
\begin{abstract}
		
Let $F$ be a non-archimedean locally compact field of residue characteristic $p\neq2$, let $G=\mathrm{GL}_{n}(F)$ and let $H$ be an orthogonal subgroup of $G$. For $\pi$ a complex smooth supercuspidal representation of $G$, we give a full characterization for the distinguished space $\mathrm{Hom}_{H}(\pi,1)$ being non-zero and we further study its dimension as a complex vector space, which generalizes a similar result of Hakim for tame supercuspidal representations. As a corollary, the embeddings of $\pi$ in the space of smooth functions on the set of symmetric matrices in $G$, as a complex vector space, is non-zero and of dimension four, if and only if the central character of $\pi$ evaluating at $-1$ is 1.
		
\end{abstract}
	
	\tableofcontents
	
	\section{Introduction}
	
	\subsection{Background}
	
	Let $F$ be a non-archimedean locally compact field of residue characteristic $p$, and let $H\subset G$ be algebraic groups over $F$ (we also use $G$, $H$ to denote their $F$-points by abuse of notation). One important question in the representation theory of $p$-adic groups is to study the right $G$-action on the space of uniformly locally constant functions on $H\backslash G$ with complex values, denoted by $\mcc^{\infty}(H\backslash G)$. In particular for any irreducible smooth representation $\pi$ of $G$, it is important to study
	$$\mathrm{Hom}_{G}(\pi,\mcc^{\infty}(H\backslash G))\simeq\mathrm{Hom}_{H}(\pi,1)$$
	as a complex vector space and its dimension. We call $\pi$ distinguished by $H$ if the above vector space is non-zero. We temporally assume that the local Langlands correspondence for $G$ is valid, which, roughly speaking, is a finite-to-one correspondence from the set of irreducible representations of $G$ to the set of $L$-parameters with respect to the $L$-group of $G$ satisfying certain ``desiderata" (see \cite{borel1979automorphic} for more details). For each $L$-parameter $\phi$, its inverse image is called an $L$-packet of $\phi$.
	The so-called ``relative Langlands correspondence" is a conjectural proposal, which believes that under good conditions, a certain irreducible representation in the $L$-packet of $\phi$ is $H$-distinguished if and only if certain properties of $\phi$ are satisfied. And more optimistically, the corresponding distinguished spaces for all the representations in that $L$-packet could be fully studied (see \cite{prasad2015arelative} for at least Galois case).
	
	In the remarkable book \cite{sakellaridis2017periods}, Sakellaridis and Venkatesh proposed a general framework to study the relative Langlands correspondence under the setting of spherical varieties. Let $F$ be a $p$-adic field, let $G$ be a split reductive group over $F$ and let $X=H\backslash G$ be a spherical variety over $F$. To sum up some of their results under local settings, they
	\begin{itemize}
		\item defined (section 2-3, \emph{ibid.}) a dual group $\check{G}_{X}$ for $X$, under an assumption on the roots of $X$, together with a canonical morphism
		$$\iota_{X}:\check{G}_{X}\times\mathrm{SL}_{2}(\mbc)\longrightarrow \check{G},$$
		where $\check{G}$ denotes the complex dual group of $G$;
		
		\item proved (section 5, \emph{ibid.}), under the wavefront condition, the finiteness of the dimension of the distinguished space with respect to any smooth irreducible representation $\pi$ of $G$, that is,
		$$\mathrm{dim}_{\mbc}\mathrm{Hom}_{G}(\pi,\mcc^{\infty}(X))<\infty;$$
		
		\item provided (section 6, \emph{ibid.}) a Plancherel formula for $L^{2}(X)$ under the wavefront and strongly tempered conditions, which enables the direct integral decomposition of $L^{2}(X)$ as tempered representations, where $L^{2}(X)$ being the space of the square integrable functions on $X$ is itself tempered as a $G$ representation;
		
		\item conjectured (section 16, \emph{ibid.}) that, under some reasonable assumptions, $L^{2}(X)$ has a direct integral decomposition, where each summand is isomorphic to the direct sum of irreducible representations belonging to the Arthur packet with corresponding Arthur parameter ($A$-parameter and $A$-packet for short) factoring through $\iota_{X}$.
		
	\end{itemize}
	Recall that $A$-parameter and $A$-packet are generalizations of $L$-parameter and $L$-packet respectively which are more suitable for global considerations. Those $A$-parameters factoring through $\iota_{X}$ are called $X$-distinguished. So the result mentioned above provides a clear correspondence between distinguished representations on the $p$-adic side, and $X$-distinguished $A$-parameters on the Galois side, which seems to be a good starting point to study ``relative Langlands correspondence".
	
	Now we focus on a special case, in which $G=\mrgl_{n}$ and $H$ is the orthogonal subgroup of $G$ fixed by an orthogonal involution defined over $F$. It is of special interest because it is in some sense out of the reach of the consideration of Sakellaridis and Venkatesh. The main reason is that the assumption of the first statement mentioned above is not satisfied, marking the failure of defining the dual group $\check{G}_{X}$. However the conditions in the second and third statements are satisfied (see \cite{vust1990plongements}, Proposition 2.4. for being wavefront and \cite{gurevich2016criterion} for being strongly tempered), thus it is still of great interest to study all the tempered representations of $G$ distinguished by a certain $H$, where the dimension of the distinguished space is finite by the second statement. If we write $\mcs$ for the set of invertible symmetric matrices as a topological subspace of $G$, which is endowed with a continuous right $G$-action as follows:
	$$\varepsilon\cdot g:=\,^{t}g\varepsilon g,\quad g\in G,\ \varepsilon\in\mathcal{S},$$
	then we have the following decomposition as $G$-spaces
	$$\mcs=\bigoplus_{[\varepsilon]}H_{\varepsilon}\backslash G,$$
	where the direct sum ranges over $\mcs/G$, and $H_{\varepsilon}$ is the orthogonal group defined by a certain representative $\varepsilon$ in the class $[\varepsilon]$. A more uniformed version of the above problem is to study the space
	\begin{equation}\label{eqpispihepsilon}
		\mathrm{Hom}_{G}(\pi,\mcc^{\infty}(\mcs))\simeq\bigoplus_{[\varepsilon]}\mathrm{Hom}_{G}(\pi,\mathrm{Ind}^{G}_{H_{\varepsilon}}1)\simeq\bigoplus_{[\varepsilon]}\mathrm{Hom}_{H_{\varepsilon}}(\pi,1),
	\end{equation}
	at least for tempered representations $\pi$ of $G$, and to determine a criterion for the space being non-zero and to study the corresponding dimension, and finally, to provide a concrete direct integral decomposition for $L^{2}(X)$ via that criterion.
	
	The study of this special example was first proposed by Jacquet \cite{jacquet1991representations}. The idea is first to consider the global analogue of the same question, and then to initiate a global to local argument. The key point is to compare two trace formulae: one relates to the relative trace formula for symmetric matrices or orthogonal groups, and the other relates to the Kuznetsov trace formula for the two-fold metaplectic covering of $\mrgl_{n}$ (see \cite{mao1998fundamental} for a brief introduction). Then usually the routine procedure proposed by Jacquet is as follows:
	\begin{itemize}
		\item Prove the smooth transfer and fundamental lemma for the geometric sides of the two trace formulae at local places, then the two trace formulae are equal for good matching pairs of test functions;
		
		\item Calculate the spectral sides of the two trace formulae;
		
		\item Study the possible factorization for the terms of both spectral sides into a product of local components.
		
	\end{itemize}
	
	However, for this specific question each step seems to be difficult and only partial results are known, for which we provide a brief summary for ease of future research. In \cite{offen2005kloosterman}, Offen followed Jacquet's argument \cite{jacquet2003smooth} to consider the Kloosterman-Fourier transform for orbital integrals with respect to symmetric matrices, which might be a partial step to prove the existence of smooth transfer in the non-archimedean case, and the corresponding archimedean case remains a mystery. For the fundamental lemma for unit Hecke elements, Mao \cite{mao1998fundamental} gave a proof, for $n=3$, by direct calculation and Do first proved, for general $n$, the analogue for local fields of positive characteristic via geometric method \cite{do2015lemma}, and then he transferred the result to $p$-adic fields for $p$ large enough \cite{do2018transfer}. However for ease of later applications, a stronger version of fundamental lemma working with general Hecke elements is needed but remains unknown. The spectral sides of both trace formulae are less studied. Partial results due to Chinta and Offen \cite{chinta2012orthogonal}, \cite{chinta2013metaplectic}, on the one hand, shed some light on the spectral expansions, but on the other hand, indicate the difficulty of resolving the full question. In particular, since the local Whittaker model for the two-fold metaplectic covering group of $\mrgl_{n}$ is not unique, the terms of the spectral side of Kuznetsov trace formula are not factorizable, adding the difficulty to the global to local argument. Instead of understanding the full question, it should also be fruitful if enlightening partial results or even reasonable guesses could be made, which is the work in progress of the author.
	
	Another strategy starts from studying the distinction of supercuspidal representations, and then uses parabolic induction to get at least some partial results for more general representations. For the study of a supercuspidal representation $\pi$, the rough idea is first to regard it as the compact induction of a finite dimensional representation $\Lambda$ of an open subgroup $\bs{J}$ of $G$ which is compact modulo its centre, and then to use the Mackey formula and the Frobenius reciprocity to write the original distinguished space as a direct product, ranging over the double cosets in $\bs{J}\backslash G/H$, of distinguished spaces with respect to $\Lambda$. Under the assumption that $p\neq 2$, the question is completely addressed by Hakim and Mao \cite{hakim1999cuspidal} when $\pi$ is of level 0 and by Hakim and Lansky \cite{hakim2012distinguished} and Hakim \cite{hakim2013tame} when $\pi$ is tamely ramified. The goal of this article is to generalize their results to all supercuspidal representations of $G$, which we explain in the following subsection.
	
	\subsection{Statement of the main theorems}
	
	Let $F$ be a non-archimedean locally compact field of residue characteristic $p\neq 2$ and let $G=\mrgl_{n}(F)$. For $\pi$ a supercuspidal representation of $G$, we recall several invariants given by the simple type theory of Bushnell-Kutzko \cite{bushnell129admissible} and the theory of endo-class of Bushnell-Henniart \cite{bushnell1996local}, for which we refer to \S \ref{subsectiontypetheory} below for more details. First of all, there is a unique tamely ramified extension $T/F$ up to $F$-isomorphism, called the tame parameter field of $\pi$. We write $d$ for the degree of the endo-class of $\pi$ which divides $n$ and is divided by $[T:F]$. We write $m$ for the integer such that $n=md$. Let $T_{m}$ be the unramified extension of degree $m$ over $T$. Here $T$, $d$, $m$, $T_{m}$ are intrinsically determined by $\pi$.
	
	To give an impression of what these invariants should be, we let $\varphi_{\pi}$ be the irreducible representation of the Weil group $\mathcal{W}_{F}$ corresponding to $\pi$ via the local Langlands correspondence. Then the restriction of $\varphi_{\pi}$ to the wild inertia subgroup $\mathcal{P}_{F}$ of $\mathcal{W}_{F}$ is semi-simple and can be written as a direct sum of irreducible representations with each irreducible component of multiplicity exactly $m$. We choose $\alpha$ to be any irreducible component of $\varphi_{\pi}|_{\mathcal{P}_{F}}$, then there exists a finite tamely ramified extension $T/F$ such that
	$$N_{F}(\alpha):=\{g\in\mathcal{W}_{F}|\alpha^{g}\simeq\alpha\}$$
	as a subgroup of $\mathcal{W}_{F}$ equals $\mathcal{W}_{T}$. And it turns out that $T/F$ is uniquely determined up to an $F$-isomorphism and independent of the choice of $\alpha$. We let $n=\mathrm{dim}(\varphi_{\pi})$, $d=n/m$ and $T_{m}$ be as above, then $T$, $d$, $m$, $T_{m}$ defined from the Galois side match with those defined from the $\mrgl_{n}$ side (see \cite{bushnell2014effective} for more details).
	
	For $\varepsilon$ a symmetric matrix in $G$, we denote by
	$$\tau_{\varepsilon}(x)=\varepsilon^{-1}\,^{t}x^{-1}\varepsilon\quad\text{for any}\ x\in G$$
	the orthogonal involution with respect to $\varepsilon$, and by $G^{\tau_{\varepsilon}}$ the orthogonal subgroup of $G$ composed of fixed points of $\tau_{\varepsilon}$. We have the following theorem as a criterion for distinction.
	
	\begin{theorem}\label{thmdistcri}
		
		Let $\pi$ be a supercuspidal representation of $G$ and let $T$, $d$, $m$, $T_{m}$ be as above. Then $\pi$ is distinguished by an orthogonal subgroup $H$ if and only if the following two conditions hold:
		\begin{enumerate}
			\item $\omega_{\pi}(-1)=1$, where $\omega_{\pi}$ denotes the central character of $\pi$;
			
			\item Precisely one of the following conditions holds:
			\begin{itemize}
				
				\item $\mrn_{T_{m}/F}(T_{m}^{\times})F^{\times2}/F^{\times2}=F^{\times}/F^{\times2}$ and $H$ is split;
				\item $\mrn_{T_{m}/F}(T_{m}^{\times})F^{\times2}/F^{\times2}$ is a subgroup of $F^{\times}/F^{\times2}$ of order 2 and $H$ is either split or $H=G^{\tau_{\varepsilon}}$ which is quasisplit but not split, where $\varepsilon$ is a symmetric matrix such that $(-1)^{n(n-1)/2}\mrdet(\varepsilon)\in\mrn_{T_{m}/F}(T_{m}^{\times})-F^{\times2}$;
				\item $\mrn_{T_{m}/F}(T_{m}^{\times})F^{\times2}/F^{\times2}=\{1\}$ and $H$ is either split or not quasisplit.
				
			\end{itemize}
			
		\end{enumerate}
		
	\end{theorem}
	
	In particular, it is easily seen that:
	
	\begin{corollary}
		
		When $H$ is split, $\pi$ is distinguished by $H$ if and only if $\omega_{\pi}(-1)=1$.
		
	\end{corollary}
	
	Moreover, the following theorem calculates the dimension of the distinguished space.
	
	\begin{theorem}\label{thmdistdim}
		
		Let $\pi$ be a supercuspidal representation of $G$ such that $\omega_{\pi}(-1)=1$ and let $H$ be an orthogonal subgroup satisfying the condition 2 of Theorem \ref{thmdistcri}.
		
		\begin{enumerate}
			\item If $H$ is not split, then $\mrdim_{\mbc}\mrhom_{H}(\pi,1)=1$;
			
			\item If $H$ is split, then
			\begin{itemize}
				\item If $\mrn_{T_{m}/F}(T_{m}^{\times})F^{\times2}/F^{\times2}=F^{\times}/F^{\times2}$, then $\mrdim_{\mbc}\mrhom_{H}(\pi,1)=1$;
				\item If $\mrn_{T_{m}/F}(T_{m}^{\times})F^{\times2}/F^{\times2}$ is a subgroup of $F^{\times}/F^{\times2}$ of order 2, then $\mrdim_{\mbc}\mrhom_{H}(\pi,1)=2$;
				\item If $\mrn_{T_{m}/F}(T_{m}^{\times})F^{\times2}/F^{\times2}=\{1\}$, then $\mrdim_{\mbc}\mrhom_{H}(\pi,1)=3$.
			\end{itemize}
			
		\end{enumerate}
		
	\end{theorem}
	
	Finally using (\ref{eqpispihepsilon}) and the same argument in \cite{hakim2013tame}, section 1.4, the following theorem holds as a corollary of Theorem \ref{thmdistdim}.
	
	\begin{theorem}\label{thmsymmatmod}
		
		For $\pi$ a supercuspidal representation of $G$, it is distinguished by a certain orthogonal subgroup if and only if $\omega_{\pi}(-1)=1$. Moreover, if this condition holds, then
		$$\mrdim_{\mbc}\mrhom_{G}(\pi,\mcc^{\infty}(\mcs))=4.$$
		
	\end{theorem}
	
	Thus for $p\neq 2$ and any supercuspidal representation $\pi$ of $G=\mrgl_{n}(F)$, the problem of distinction for orthogonal subgroups is fully settled. The only restriction on $\pi$, being the triviality of its central character on $-1$, can also be rephrased as the triviality of the determinant character of its Langlands parameter on $-1$ via the local Langlands correspondence for $\mrgl_{n}$.
	
	\subsection{Sketch of the proof and the structure of the article}
	
	We sketch the proof and the structure of the article. As already mentioned above, for supercuspidal representation $\pi$ of $G$, the first step is to write $\pi$ as the compact induction of a finite dimensional representation $\Lambda$ of an open subgroup $\bs{J}$ of $G$. This is exactly one of the main results of the simple type theory built up by Bushnell-Kutzko (\cite{bushnell129admissible}) and such a pair $(\bs{J},\Lambda)$ is called a simple type. We briefly recall the simple type theory in section 2. And in section 3 we build up necessary results for symmetric matrices, orthogonal involutions and orthogonal groups for future use.
	
	In section 4 we prove our first main theorem, the tau-selfdual type theorem, which says that for a certain well-chosen orthogonal involution $\tau_{0}$ depending on $\pi$, there exists a simple type $(\bs{J},\Lambda)$ compactly inducing $\pi$ such that $\tau_{0}(\bs{J})=\bs{J}$ and $\Lambda\circ\tau_{0}=\Lambda^{\vee}$, where $\Lambda^{\vee}$ denotes the contragredient of $\Lambda$. In fact, for each orthogonal group $H$ satisfying Theorem \ref{thmdistcri}, condition 2, we may find a $\tau_{0}$ satisfying $H=G^{\tau_{0}}$ and the tau-selfdual theorem.
	Such a simple type is called $\tau_{0}$-selfdual and will be regarded as the starting point to pursue the problem of distinction.
	
	In section 5, we study the distinction with respect to an arbitrary orthogonal involution $\tau$ and the corresponding orthogonal group $G^{\tau}$. We fix a $\tau_{0}$-selfdual simple type $(\bs{J},\Lambda)$ and we use the Mackey formula and Frobenius reciprocity to write the distinguished space as follows:
	$$\mathrm{Hom}_{G^{\tau}}(\pi,1)\simeq\prod_{g\in\bs{J}\backslash G/G^{\tau}}\mathrm{Hom}_{\bs{J}^{g}\cap G^{\tau}}(\Lambda^{g},1).$$
	The distinguished type theorem says that for those double cosets $g\in\bs{J}\backslash G/G^{\tau}$ contributing to the distinction, the simple type $(\bs{J}^{g},\Lambda^{g})$ is
	$\tau$-selfdual. In particular, when $\tau=\tau_{0}$  we may also give out all the possible $\bs{J}$-$G^{\tau_{0}}$ double cosets contributing to the distinction.
	
	Finally in section 6, we continue to study the distinguished space $\mathrm{Hom}_{\bs{J}^{g}\cap G^{\tau}}(\Lambda^{g},1)$. The techniques developed in section 5 enable us to further study the distinguished space via the more delicate structure given by the simple type theory, and finally reduce the question to study the distinguished space $\mathrm{Hom}_{\overline{H}}(\overline{\rho},\overline{\chi})$, where $\overline{H}$ is an orthogonal subgroup of a finite general linear group $\overline{G}=\mrgl_{m}(\mathbb{F}_{q})$, and $\overline{\rho}$ is a supercuspidal representation of $\overline{G}$, and $\overline{\chi}$ is a character of $\overline{H}$ of order 1 or 2. Using the Deligne-Lusztig theory, the condition for the space being non-zero is given and the dimension is at most one. The condition turns out to be the central character of $\pi$ being trivial at $-1$. Thus for those special $\tau_{0}$ in section 4, we may fully study the distinguished space and the corresponding dimension. Since those $\tau_{0}$ correspond exactly to the orthogonal groups in Theorem \ref{thmdistcri} and Theorem \ref{thmdistdim}, we prove the ``if" part of Theorem \ref{thmdistcri} and Theorem \ref{thmdistdim}.
	
	It remains the ``only if" part of Theorem \ref{thmdistcri}, of which we take advantage to explain the condition for the orthogonal groups or corresponding orthogonal involutions in the theorem. For $E_{m}/F$ an extension of degree $n$ and $\tau$ an orthogonal involution, we call $E_{m}$ $\tau$-split if there exists an embedding $\iota:E_{m}^{\times}\hookrightarrow\mrgl_{n}(F)$ such that $\tau(\iota(x))=\iota(x)^{-1}$ for any $x\in E_{m}^{\times}$.
	The following intermediate proposition gives important information for $\pi$ being distinguished by $G^{\tau}$:
	
	\begin{proposition}\label{propintermid}
		
		For $\pi$ a given supercuspidal representation of $G$ with $\omega_{\pi}(-1)=1$, there exists a field $E_{m}$ of degree $n$ over $F$ which is totally wildly ramified over $T_{m}$, such that if $\pi$ is distinguished by $G^{\tau}$, then $E_{m}$ is $\tau$-split.
		
	\end{proposition}
	
	The construction of $E_{m}$ is derived from the construction of a $\tau_{0}$-selfdual simple type given in section 4. In particular, when $\tau_{0}$ corresponds to a split orthogonal group, from the ``if" part of Theorem \ref{thmdistcri}, $E_{m}$ is $\tau_{0}$-split. Once knowing this, it is not hard to study all the involutions $\tau$ such that $E_{m}$ is $\tau$-split, which turn out to be those involutions satisfying the condition of Theorem \ref{thmdistcri}, proving the ``only if" part of the theorem.
	
	When $T_{m}/F$ is of degree $n$, or equivalently when $\pi$ is essentially tame in the sense of Bushnell-Henniart \cite{bushnell2005essentially1}, which is the same as being tamely ramified in the context of Hakim \cite{hakim2013tame} thanks to the work of Mayeux \cite{mayeux2020comparison}, our result gives another proof for the result of Hakim by using the simple type theory instead of Howe's construction for tamely ramified representations. 
	
	It is worth to point out that we also borrow many lemmas from \cite{hakim1999cuspidal}, \cite{hakim2012distinguished}, \cite{hakim2013tame}, which effectively help us to reduce our task.
	
	Last but not least, it should also be pointed out that the method we use in this article is not new. It has first been developed by S\'echerre to solve the similar problem where $\tau$ is a Galois involution \cite{anandavardhanan2019galois}, \cite{secherre2019supercuspidal}, and then by the author for the the case where $\tau$ is a unitary involution \cite{zou2019supercuspidal}, and then by S\'echerre for the case where $\tau$ is an inner involution \cite{secherre2020repr} (there $G$ can also be an inner form of $\mrgl_{n}(F)$). The sketches of the proof in different cases are similar, but one major difference in the current case is worth to be mentioned, that is, we need to consider those involutions $\tau$ not contributing to the distinction. In this case we cannot construct a $\tau$-selfdual simple type $(\bs{J},\Lambda)$ using the method in section 4. The novelty of our argument is first to consider a special involution $\tau_{0}$, and then to regard $\tau$ as another involution which differs from $\tau_{0}$ up to a $G$-conjugation. Thus we choose $(\bs{J},\Lambda)$ to be a $\tau_{0}$-selfdual simple type and, using the general results built up by the author in \cite{zou2019supercuspidal}, we can still study those $\bs{J}$-$G^{\tau}$ double cosets contributing to the distinction. If one wants to fit the method in the above cases to a general involution $\tau$, one major problem encountered is to construct a $\tau$-selfdual simple type, which, as we explained, may be impossible if $G^{\tau}$ does not contribute to the distinction. The strategy we explained above gives a possible solution, which helps to consider the same question for an abstract involution.
	
	\subsection{Acknowledgement}
	
	We thank Vincent S\'echerre for careful reading and useful comments. We thank Nadir Matringe and Rapha\"el Beuzart-Plessis for helpful discussions for the possible generalization of the current article. We thank Erez Lapid for pointing out that Theorem \ref{thmmaintwist} could be a by-product of our argument. We thank an anonymous referee for his (her) helpful advice which clarifies many ambiguous points. The work is supported by EDMH as part of the PhD thesis of the author.
	
	\section{Notation}
	
	\subsection{General notation}
	
	Let $F$ be a non-archimedean locally compact field of residue characteristic $p\neq 2$. We write $\mfo_{F}$, $\mfp_{F}$, $\boldsymbol{k}$ for its ring of integers, the corresponding maximal ideal and its residue field respectively. We fix $\psi_{F}:F\rightarrow\mathbb{C}^{\times}$ an additive character which is trivial on $\mfp_{F}$ but not on $\mfo_{F}$.
	
	Fix $n$ a positive integer. We write $G=\mrgl_{n}(F)$ as a locally profinite group. By \emph{representations} of $G$ and its closed subgroups, we always mean complex smooth representations. For a closed subgroup $H$ of $G$, an element $g\in G$ and a representation $\pi$ of $H$, we write $H^{g}:=\{g^{-1}hg|h\in H\}$ a subgroup of $G$, and $\pi^{g}:g\mapsto\pi(ghg^{-1})$ its representation. We write $\pi^{\vee}$ for the contragredient of $\pi$. Given $\tau$ a continuous involution of $G$, we write $\pi^{\tau}$ for the representation $\pi\circ\tau$ of $\tau(H)$. We say that $\pi$ is $\tau$-selfdual if $\tau(H)=H$ and $\pi^{\tau}\simeq\tau^{\vee}$.
	
	Given $\pi$ a representation of $H$ and $\mu$ a representation of $G^{\tau}\cap H$, we say that $\pi$ is $\mu$-\emph{distinguished} if $\mathrm{Hom}_{G^{\tau}\cap H}(\pi,\mu)\neq 0$, where $G^{\tau}$ denotes the subgroup of $G$ consisting of the elements fixed by $\tau$. In particular, if $\mu$ is the trivial character, we simply call $\pi$ to be $G^{\tau}\cap H$-\emph{distinguished}.
	
	\subsection{A brief recall of the simple type theory}\label{subsectiontypetheory}
	
	In this subsection, we follow the introduction of the simple type theory given in \cite{zou2019supercuspidal}, section 3 summarizing results of \cite{bushnell129admissible}, \cite{bushnell1996local}, \cite{bushnell2014effective}. Since it seems redundant to repeat the same words again, we simply recall the necessary notation and refer to \cite{zou2019supercuspidal} for more details.
	
	We write $[\mathfrak{a},\beta]$ for a \emph{simple stratum} in $\mathrm{M}_{n}(F)$, where $\mathfrak{a}$ is a hereditary order in $\mathrm{M}_{n}(F)$ and $\beta$ is an element in $\mathrm{GL}_{n}(F)$ such that
	
	(1) the $F$-algebra $E=F[\beta]$ is a field, where $[E:F]=d$ and $n=md$ for a positive integer $m$;
	
	(2) $E^{\times}$ normalizes $\mathfrak{a}^{\times}$.
	
	We write $B$ for the centralizer of $E$ in $\mrm_{n}(F)$ identifying with $\mrm_{m}(E)$, and $\mathfrak{b}=\mathfrak{a}\cap B$ for the hereditary order in $B$. We denote by $\mathfrak{p}_{\mathfrak{a}}$ (resp. $\mathfrak{p}_{\mathfrak{b}}$) the Jacobson radical of $\mathfrak{a}$ (resp. $\mfb$), and $U^{1}(\mathfrak{a})$ (resp. $U^{1}(\mathfrak{b})$) the compact open pro-$p$-subgroup $1+\mathfrak{p}_{\mathfrak{a}}$ (resp. $1+\mathfrak{p}_{\mathfrak{b}}$) of $\mathrm{GL}_{n}(F)$ (resp. $B^{\times}$).
	
	Associated to $[\mathfrak{a},\beta]$, there are compact open subgroups
	$$H^{1}(\mathfrak{a},\beta)\subset J^{1}(\mathfrak{a},\beta)\subset J(\mathfrak{a},\beta)$$
	of $\mathfrak{a}^{\times}$, and there is a finite set $\mathcal{C}(\mathfrak{a},\beta)$ of characters of $H^{1}(\mathfrak{a},\beta)$, depending on the choice of $\psi_{F}$, called \emph{simple characters}. We denote by $\boldsymbol{J}(\mathfrak{a},\beta)$ the subgroup of $G$ generated by $J(\mathfrak{a},\beta)$ and the normalizer of $\mathfrak{b}^{\times}$ in $B^{\times}$ which is compact modulo the centre $F^{\times}$. We write $\boldsymbol{J}$, $J$, $J^{1}$, $H^{1}$ for short for $\boldsymbol{J}(\mathfrak{a},\beta)$, $J(\mathfrak{a},\beta)$, $J^{1}(\mathfrak{a},\beta)$, $H^{1}(\mathfrak{a},\beta)$ respectively if $\mathfrak{a}$ and $\beta$ are clear to us. When $\mathfrak{b}$ is a maximal order in $B$, we call the simple stratum $[\mathfrak{a},\beta]$ and the simple characters in $\mathcal{C}(\mathfrak{a},\beta)$  \emph{maximal}. In this case $\mathfrak{b}^{\times}/1+\mathfrak{p}_{\mathfrak{b}}\simeq\mrgl_{m}(\bs{l})$, where $\bs{l}$ is the residue field of $E$.
	
	We denote by $(\boldsymbol{J},\Lambda)$ an \emph{extended maximal simple type} (we always write \emph{simple type} for short) in $\mrgl_{n}(F)$, which means that there are a maximal simple stratum $[\mathfrak{a},\beta]$ in $\mathrm{M}_{n}(F)$ and a maximal simple character $\theta\in\mathcal{C}(\mathfrak{a},\beta)$ such that $\boldsymbol{J}(\mathfrak{a},\beta)=\boldsymbol{J}$ and $\theta$ is contained in the restriction of $\Lambda$ to $H^{1}$. We write $\eta$ for the \emph{Heisenberg representation} associated to $\theta$ as a representation of $J^{1}$. For any representation $\boldsymbol{\kappa}$ of $\boldsymbol{J}$ extending $\eta$, there is, up to isomorphism, a unique irreducible representation $\boldsymbol{\rho}$ of $\boldsymbol{J}$ such that $\Lambda\simeq\boldsymbol{\kappa}\otimes\boldsymbol{\rho}$, and moreover $\bs{\rho}|_{J}$ is the inflation of a supercuspidal representation of $J/J^{1}\simeq\mrgl_{m}(\bs{l})$. For $\pi$ a supercuspidal representation of $G$, there exists a unique $G$-conjugacy class of simple type $(\boldsymbol{J},\Lambda)$ such that $\pi\simeq$$c$-$\mathrm{Ind}_{\boldsymbol{J}}^{G}\Lambda$, the compact induction of $\Lambda$.
	
	For $[\mathfrak{a},\beta]$ a simple stratum in $\mathrm{M}_{n}(F)$ and $[\mathfrak{a}',\beta']$ a simple stratum in $\mathrm{M}_{n'}(F)$ with $n,n'\geq 1$, if we have a given $F$-algebra isomorphism $\phi: F[\beta]\rightarrow F[\beta']$ such that $\phi(\beta)=\beta'$, then we denote by
	$$t_{\mathfrak{a},\mathfrak{a}'}^{\beta,\beta'}:\mathcal{C}(\mathfrak{a},\beta)\rightarrow\mathcal{C}(\mathfrak{a}',\beta')$$ the corresponding \emph{transfer map}. For $n_{1},n_{2}\geq 1$ and  $[\mfa_{1},\beta_{1}]$ and $[\mfa_{2},\beta_{2}]$ two simple strata in $\mrm_{n_{1}}(F)$ and $\mrm_{n_{2}}(F)$ respectively, $\theta_{1}\in\mcc(\mfa_{1},\beta_{1})$ and $\theta_{2}\in\mcc(\mfa_{1},\beta_{1})$ are called endo-equivalent if there exist $n'\geq 1$, $[\mfa',\beta_{1}]$, $[\mfa',\beta_{2}']$ two simple strata in $\mrm_{n'}(F)$ with $F[\beta_{1}]\simeq F[\beta_{1}']$ and $F[\beta_{2}]\simeq F[\beta_{2}']$ mapping $\beta_{1}$ to $\beta_{1}'$ and $\beta_{2}$ to $\beta_{2}'$, such that $t_{\mfa_{1},\mfa'}^{\beta_{1},\beta_{1}'}(\theta_{1})$ is $\mrgl_{n'}(F)$-conjugate to $t_{\mfa_{2},\mfa'}^{\beta_{2},\beta_{2}'}(\theta_{2})$. A corresponding equivalence class on the set of simple characters is called an endo-class. Usually we use capital Greek letter $\Theta$ to denote the \emph{endo-class} of a simple character $\theta$ and $\Theta_{\pi}$ to denote the endo-class of $\pi$, a supercuspidal representation of $G$. We write $d=[F[\beta]:F]$ for the degree of $\Theta$ which does not depend on the choice of $[\mfa,\beta]$ and $\theta$, but only on $\Theta$ itself.

	
	Let $\Theta$ be as above and let $T$ be its tame parameter field with respect to $E/F$, that is, the maximal tamely ramified subextension of $E$ over $F$. Noting that $T$ only depends on $\Theta$ up to $F$-isomorphism, so it is also called the tame parameter field of $\Theta$. Let $C\simeq\mathrm{M}_{n/t}(T)$ denote the centralizer of $T$ in $\mathrm{M}_{n}(F)$, where $t=[T:F]$. The intersection $\mathfrak{c}=\mathfrak{a}\cap C$ is an order in $C$, which gives rise to a simple stratum $[\mathfrak{c},\beta]$. The restriction of $\theta$ to $H^{1}(\mathfrak{c},\beta)$, denoted by $\theta_{T}$ and called the \emph{interior $T/F$-lift of $\theta$}, is a simple character associated to the simple stratum $[\mathfrak{c},\beta]$. If we change our choice of simple stratum $[\mathfrak{a},\beta]$ but fix $T\hookrightarrow\mathrm{M}_{n}(F)$ unchanged, then the map
	$$\mathfrak{a}\mapsto\mathfrak{a}\cap C$$
	is injective from the set of hereditary orders in $\mathrm{M}_{n}(F)$ normalized by $T^{\times}$ to the set of hereditary orders in $C$ (see \cite{bushnell1996local}, section 2). For $[\mfa,\beta_{1}]$, $[\mfa,\beta_{2}]$ two simple strata, and $\theta_{1}\in \mcc(\mfa,\beta_{1})$, $\theta_{2}\in \mcc(\mfa,\beta_{2})$ two simple characters, such that $\theta_{1}$ and $\theta_{2}$ have the same tame parameter field $T$, if
	$$\mcc(\mfc,\beta_{1})= \mcc(\mfc,\beta_{2})\quad \text{and}\quad(\theta_{1})_{T}=(\theta_{2})_{T},$$ then we have $$\mcc(\mfa,\beta_{1})=\mcc(\mfa,\beta_{2}) \quad\text{and}\quad \theta_{1}=\theta_{2}$$ (see \cite{bushnell1996local}, Theorem 7.10, Theorem 7.15). In particular, when $\beta_{1}=\beta_{2}=\beta$, the interior $T/F$-lift is injective from $\mcc(\mfa,\beta)$ to $\mcc(\mfc,\beta)$.
	
	\section{Symmetric matrices and orthogonal involutions}
	
	In this section, we recall some basic but important results about symmetric matrices and orthogonal involutions. Let $E$ be a non-archimedean locally compact field of residue characteristic $p\neq 2$, let $\varpi_{E}$ be a uniformizer of $E$ and let $m$ be a fixed positive integer.
	
	\subsection{Orbits of symmetric matrices, orthogonal involutions and orthogonal groups}
	
	Let $\mathcal{S}$ denote the set of the symmetric matrices in $\mathrm{GL}_{m}(E)$, that is
	$$\mathcal{S}:=\{\varepsilon\in\mrgl_{m}(E)|\,^{t}\varepsilon=\varepsilon\}.$$
	Especially, if we write
	$$J_{m}:=\begin{pmatrix}0 & 0 & \ldots & 0 & 1\\ 0 & \iddots & \iddots & 1 & 0 \\ \vdots & \iddots & \iddots & \iddots & \vdots \\ 0 & 1 & \iddots & \iddots & 0 \\ 1 & 0 & \ldots & 0 & 0 \end{pmatrix}\in\mrgl_{m}(E),$$
	then it is an element in $\mcs$.
	
	We consider $\mrgl_{m}(E)$-action on $\mathcal{S}$ as follows:
	$$ \varepsilon\cdot g:=\,^{t}g\varepsilon g,\quad g\in \mrgl_{m}(E),\ \varepsilon\in\mcs.$$ We say that two elements in $\mathcal{S}$ are \emph{similar} if they are in the same $\mrgl_{m}(E)$-orbit.
	For $\varepsilon\in\mathcal{S}$, we denote by $\mathrm{disc}_{E}(\varepsilon)$ its \emph{discriminant}, which is the image of $\mathrm{det}(\varepsilon)$ in $E^{\times}/E^{\times2}\simeq\mbz/2\mbz\times\mbz/2\mbz$. We denote by\footnote{In \cite{hakim2013tame} Hakim used $i\leq j$ instead of $i<j$ in the product for the definition, however in the proof of various propositions (for example, Proposition 6.6. of \emph{ibid.}) he indeed used the second definition ($i<j$). This little inconsideration of course does not affect his results and proofs.}
	$$\mathrm{Hasse}_{E}(\varepsilon)=\prod_{i< j}\mathrm{Hil}_{E}(a_{i},a_{j})\in\{1,-1\}$$
	its \emph{Hasse invariant}, where $\mathrm{diag}(a_{1},...,a_{m})$ denotes a diagonal matrix similar to $\varepsilon$, and
	$$\mathrm{Hil}_{E}(a,b)=\begin{cases}1,\quad &\text{if}\ ax^{2}+by^{2}=1\ \text{has a solution}\ (x,y)\in E\times E;\\
		-1,\quad &\text{otherwise}. \end{cases}$$
	denotes the \emph{Hilbert symbol} for $a, b\in E^{\times}$. Noting that the definition of $\mathrm{Hasse}_{E}(\varepsilon)$ does not depend on the choice of $\mathrm{diag}(a_{1},...,a_{m})$ similar to $\varepsilon$ (see \cite{o2013introduction}, 63.13). When $E$ is clear to us, we simply write $\mathrm{disc}$, $\mathrm{Hil}$ and $\mathrm{Hasse}$ instead.
	
	The following proposition characterizes all the $\mrgl_{m}(E)$-orbits in $\mathcal{S}$.
	
	\begin{proposition}[\cite{o2013introduction}, Theorem 63.20]\label{propSGLEorbit}
		
		(1) When $m=1$, there are four $\mathrm{GL}_{m}(E)$-orbits in $\mathcal{S}$ represented by elements in $E^{\times}/E^{\times2}$;
		
		(2) When $m\geq 2$, any two $\mathrm{GL}_{m}(E)$-orbits in $\mathcal{S}$ are different if and only if their discriminants or Hasse invariants are different. Moreover,
		\begin{itemize}
			\item When $m\geq3$ there are eight $\mathrm{GL}_{m}(E)$-orbits;
			\item When $m=2$, any $\varepsilon\in\mcs$ with $\mathrm{disc}(\varepsilon)=-1$ satisfies $\mathrm{Hasse}(\varepsilon)=1$, and there are seven $\mathrm{GL}_{m}(E)$-orbits.
		\end{itemize}
		
	\end{proposition}
	
	We may also consider $\mathrm{GL}_{m}(\mathfrak{o}_{E})$-orbits of $\mcs$. We consider $\alpha=(\alpha_{1},...,\alpha_{r})$ of certain triples $\alpha_{i}=(a_{i},m_{i},\epsilon_{i})$, such that $a_{1}>...>a_{r}$ is a decreasing sequence of integers, and $m_{1},...,m_{r}$ are positive integers such that $m_{1}+...+m_{r}=m$, and $\epsilon_{1},...,\epsilon_{r}$ are either 1 or $\epsilon_{0}$, where $\epsilon_{0}\in \mfo_{E}^{\times}\backslash\mfo_{E}^{\times2}$ is fixed. For each $\alpha=(\alpha_{1},...,\alpha_{r})$ as above, we introduce a symmetric matrix $$\varpi_{E}^{\alpha}=\varpi_{E}^{\alpha_{1}}\oplus...\oplus\varpi_{E}^{\alpha_{r}},$$
	where  $$\varpi_{E}^{\alpha_{i}}:=\varpi_{E}^{a_{i}}\mathrm{diag}(1,...,1,\epsilon_{i})\in \mrgl_{m_{i}}(E).$$
	The following proposition studies all the $\mrgl_{m}(\mfo_{E})$-orbits.
	
	\begin{proposition}[\cite{o2013introduction}, \S92]\label{propSGLoEorbit}
		
		Each $\mrgl_{m}(\mfo_{E})$-orbit in $\mcs$ contains exactly one representative of the form $\varpi_{E}^{\alpha}$ defined as above.
		
	\end{proposition}
	
	Now for $\varepsilon\in\mcs$ a given symmetric matrix, we denote by $$\tau_{\varepsilon}(x):=\varepsilon^{-1}\,^{t}x^{-1}\varepsilon\quad\text{for any}\ x\in\mrgl_{m}(E)$$ the orthogonal involution corresponding to $\varepsilon$. The group $\mrgl_{m}(E)$ acts on the set of orthogonal involutions by
	$$g\cdot \tau_{\varepsilon}=\tau_{\varepsilon\cdot g}=\tau_{\,^{t}g\varepsilon g}.$$
	Given $\varepsilon_{1},\varepsilon_{2}$, it is elementary to see that $\tau_{\varepsilon_{1}}=\tau_{\varepsilon_{2}}$ if and only if $\varepsilon_{1}E^{\times}=\varepsilon_{2}E^{\times}$. Thus we build up a bijection between $\mcs/E^{\times}$ and the set of orthogonal involutions, which is given by $\varepsilon E^{\times}\mapsto\tau_{\varepsilon}$. The following proposition studies the $\mrgl_{m}(E)$-orbits of $\mcs/E^{\times}$, thus classifies all the $\mrgl_{m}(E)$-orbits of orthogonal involutions.
	
	\begin{proposition}\label{proptauGLEorbit}
		
		(1) When $m=1$, there is one $\mrgl_{m}(E)$-orbit in $\mcs/E^{\times}$;
		
		(2) When $m\geq 3$ is odd, there are two $\mrgl_{m}(E)$-orbits in $\mcs/E^{\times}$. A representative in each orbit can be chosen to have any given discriminant, and two representatives with the same discriminant represent different orbits if and only if they have different Hasse invariants;
		
		(3) When $m=2$, there are four $\mrgl_{m}(E)$-orbits in $\mcs/E^{\times}$ determined by the discriminants;
		
		(4) When $m\geq 4$ is even, the discriminant leads to a map from $(\mcs/E^{\times})/\mrgl_{m}(E)$ to $E^{\times}/E^{\times2}$ which is surjective. The fiber corresponding to $(-1)^{m(m-1)/2}$, the discriminant of $J_{m}$, is composed of two orbits distinguished exactly by the Hasse invariant, and the other three fibers are composed of exactly one orbit.
		
	\end{proposition}
	
	\begin{proof}
		
		The proof is a refinement of Proposition \ref{propSGLEorbit}. For more details, see \cite{o2013introduction}, \S 63.
		
	\end{proof}
	
	For $\tau=\tau_{\varepsilon}$ an orthogonal involution, we denote by
	$$\mrgl_{m}(E)^{\tau}:=\{x\in\mrgl_{m}(E)|\tau(x)=x\}$$ the orthogonal group corresponding to $\tau$.
	
	\begin{lemma}\label{lemmatauO}
		
		Let $\tau_{1}$ and $\tau_{2}$ be two orthogonal involutions such that $\mrgl_{m}(E)^{\tau_{1}}=\mrgl_{m}(E)^{\tau_{2}}$, then $\tau_{1}=\tau_{2}$. As a result, $\tau\mapsto \mrgl_{m}(E)^{\tau}$ gives a bijection between $\mrgl_{m}(E)$-orbits of orthogonal involutions and the set of $\mrgl_{m}(E)$-conjugacy classes of orthogonal subgroups of $\mrgl_{m}(E)$.
		
	\end{lemma}
	
	\begin{proof}
		
		For a proof, see \cite{hakim2013tame}, Lemma 2.7.
		
	\end{proof}
	
	Combining Proposition \ref{proptauGLEorbit} and Lemma \ref{lemmatauO}, we get all the possible $\mrgl_{m}(E)$-conjugacy classes of orthogonal groups.
	
	\begin{proposition}\label{propGconjortho}
		
		(1) When $m=1$, there is only one orthogonal group $\{1,-1\}$;
		
		(2) When $m\geq 3$ is odd, there are two $\mrgl_{m}(E)$-conjugacy classes of orthogonal groups, the one corresponding to the symmetric matrix $J_{m}$ is split, and the other one is not quasisplit;
		
		(3) When $m=2$, there are four $\mrgl_{m}(E)$-conjugacy classes of orthogonal groups, the one corresponding to the symmetric matrix $J_{m}$ is split, and the other three are quasisplit but not split;
		
		(4)  When $m\geq 4$ is even, there are five $\mrgl_{m}(E)$-conjugacy classes of orthogonal groups. The one corresponding to the symmetric matrix $J_{m}$ is split, and the one whose corresponding symmetric matrix is in the same fiber as $J_{m}$ but not similar to $J_{m}$, as mentioned in Proposition \ref{proptauGLEorbit}, is not quasisplit, and the other three orthogonal groups are quasisplit but not split.
		
	\end{proposition}
	
	
	\subsection{$\tau$-split embedding}
	
	Now for $E_{m}$ a field extension of degree $m$ over $E$ and $\varepsilon\in\mcs$, we say that an $E$-algebra embedding $\iota:E_{m}\rightarrow\mrm_{m}(E)$ is \emph{$\varepsilon$-symmetric} if its image consists of $\varepsilon$-symmetric matrices, or in other words, $$\varepsilon^{-1}\,^{t}\iota(x)\varepsilon=\iota(x)\quad \text{for any}\ x\in E_{m}.$$
	For $\tau=\tau_{\varepsilon}$ an orthogonal involution, we say that $E_{m}$ is \emph{$\tau$-split} if there exists an embedding $\iota$ as above such that it is \emph{$\varepsilon$-symmetric}, or equivalently for any $x\in E_{m}^{\times}$, we have $\tau(\iota(x))=\iota(x)^{-1}$. In particular, we get $\tau(E_{m}^{\times})=E_{m}^{\times}$. We have the following important proposition which gives all the possible symmetric matrices via a given symmetric embedding:
	
	\begin{proposition}\label{propsymembedepi0}
		
		Let $\tau=\tau_{\varepsilon_{0}}$ be a given orthogonal involution with $\varepsilon_{0}\in\mcs$ and let
		$$\iota_{0}:E_{m}\rightarrow \mrm_{m}(E)$$
		be an $\varepsilon_{0}$-symmetric embedding. Then any symmetric matrix $\varepsilon$ in $\mcs$ such that there exists
		$$\iota:E_{m}\rightarrow \mrm_{m}(E)$$ as an $\varepsilon$-symmetric embedding is similar to an element in $\varepsilon_{0}\iota_{0}(E_{m}^{\times})$. 
		
	\end{proposition}
	
	\begin{proof}
		
		We follow the proof of \cite{hakim2013tame}, Proposition 4.3. For $\varepsilon\in\mcs$ and corresponding $\iota$ satisfying our condition, by the Skolem-Noether theorem, there exists $g\in\mrgl_{m}(E)$ such that
		$$\iota(x)=g^{-1}\iota_{0}(x)g$$
		for any $x\in E_{m}^{\times}$. Then we have
		$$\tau_{0}(\iota_{0}(x))=\iota_{0}(x)^{-1}\quad\text{and}\quad\tau(\iota(x))=\iota(x)^{-1},$$
		thus
		$$\tau(g)^{-1}\varepsilon^{-1}\varepsilon_{0}\iota_{0}(x)^{-1}\varepsilon_{0}^{-1}\varepsilon\tau(g)=\tau(g)^{-1}\tau(\iota_{0}(x))\tau(g)=\iota(x)^{-1}=g^{-1}\iota_{0}(x)^{-1}g,$$
		which means that $$\varepsilon_{0}^{-1}\varepsilon\tau(g)g^{-1}=\varepsilon_{0}^{-1}\,^{t}g^{-1}\varepsilon g^{-1}$$
		commutes with any $\iota_{0}(x)\in\iota_{0}(E_{m}^{\times})$. Thus $\varepsilon_{0}^{-1}\,^{t}g^{-1}\varepsilon g^{-1}\in Z_{\mrm_{m}(E)}(\iota_{0}(E_{m}))\backslash\{0\}=\iota_{0}(E_{m}^{\times})$, which means that $\varepsilon$ is similar to an element in $\varepsilon_{0}\iota_{0}(E_{m}^{\times})$.
		
	\end{proof}
	
	In particular, we call an $E$-algebra embedding $$\iota:E_{m}\rightarrow\mrm_{m}(E),$$
	\emph{$J$-symmetric} if it is $J_{m}$-symmetric, omitting the size of matrices. The following proposition ensures the existence of $J$-symmetric embedding when $E_{m}/E$ is tamely ramified.
	
	\begin{proposition}\label{proptameembed}
		
		When $E_{m}/E$ is tamely ramified, there exists a $J$-symmetric embedding $\iota$.
		
	\end{proposition}
	
	\begin{proof}
		
		See for example \cite{hakim2012distinguished}, Proposition 5.15 or \cite{hakim2013tame}, \S 4.2.
		
	\end{proof}
	
	\begin{remark}
		
		We don't know whether Proposition \ref{proptameembed} is true or not when $E_{m}/E$ is not necessarily tamely ramified.
		
	\end{remark}
	
	\subsection{Calculation of Hilbert symbol and Hasse invariant in certain cases}\label{subsechilberthasse}
	
	In this subsection, we display elementary results for calculating Hilbert symbol and Hasse invariant.
	
	\begin{lemma}[\cite{hakim2012distinguished}, Lemma 5.9]\label{lemmaGLnOFsymhasse}
		
		If $\varepsilon\in\mrgl_{m}(\mfo_{E})\cap\mcs$, then $\mathrm{Hasse}(\varepsilon)=1$.
		
	\end{lemma}
	
	\begin{lemma}\label{lemmahasseblock}
		
		Let $A\in\mrgl_{n_{1}}(E)$ and $B\in\mrgl_{n_{2}}(E)$ be two symmetric matrices, then
		$$\mathrm{Hasse}\bigg(\begin{matrix} A & 0 \\ 0 & B\end{matrix}\bigg)=\mathrm{Hasse}(A)\cdot\mathrm{Hasse}(B)\cdot\mathrm{Hil}(\mathrm{det}(A),\mathrm{det}(B)).$$
		
	\end{lemma}
	
	\begin{proof}
		
		We assume that $A$ is similar to $\mathrm{diag}(a_{1},...,a_{n_{1}})$ and $B$ is similar to $\mathrm{diag}(b_{1},...,b_{n_{2}})$, thus by definition
		\begin{align*}
			\mathrm{Hasse}\bigg(\begin{matrix} A & 0 \\ 0 & B\end{matrix}\bigg)=\mathrm{Hasse}(\mathrm{diag}(a_{1},...,a_{n_{1}},b_{1},...,b_{n_{2}}))&=\mathrm{Hasse}(A)\cdot\mathrm{Hasse}(B)\prod_{i,j=1}^{n_{1},n_{2}}\mathrm{Hil}(a_{i},b_{j})\\
			&=\mathrm{Hasse}(A)\cdot\mathrm{Hasse}(B)\cdot\mathrm{Hil}(\mathrm{det}(A),\mathrm{det}(B)).
		\end{align*}
		
	\end{proof}
	
	\begin{corollary}\label{corhassekmatrix}
		
		Let $A_{i}\in\mrgl_{n_{i}}(E)$ be symmetric matrices for $i=1,...,k$ such that for any $1\leq i<j\leq k$, we have $\mathrm{Hil}(\mathrm{det}(A_{i}),\mathrm{det}(A_{j}))=1$, then
		$$\mathrm{Hasse}(\mathrm{diag}(A_{1},...,A_{k}))=\prod_{i=1}^{k}\mathrm{Hasse}(A_{i}).$$
		
	\end{corollary}
	
	\begin{proof}
		
		We use Lemma \ref{lemmahasseblock} for $k-1$ times to finish the proof.
		
	\end{proof}
	
	\begin{lemma}\label{lemmahilbertresidue}
		
		For $\epsilon_{1}, \epsilon_{2}\in\mfo_{E}^{\times}$ and $\varpi_{E}$ a uniformizer of $E$, we denote by $\bs{l}$ the residue field of $E$, and $\overline{\epsilon_{1}},\overline{\epsilon_{2}}$ the image of $\epsilon_{1},\epsilon_{2}$ in $\bs{l}$ respectively, then:
		
		(1)
		$$\mathrm{Hil}(\varpi_{E}\epsilon_{1},\varpi_{E}\epsilon_{2})=
		\begin{cases}1 \quad & \text{if}\ -\overline{\epsilon_{1}}/\overline{\epsilon_{2}}\in\bs{l}^{\times2},\\
			-1 \quad & \text{otherwise}.
		\end{cases}$$
		
		(2) $$\mathrm{Hil}(\epsilon_{1},\varpi_{E}\epsilon_{2})=
		\begin{cases}1 \quad & \text{if}\ \overline{\epsilon_{1}}\in\bs{l}^{\times2},\\
			-1 \quad & \text{otherwise}.
		\end{cases}$$
		
	\end{lemma}
	
	\begin{proof}
		
		For (1)
		we notice that $$\mathrm{Hil}(\varpi_{E}\epsilon_{1},\varpi_{E}\epsilon_{2})=1$$ if and only if
		$$Z^{2}+\epsilon_{2}/\epsilon_{1}-\varpi_{E}C^{2}/\epsilon_{1}=0\ \text{has a solution for}\ Z\in\mfo_{E}^{\times}\ \text{and} \ C\in\mfo_{E}.$$
		Since if the equation $\varpi_{E}\epsilon_{1}X^{2}+\varpi_{E}\epsilon_{2}Y^{2}=1$ has a solution, comparing the order we must have $X^{-1},Y^{-1}\in \mfp_{E}$ and $X/Y\in\mfo_{E}^{\times}$. Thus we can change the variables $Z=X/Y$ and $C=\varpi_{E}^{-1}Y^{-1}$.
		Using the Hensel lemma for the polynomial $P(Z)=Z^{2}+\epsilon_{2}/\epsilon_{1}-\varpi_{E}C^{2}/\epsilon_{1}$ and the fact that $p\neq 2$, the condition above is true if and only if
		$$\overline{Z}^{2}=-\overline{\epsilon_{2}}/\overline{\epsilon_{1}}\ \text{has a solution for}\ \overline{Z}\in\bs{l}^{\times},$$
		which is equivalent to $-\overline{\epsilon_{1}}/\overline{\epsilon_{2}}\in\bs{l}^{\times2}$. Thus we finish the proof of (1), and the proof of (2) is similar.
		
	\end{proof}
	
	\begin{remark}
		
		In the latter sections, we mainly consider two cases: $E=F$ or $E/F$ is a field extension of degree $d$ given by a certain simple stratum related to a given supercuspidal representation. In the former case, we have $m=n$; In the latter case, we have $m$ such that $n=md$ with $d=[E:F]$. Moreover, we will simply write $\mathrm{det}$, $\mathrm{disc}$ and $\mathrm{Hasse}$ for short when $E=F$.
		
	\end{remark}
	
	From now on until the end of this section, we assume $E$ to be a tamely ramified extension of degree $d=ef$ over $F$, where $f$ denotes its inertia degree and $e$ denotes its ramification index. Using Proposition \ref{proptameembed}, we fix a $J$-symmetric embedding $E\hookrightarrow\mrm_{d}(F)$. We fix $\epsilon_{0}\in\mfo_{E}^{\times}\backslash\mfo_{E}^{\times2}$ and $\varpi_{E}$ a uniformizer of $E$, such that $E^{\times}/E^{\times2}=\{1,\epsilon_{0},\varpi_{E},\epsilon_{0}\varpi_{E}\}$. By Lemma 3.8. of \cite{hakim2013tame}, we have three different cases:
	
	\begin{proposition}\label{propE/Fbasic}
		
		(1) $\mrn_{E/F}(E^{\times})F^{\times2}/F^{\times 2}=\{1\}$ if and only if $E$ contains three quadratic subextensions over $F$. Note that exactly one of them is unramified. Thus this is the case where both $e$ and $f$ are even;
		
		(2) $\mrn_{E/F}(E^{\times})F^{\times2}/F^{\times 2}$ is of order 2 if and only if $E$ contains exactly one quadratic subextension over $F$. Thus either $e$ or $f$ is even (but not both);
		
		(3) $\mrn_{E/F}(E^{\times})F^{\times2}/F^{\times 2}=F^{\times}/F^{\times2}$ if and only if $E$ contains no quadratic subextension over $F$. Thus $d=ef$ is odd.
		
	\end{proposition}
	
	For case (1), we have the following lemma:
	
	\begin{lemma}\label{lemmahassecase3a}
		
		If $\mathrm{N}_{E/F}(E^{\times})F^{\times2}/F^{\times2}=\{1\}$, then we may further choose the uniformizer $\varpi_{E}$ of $E$, such that
		$$\mathrm{Hasse}(J_{d}\varpi_{E})=1\quad \text{and} \quad\mathrm{Hasse}(J_{d}\varpi_{E}\epsilon_{0})=-1,$$
		where $J_{d}\varpi_{E}$ and $J_{d}\varpi_{E}\epsilon_{0}$ are symmetric matrices in $\mrgl_{d}(F)$.
		
	\end{lemma}
	
	\begin{proof}
		
		We may use \cite{hakim2013tame}, Proposition 6.6 directly.
		
	\end{proof}
	
	For case (2), first we assume that $f$ is odd and $e$ is even. We have:
	
	\begin{lemma}\label{lemmahassecase2b2}
		
		For $f$ odd and $e$ even, we have $\mathrm{Hasse}(J_{d}\varpi_{E})\neq\mathrm{Hasse}(J_{d}\varpi_{E}\epsilon_{0})$.
		
	\end{lemma}
	
	\begin{proof}
		
		We use the proof of \cite{hakim2013tame}, Proposition 6.6 directly, except that right now $f$ is odd instead of being even. Our question reduces to calculate the following term
		$$\mathrm{Hasse}(\mathrm{diag}(u_{1},...,u_{f},u_{1}\varpi_{F},...,u_{f}\varpi_{F}))\quad (\text{with}\ u_{1},...,u_{f}\in\mfo_{F}^{\times})$$
		in the case where $\prod_{i=1}^{f}u_{i}\in F^{\times2}$ or $\epsilon_{0}'F^{\times2}$ respectively with $\epsilon_{0}'\in\mfo_{F}^{\times}\backslash\mfo_{F}^{\times2}$ fixed, and to show that they are different. From the calculation in \emph{loc. cit.}, we have
		\begin{align*}
			\mathrm{Hasse}(\mathrm{diag}(u_{1},...,u_{f},u_{1}\varpi_{F},...,u_{f}\varpi_{F}))&=(\prod_{i=1}^{f}\mathrm{Hil}(u_{i},\varpi_{F}))^{2f-1}\cdot\mathrm{Hil}(\varpi_{F},\varpi_{F})^{f(f-1)/2}\\
			&=\mathrm{Hil}(\prod_{i=1}^{f}u_{i},\varpi_{F})\cdot\mathrm{Hil}(\varpi_{F},\varpi_{F})^{f(f-1)/2}
		\end{align*}
		Thus by Lemma \ref{lemmahilbertresidue}.(2), when $\prod_{i=1}^{f}u_{i}\in F^{\times2}$ or $\epsilon_{0}'F^{\times2}$ respectively, the corresponding terms are different.
		
	\end{proof}

	\begin{corollary}\label{corhassecase3b}
		
		Under the assumption of Lemma \ref{lemmahassecase2b2}, the Hasse invariants
		$$\mathrm{Hasse}(\mathrm{diag}(J_{d}\varpi_{E},...,J_{d}\varpi_{E},J_{d}\varpi_{E}))\quad\text{and}\quad\mathrm{Hasse}(\mathrm{diag}(J_{d}\varpi_{E},...,J_{d}\varpi_{E},J_{d}\varpi_{E}\epsilon_{0}))$$
		are different, where the two matrices are in $\mrm_{m}(\mrm_{d}(F))=\mrm_{md}(F)$.
		
	\end{corollary}
	
	\begin{proof}
		
		We write $$A=\mathrm{diag}(J_{d}\varpi_{E},...,J_{d}\varpi_{E})\in\mrm_{m-1}(\mrm_{d}(F))=\mrm_{(m-1)d}(F),$$ 
		then using Lemma \ref{lemmahasseblock}, we have
		$$\mathrm{Hasse}(\mathrm{diag}(J_{d}\varpi_{E},...,J_{d}\varpi_{E},J_{d}\varpi_{E}))=\mathrm{Hasse}(A)\cdot\mathrm{Hasse}(J_{d}\varpi_{E})\cdot\mathrm{Hil}(\mathrm{det}(A),\mathrm{det}(J_{d}\varpi_{E}))$$
		and
		$$\mathrm{Hasse}(\mathrm{diag}(J_{d}\varpi_{E},...,J_{d}\varpi_{E},J_{d}\varpi_{E}\epsilon_{0}))=\mathrm{Hasse}(A)\cdot\mathrm{Hasse}(J_{d}\varpi_{E}\epsilon_{0})\cdot\mathrm{Hil}(\mathrm{det}(A),\mathrm{det}(J_{d}\varpi_{E}\epsilon_{0})).$$
		Thus using Lemma \ref{lemmahassecase2b2}, we only need to show that
		$$\mathrm{Hil}(\mathrm{det}(A),\mathrm{det}(J_{d}\varpi_{E}))=\mathrm{Hil}(\mathrm{det}(A),\mathrm{det}(J_{d}\varpi_{E}\epsilon_{0})),$$
		which follows from the fact that $\mathrm{det}(\epsilon_{0})=\mrn_{E/F}(\epsilon_{0})\in F^{\times 2}$ when $e$ is even.
		
	\end{proof}
	
	Now we assume that $e$ is odd. First we consider the case where $f$ is even. In this case, $\mrn_{E/F}(\epsilon_{0})\notin F^{\times2}$. We choose $\varpi_{E}'$ to be another uniformizer of $E$ such that $\mrn_{E/F}(\varpi_{E}')\in F^{\times2}$. 
	
	\begin{lemma}\label{lemmahassecase2b1}
		
		If $e$ and $m$ are odd and if $f$ is even, then
		$$\mathrm{Hasse}(\mathrm{diag}(J_{d}\varpi_{E},...J_{d}\varpi_{E},J_{d}\varpi_{E}'))=1$$
		and
		$$\mathrm{Hasse}(\mathrm{diag}(J_{d}\varpi_{E}\epsilon_{0},...J_{d}\varpi_{E}\epsilon_{0},J_{d}\varpi_{E}'\epsilon_{0}))=-1,$$
		where the two matrices are in $\mrm_{m}(\mrm_{d}(F))=\mrm_{md}(F)$.
		
	\end{lemma}
	
	\begin{proof}
		
		
		To begin with, we state and proof the following general lemma which is useful not only in this proof, but in the latter sections.
		
		\begin{lemma}\label{lemmaE/Lextension}
			
			Let $E/L$ be a finite extension of non-archimedean locally compact fields of residue characteristic $p\neq 2$ of odd degree, and let $$L^{\times}/L^{\times2}\rightarrow E^{\times}/E^{\times2}$$ be the homomorphism induced by the canonical embedding $L\hookrightarrow E$, then the homomorphism above induces two isomorphisms $$L^{\times}/L^{\times2}\simeq E^{\times}/E^{\times2}\quad\text{and}\quad\mfo_{L}^{\times}/\mfo_{L}^{\times2}\simeq\mfo_{E}^{\times}/\mfo_{E}^{\times2}.$$
			
		\end{lemma}
		
		\begin{proof}
			
			The embedding $L\hookrightarrow E$ leads to the following embedding:
			$$L^{\times}/ E^{\times2}\cap L^{\times}\hookrightarrow E^{\times}/E^{\times2}.$$
			First we have $L^{\times2}\subset E^{\times2}\cap L^{\times}$. And for $x\in E^{\times2}\cap L^{\times}$, let $x=y^{2}$ with $y\in E^{\times}$. Thus $L[y]$ is a subextension of $E$ over $L$ which is of degree $1$ or $2$. Since $[E:L]$ is odd, we must have $L[y]=L$ and $y\in L$. So $x\in L^{\times2}$, which means that $E^{\times2}\cap L^{\times}=L^{\times2}$ since $x$ is arbitrary. Thus the homomorphism in the lemma is injective, which is an isomorphism since $[E^{\times}:E^{\times2}]=[L^{\times}:L^{\times2}]=4$.
			
			Moreover, since $|\mfo_{L}^{\times}/\mfo_{L}^{\times2}|=|\mfo_{E}^{\times}/\mfo_{E}^{\times2}|=2$, the isomorphism above also leads to an isomorphism
			$$\mfo_{L}^{\times}/\mfo_{L}^{\times2}\simeq\mfo_{E}^{\times}/\mfo_{E}^{\times2}.$$
			
		\end{proof}
		
		Come back to the original proof. We write $L$ for the maximal unramified subextension of $E$ over $F$, then $[L:F]=f$ and $[E:L]=e$. Since $e$ is odd, by Lemma \ref{lemmaE/Lextension} we have an isomorphism $$\mfo_{E}^{\times}/\mfo_{E}^{\times2}\simeq\mfo_{L}^{\times}/\mfo_{L}^{\times2}.$$ Since the result does not depend on the choice of $\varpi_{E}$, $\varpi_{E}'$ and $\epsilon_{0}$ as representatives in $E^{\times}/E^{\times2}$, we may assume that $\varpi_{E}^{e}=\varpi_{L}$ is a uniformizer in $L$, $\varpi_{E}'^{e}=\varpi_{L}'$ is a uniformizer in $L$ such that $\mrn_{L/F}(\varpi_{L}')\in F^{\times2}$, and $\epsilon_{0}\in\mfo_{L}^{\times}\backslash\mfo_{L}^{\times2}$. From the construction of the $J$-symmetric embedding in Proposition \ref{proptameembed} (see the proof of \cite{hakim2013tame}, Proposition 6.6 for more details), we may write
		$$J_{d}\varpi_{E}=\mathrm{diag}(J_{(e-1)f},J_{f}\varpi_{L})\quad \text{and}\quad J_{d}\varpi_{E}'=\mathrm{diag}(J_{e(f-1)}, J_{f}\varpi_{L}')$$
		and
		$$J_{d}\varpi_{E}\epsilon_{0}=\mathrm{diag}(J_{(e-1)f}\epsilon_{0}, J_{f}\varpi_{L}\epsilon_{0})\quad \text{and}\quad J_{d}\varpi_{E}'\epsilon_{0}=\mathrm{diag}(J_{e(f-1)}\epsilon_{0}, J_{f}\varpi_{L}'\epsilon_{0}).$$
		Since $\mathrm{det}(J_{(e-1)f})\in\mfo_{F}^{\times}$, and since
		$\mathrm{det}(\mathrm{diag}(J_{f}\varpi_{L},...,J_{f}\varpi_{L},J_{f}\varpi_{L}'))$ is of even order in $F^{\times}$,
		using Lemma \ref{lemmaGLnOFsymhasse} and Corollary \ref{corhassekmatrix}, we get
		\begin{equation}\label{eq2b1hasse1}
			\mathrm{Hasse}(\mathrm{diag}(J_{d}\varpi_{E},...,J_{d}\varpi_{E},J_{d}\varpi_{E}'))=\mathrm{Hasse}(\mathrm{diag}(J_{f}\varpi_{L},...,J_{f}\varpi_{L},J_{f}\varpi_{L}')),
		\end{equation}
		where the matrix in the Hasse of the right hand side is of size $fm$.
		Similarly we have
		\begin{equation}\label{eq2b1hasse2}
			\mathrm{Hasse}(\mathrm{diag}(J_{d}\varpi_{E}\epsilon_{0},...,J_{d}\varpi_{E}\epsilon_{0},J_{d}\varpi_{E}'\epsilon_{0}))=\mathrm{Hasse}(\mathrm{diag}(J_{f}\varpi_{L}\epsilon_{0},...,J_{f}\varpi_{L}\epsilon_{0},J_{f}\varpi_{L}'\epsilon_{0})),
		\end{equation}
		where the matrix in the Hasse of the right hand side is also of size $fm$. Since $L/F$ is unramified, we may write $\varpi_{L}=\varpi_{F}v$  and $\varpi_{L}'=\varpi_{F}v'$ with $v,v'\in\mfo_{L}^{\times}$, thus the term in (\ref{eq2b1hasse1}) equals
		\begin{equation}\label{eq2b1hasse1'}
			\mathrm{Hasse}(\mathrm{diag}(J_{f}v\varpi_{F},...,J_{f}v\varpi_{F},J_{f}v'\varpi_{F})),
		\end{equation}
		and the term in (\ref{eq2b1hasse2}) equals
		\begin{equation}\label{eq2b1hasse2'}
			\mathrm{Hasse}(\mathrm{diag}(J_{f}v\epsilon_{0}\varpi_{F},...,J_{f}v\epsilon_{0}\varpi_{F},J_{f}v'\epsilon_{0}\varpi_{F})).
		\end{equation}
		Since $f$ is even, $\mathrm{det}(J_{f}\varpi_{F})$ and $\mathrm{det}(J_{f}v'\varpi_{F})$ are of even order in $F^{\times}$, thus by Lemma \ref{corhassekmatrix}, (\ref{eq2b1hasse1'}) equals
		$$\mathrm{Hasse}(J_{f}v\varpi_{F})^{m-1}\cdot\mathrm{Hasse}(J_{f}v'\varpi_{F})=\mathrm{Hasse}(J_{f}v'\varpi_{F})$$
		and similarly
		(\ref{eq2b1hasse2'}) equals
		$$\mathrm{Hasse}(J_{f}v\epsilon_{0}\varpi_{F})^{m-1}\cdot\mathrm{Hasse}(J_{f}v'\epsilon_{0}\varpi_{F})=\mathrm{Hasse}(J_{f}v'\epsilon_{0}\varpi_{F}).$$
		We assume that $J_{f}v'$ is similar to $\mathrm{diag}(1,...,1,u_{1})$
		and $J_{f}v'\epsilon_{0}$ is similar to $\mathrm{diag}(1,...,1,u_{2})$ with $u_{1}, u_{2}\in\mfo_{F}^{\times}$, then (\ref{eq2b1hasse1'}) equals
		$$\mathrm{Hasse}(\mathrm{diag}(\varpi_{F},...,\varpi_{F},\varpi_{F}u_{1})),$$
		and (\ref{eq2b1hasse2'}) equals
		$$\mathrm{Hasse}(\mathrm{diag}(\varpi_{F},...,\varpi_{F},\varpi_{F}u_{2})).$$
		By direct calculation, we get $$\mathrm{det}(J_{f}v'\varpi_{F})=(-1)^{f(f-1)/2}\mrn_{L/F}(\varpi_{L}')\in (-1)^{f(f-1)/2}F^{\times2}$$
		and
		$$\mathrm{det}(J_{f}v'\epsilon_{0}\varpi_{F})=(-1)^{f(f-1)/2}\mrn_{L/F}(\epsilon_{0}\varpi_{L}')\in (-1)^{f(f-1)/2}\mrn_{L/F}(\epsilon_{0})F^{\times2},$$
		where $\mrn_{L/F}(\epsilon_{0})\in\mfo_{F}^{\times}\backslash\mfo_{F}^{\times2}$.
		
		If $-1\in F^{\times2}$ or if $-1\notin F^{\times2}$ and $4|f$, then $\mathrm{det}(J_{f}v')\in \mfo_{F}^{\times2}$ and $\mathrm{det}(J_{f}v'\epsilon_{0})\in \mfo_{F}^{\times}\backslash\mfo_{F}^{\times2}$. We may assume $u_{1}=1$ and $u_{2}\in\mfo_{F}^{\times}\backslash\mfo_{F}^{\times2}$, where in the latter case we may further assume $u_{2}=-1$. So by Lemma \ref{lemmahilbertresidue}.(1), when $-1\in F^{\times2}$ we have $$\mathrm{Hasse}(\mathrm{diag}(\varpi_{F},...,\varpi_{F},\varpi_{F}u_{1}))=1$$
		and
		$$\mathrm{Hasse}(\mathrm{diag}(\varpi_{F},...,\varpi_{F},\varpi_{F}u_{2}))=(-1)^{f-1}=-1.$$ When $-1\notin F^{\times2}$ and $4|f$, we have
		$$\mathrm{Hasse}(\mathrm{diag}(\varpi_{F},...,\varpi_{F},\varpi_{F}u_{1}))=(-1)^{f(f-1)/2}=1,$$
		and
		$$\mathrm{Hasse}(\mathrm{diag}(\varpi_{F},...,\varpi_{F},-\varpi_{F}))=(-1)^{(f-1)(f-2)/2}=-1.$$
		
		If $-1\notin F^{\times2}$ and $4\nmid f$, then $\mathrm{det}(J_{f}v')\in\mfo_{F}^{\times}\backslash\mfo_{F}^{\times2}$ and $\mathrm{det}(J_{f}v'\epsilon_{0})\in\mfo_{F}^{\times2}$. We may assume $u_{1}=-1$ and $u_{2}=1$ and we have
		$$\mathrm{Hasse}(\mathrm{diag}(\varpi_{F},...,\varpi_{F},-\varpi_{F}))=(-1)^{(f-1)(f-2)/2}=1$$ and
		$$\mathrm{Hasse}(\mathrm{diag}(\varpi_{F},...,\varpi_{F},\varpi_{F}))=(-1)^{f(f-1)/2}=-1.$$ Thus we finish the proof.
		
	\end{proof}
	
	Finally, we drop the assumption that $f$ is even.
	
	\begin{lemma}\label{lemmahassecase2a1}
		
		If $e$ is odd, $m$ is even and one of the three cases happens:
		\begin{itemize}
			\item $2|d$;
			\item $2\nmid d$ and $4|m$;
			\item $2\nmid d$, $4\nmid m$ and $-1\in F^{\times2}$,
		\end{itemize} then
		$\mathrm{Hasse}(\mathrm{diag}(J_{d}\varpi_{E},...,J_{d}\varpi_{E},J_{d}\varpi_{E}\epsilon_{0}))=-1,$ where the matrix in Hasse is in $\mrm_{md}(F)$.
		
	\end{lemma}
	
	\begin{proof}
		
		
		We write $L$ for the maximal unramified extension of $F$ contained in $E$, thus $[L:F]=f$ and $[E:L]=e$. Since $e$ is odd, by Lemma \ref{lemmaE/Lextension} we get $$\mfo_{E}^{\times}/\mfo_{E}^{\times2}\simeq\mfo_{L}^{\times}/\mfo_{L}^{\times2}.$$ Since the result does not depend on the choice of $\varpi_{E}$ and $\epsilon_{0}$ as representatives in $E^{\times}/E^{\times2}$, we may choose $\varpi_{E}$ as a uniformizer of $E$ such that $\varpi_{E}^{e}=\varpi_{L}$ is a uniformizer in $L$, and $\epsilon_{0}\in\mfo_{L}^{\times}\backslash\mfo_{L}^{\times2}$. As in Lemma \ref{lemmahassecase2b1}, we may write
		$$J_{d}\varpi_{E}=\mathrm{diag}(J_{(e-1)f},J_{f}\varpi_{L})\quad \text{and}\quad J_{d}\varpi_{E}\epsilon_{0}=\mathrm{diag}(J_{e(f-1)}\epsilon_{0}, J_{f}\varpi_{L}\epsilon_{0}).$$
		Thus by Corollary \ref{corhassekmatrix} and the fact that $m$ is even, we get
		\begin{equation}\label{eq2a1hasse}
			\mathrm{Hasse}(\mathrm{diag}(J_{d}\varpi_{E},...,J_{d}\varpi_{E},J_{d}\varpi_{E}\epsilon_{0}))=\mathrm{Hasse}(\mathrm{diag}(J_{f}\varpi_{L},...,J_{f}\varpi_{L},J_{f}\varpi_{L}\epsilon_{0})),
		\end{equation}
		where the last term in the Hasse is a matrix of size $fm$. Since $L/F$ is unramified, we may write $\varpi_{L}=\varpi_{F}v$ with $v\in\mfo_{L}^{\times}$, thus the term in (\ref{eq2a1hasse}) equals $$\mathrm{Hasse}(\mathrm{diag}(J_{f}v\varpi_{F},...,J_{f}v\varpi_{F},J_{f}v\varpi_{F}\epsilon_{0})).$$
		If we assume that $J_{f}v$ is similar to $\mathrm{diag}(1,...,1,u_{1})$, and $J_{f}v\epsilon_{0}$ is similar to $\mathrm{diag}(1,...,1,u_{2})$, then we get $u_{2}/u_{1}\in\mfo_{F}^{\times}\backslash\mfo_{F}^{\times2}$. Moreover we get
		\begin{equation}\label{eqJfpif}
			\mathrm{Hasse}(\mathrm{diag}(J_{f}v\varpi_{F},...,J_{f}v\varpi_{F},J_{f}v\varpi_{F}\epsilon_{0}))=\mathrm{Hasse}(\mathrm{diag}(I_{m(f-1)}\varpi_{F},u_{1}\varpi_{F},...,u_{1}\varpi_{F},u_{2}\varpi_{F})),
		\end{equation}
		where the last diagonal matrix in Hasse is of size $fm$.
		
		If $-1\in F^{\times2}$, we may choose either $u_{1}=1$ and $u_{2}=\epsilon_{0}'$, or $u_{1}=\epsilon_{0}'$ and $u_{2}=1$ with $\epsilon_{0}'\in\mfo_{F}^{\times}\backslash\mfo_{F}^{\times2}$. Thus in the former case, by Lemma \ref{lemmahilbertresidue}.(1) the (\ref{eqJfpif}) equals $$\mathrm{Hasse}(\mathrm{diag}(I_{mf-1}\varpi_{F},\varpi_{F}\epsilon_{0}'))=(-1)^{mf-1}=-1,$$
		and in the latter case, by Lemma \ref{lemmahilbertresidue}.(1) the (\ref{eqJfpif}) equals
		$$\mathrm{Hasse}(\mathrm{diag}(I_{mf-m+1}\varpi_{F}, I_{m-1}\varpi_{F}\epsilon_{0}'))=(-1)^{(mf-m+1)(m-1)}=-1.$$
		
		If $-1\notin F^{\times2}$, we may assume $\epsilon_{0}'=-1$, $u_{1}$ equals $1$ or $-1$ and $u_{2}=-u_{1}$, and for the two cases using Lemma \ref{lemmahilbertresidue}.(1) the (\ref{eqJfpif}) equals
		$$\mathrm{Hasse}(\mathrm{diag}(I_{fm-1}\varpi_{F},-\varpi_{F}))=(-1)^{(fm-1)(fm-2)/2}=-1$$
		or
		$$\mathrm{Hasse}(\mathrm{diag}(I_{(f-1)m+1}\varpi_{F}, -I_{m-1}\varpi_{F}))=(-1)^{fm(fm-1)/2-((f-1)m+1)(m-1)}=-1,$$
		where in both cases we use the fact that $4|fm$ and $2|m$, thus we finish the proof.
		
	\end{proof}
	
	Finally we have the following lemma which completes Lemma \ref{lemmahassecase2a1}.
	
	\begin{lemma}\label{lemmahassecase2a2}
		
		If $d$ is odd, $m$ is even not divided by 4 and $-1\notin F^{\times2}$, then
		$$\mathrm{Hasse}(\mathrm{diag}(J_{d}\varpi_{E},...,J_{d}\varpi_{E},J_{d}\varpi_{E}))=-1,$$
		where the matrix is in $\mrm_{md}(F)$.
		
	\end{lemma}
	
	\begin{proof}
		
		We may follow the same proof as Lemma \ref{lemmahassecase2a1}, which finally shows that
		$$\mathrm{Hasse}(\mathrm{diag}(J_{d}\varpi_{E},...,J_{d}\varpi_{E},J_{d}\varpi_{E}))=\mathrm{Hasse}(I_{fm}\varpi_{F}).$$
		Since $-1\notin F^{\times2}$, by Lemma \ref{lemmahilbertresidue}.(1) the latter term equals $(-1)^{fm(fm-1)/2}$, which is $-1$ since under our assumption $fm\equiv 2$ (mod $4$). Thus we finish the proof.
		
	\end{proof}

\begin{remark}

The above discussion does not aim at providing a full classification result. It rather provides technical backgrounds to describe those orthogonal involutions $\tau$ contributing to the distinction. For example, we didn't discuss the case where both $d$ and $m$ are odd since it is not necessary. Indeed in this case, $n=dm$ is odd and there exist one split orthogonal subgroup and one non-quasisplit orthogonal subgroup of $\mrgl_{n}(F)$. The former one is easy to describe which will contribute to the distinction, while the latter one can be eliminated using the argument in \S \ref{subsectionotherortho}.

\end{remark}

\begin{remark}

We emphasize that for those results from Proposition \ref{propE/Fbasic} to Lemma \ref{lemmahassecase2a2}, the condition $E/F$ being tamely ramified is important, and the author does not know if they remain valid or not if we erase this condition. However in Corollary \ref{corcalhasse} we indeed verified these results for certain $E/F$ that are not necessarily tamely ramified.

\end{remark}

	\section{$\tau$-selfdual type theorem}\label{sectiontauselfdualtype}
	
	Let $\pi$ be a supercuspidal representation of $G$.
	Let $\tau=\tau_{\varepsilon}$ be the orthogonal involution corresponding to a symmetric matrix $\varepsilon$, such that for $H=G^{\tau}$ as the orthogonal group corresponding to $\tau$, it satisfies the condition 2 of Theorem \ref{thmdistcri} with respect to $\pi$. For $\mfa$ an $\mfo_{F}$-subalgebra of $\mrm_{n}(F)$, we define
	$$\tau(\mfa):=\varepsilon^{-1}\,^{t}\mfa\varepsilon$$
	which is an $\mfo_{F}$-subalgebra of $\mrm_{n}(F)$. We say that $\mfa$ is \emph{$\tau$-stable} if $\tau(\mfa)=\mfa$. For any $g\in G$, it is easy to show that $\tau(\mfa^{g})=\tau(\mfa)^{\tau(g)}$.
	
	In this section, we follow the strategy in \cite{zou2019supercuspidal}, section 5 and \cite{anandavardhanan2019galois}, section 4 to prove the following theorem:
	
	\begin{theorem}\label{thmtauselfdualstra}
		
		For $\pi$ and $\tau$ as above, there exists a maximal simple stratum $[\mfa,\beta]$ and a simple character $\theta\in\mcc(\mfa,\beta)$ contained in $\pi$, such that
		
		(1) $\tau(\mfa)=\mfa$ and $\tau(H^{1}(\mfa,\beta))=H^{1}(\mfa,\beta)$;
		
		(2) $\theta\circ\tau=\theta^{-1}$;
		
		(3) $\tau(\beta)=\beta^{-1}$.
		
	\end{theorem}
	
	As a corollary of Theorem \ref{thmtauselfdualstra}, we have the following $\tau$-selfdual type theorem.
	
	\begin{theorem}\label{thmtauselfdualtype}
		
		For $\pi$ and $\tau$ as above, there exists a $\tau$-selfdual simple type $(\bs{J},\Lambda)$ that compactly induces $\pi$.
		
	\end{theorem}
	
	\begin{proof}
		
		We only need to follow the proof of \cite{zou2019supercuspidal}, Theorem 5.3, with Theorem 5.2 in \emph{ibid.} replaced by Theorem \ref{thmtauselfdualstra}.
		
	\end{proof}
	
	Now we state the following general theorem which implies Theorem \ref{thmtauselfdualstra}.
	
	\begin{theorem}\label{thmendotauselfdualstra}
		
		Let $[\mfa,\beta]$ be a maximal simple stratum in $\mrm_{n}(F)$, let $T$ be the maximal tamely ramified subextension of $E/F$, let $T_{m}$ be the unramified extension of degree $m$ over $T$ and let $\theta\in\mcc(\mfa,\beta)$ be a simple character. Let $\tau$ be an orthogonal involution of $G$ such that $H=G^{\tau}$ satisfies the condition 2 of Theorem \ref{thmdistcri}. Then there exist a maximal simple stratum $[\mfa',\beta']$ in $\mrm_{n}(F)$ and a simple character $\theta'\in\mcc(\mfa',\beta')$ such that
		
		(1) $\tau(\mfa')=\mfa'$ and $\tau(H^{1}(\mfa',\beta'))=H^{1}(\mfa',\beta')$;
		
		(2) $\theta'$ and $\theta$ are in the same endo-class and $\theta'\circ\tau=\theta'^{-1}$;
		
		(3) $\tau(\beta')=\beta'^{-1}$.
		
	\end{theorem}
	
	For $\pi$ given as in Theorem \ref{thmtauselfdualstra}, if we choose $[\mfa,\beta]$ to be a maximal simple stratum and $\theta\in\mcc(\mfa,\beta)$ to be a simple character contained in $\pi$, then Theorem \ref{thmendotauselfdualstra} implies Theorem \ref{thmtauselfdualstra}. So from now on, we focus on the proof of Theorem \ref{thmendotauselfdualstra}. We write $E=F[\beta]$, $d=[E:F]$ and $m=n/d$. 
	In the following subsections, we gradually consider the following three cases: $E/F$ is maximal and totally wildly ramified, $E/F$ is maximal and the general case.
	
	
	
	
	To begin with, we state the following lemmas which will be useful in our future proof.
	
	\begin{lemma}\label{lemmaendotheta}
		
		Let $[\mfa,\beta]$ be a maximal simple stratum in $\mrm_{n}(F)$ and let $\theta\in\mcc(\mfa,\beta)$, then for $\tau$ an orthogonal involution on $G$, the simple characters $\theta\circ\tau$ and $\theta^{-1}$ are in the same endo-class. In particular, if $\tau(\mfa)=\mfa$, then $\theta\circ\tau$ conjugates to $\theta^{-1}$ by an element in $U(\mfa)$.
		
	\end{lemma}
	
	\begin{proof}
		
		We follow the same proof of \cite{zou2019supercuspidal}, Lemma 5.7, with $\sigma$ in $\emph{loc. cit.}$ replaced by the trivial action.
		
	\end{proof}
	
	\begin{lemma}\label{lemmatauGaction}
		
		Let $\tau=\tau_{\varepsilon}$ be the orthogonal involution on $G$ corresponding to a symmetric matrix $\varepsilon$, let $[\mathfrak{a},\beta]$ be a maximal simple stratum in $\mathrm{M}_{n}(F)$ and let $\theta\in\mathcal{C}(\mathfrak{a},\beta)$ be a simple character, such that $$\tau(\mathfrak{a})=\mathfrak{a},\quad \theta\circ\tau=\theta^{-1}\quad (\text{and}\ \tau(\beta)=\beta^{-1}).$$ Then for $\tau'=\tau_{\varepsilon'}$ the orthogonal involution on $G$ corresponding to the symmetric matrix $\varepsilon'=\,^{t}g\varepsilon g$, we have $$\tau'(\mathfrak{a}^{g})=\mathfrak{a}^{g},\quad \theta^{g}\circ\tau'=(\theta^{g})^{-1}\quad (\text{and}\ \tau'(\beta^{g})=(\beta^{g})^{-1}).$$
		
	\end{lemma}
	
	\begin{proof}
		
		Same proof as \cite{zou2019supercuspidal}, Lemma 5.8.
		
	\end{proof}
	
	\begin{lemma}\label{lemmastevensbeta}
		
		Let $[\mfa,\beta]$ be a maximal simple stratum in $\mrm_{n}(F)$
		and let $\theta\in\mcc(\mfa,\beta)$ such that $\tau(\mfa)=\mfa$, $\tau(H^{1}(\mfa,\beta))=H^{1}(\mfa,\beta)$ and $\theta\circ\tau=\theta^{-1}$. Then there exists a simple stratum $[\mfa,\gamma]$ such that $\theta\in\mcc(\mfa,\gamma)$ and $\tau(\gamma)=\gamma^{-1}$.
		
	\end{lemma}
	
\begin{proof}
		
For $\tau=\tau_{\varepsilon}$ with respect to a symmetric matrix $\varepsilon$, we define
$$\sigma_{\varepsilon}(x):=\varepsilon^{-1}\,^{t}x\varepsilon\quad \text{for any}\ x\in\mrm_{n}(F)$$ as an anti-involution on $\mrm_{n}(F)$. First we have the following proposition as a variant of \cite{stevens2001intertwining}, Theorem 6.3.
		
\begin{proposition}
		
Let $\sigma$ be an anti-involution\footnote{It means that $\sigma(x+y)=\sigma(x)+\sigma(y)$ and $\sigma(xy)=\sigma(y)\sigma(x)$ for any $x,y\in\mrm_{n}(F)$.} on $\mrm_{n}(F)$, let $[\mfa,\beta]$ be a simple stratum in $\mrm_{n}(F)$  such that $\sigma(\mfa)=\mfa$ and let $\theta\in\mcc(\mfa,\beta)$ such that $\theta\circ\sigma=\theta$. Then there exists a $\sigma$-invariant simple stratum $[\mfa,\gamma]$ such that $\theta\in\mcc(\mfa,\gamma)$, where $\sigma$-invariance means that $\sigma(\mfa)=\mfa$ and $\sigma(\gamma)=\gamma$.
		
\end{proposition}

Since our lemma follows directly from this proposition once we set $\sigma=\sigma_{\varepsilon}$, it remains to explain how to adapt the proof of Stevens to our situation. In \emph{loc. cit.},
we replace $\sigma$ there by our $\sigma$ here, we replace $\overline{x}$ there by $-\sigma(x)$ for any $x\in\mrm_{n}(F)$, we replace the condition of a lattice chain $\Lambda$ being ``self-dual" there by the condition of the corresponding simple stratum $\mfa=\mfa(\Lambda)$ being $\sigma$-invariant, and we replace the notation ``skew simple stratum" by the notation ``$\sigma$-invariant simple stratum". In particular since \cite{stevens2001intertwining}, Proposition 1.10 is a key step in the proof, we also list the following corresponding statement to avoid ambiguity.

\begin{lemma}
	
Let $[\mfa,r,s,\beta]$ be a pure stratum with $\mfa$ being $\sigma$-invariant and $\beta-\sigma(\beta)\in\mfp_{\mfa}^{-r}$, then there exists a $\sigma$-invariant simple stratum $[\mfa,r,s,\gamma]$ equivalent to $[\mfa,r,s,\beta]$.
	
\end{lemma}

With these replacements, the original proof can be used directly. See also the last paragraph of \cite{stevens2001intertwining} for a remark.
	
\end{proof}
	
	\subsection{The maximal and totally wildly ramified case}\label{subsecmaxwildram}
	
	In this subsection, we prove the following special case of Theorem \ref{thmendotauselfdualstra}.
	
	\begin{proposition}\label{proptauselfdualstramaxwild}
		
		Let $[\mfa,\beta]$ be a simple stratum in $\mrm_{n}(F)$ and let $\theta\in\mcc(\mfa,\beta)$ be a simple character, where $n=d$ and $E/F$ is totally wildly ramified. Then for $\tau=\tau_{I_{n}}$ the orthogonal involution on $G$, there exist a simple stratum $[\mfa',\beta']$ and a simple character $\theta'\in\mcc(\mfa',\beta')$ such that $(\mfa',\theta')$ is $G$-conjugate to $(\mfa,\theta)$ with the property $\tau(\mfa')=\mfa'$ and $\theta'\circ\tau=\theta'^{-1}$. Moreover, we may further assume that $\mfa'\subset\mrm_{n}(\mfo_{F})$.
		
	\end{proposition}
	
	\begin{proof}
		
		We explain how the proof of \cite{zou2019supercuspidal}, Proposition 5.9  (also see \cite{anandavardhanan2019galois}, Proposition 4.13.) could be used directly in our case. First up to $G$-conjugacy, we may assume $\mfa$ to be the standard minimal order of $\mrm_{n}(F)$. We have the following lemma corresponding to Lemma 5.11 in \emph{ibid.}:
		
		\begin{lemma}\label{lemmamaxtotalwild}
			
			There exist $g_{1}\in \mrgl_{n}(\mfo_{F})$ and $a_{1},...,a_{n}\in\mathfrak{o}_{F}^{\times}$ such that $$\tau(g_{1})g_{1}^{-1}=A:=\begin{pmatrix}0 & 0 & \ldots & 0 & a_{1}\\ 0 & \iddots & \iddots & a_{2} & 0 \\ \vdots & \iddots & \iddots & \iddots & \vdots \\ 0 & a_{n-1} & \iddots & \iddots & 0 \\ a_{n} & 0 & \ldots & 0 & 0 \end{pmatrix}.$$
			Moreover, if we define $\mathfrak{a}'':=\mathfrak{a}^{g_{1}}$, then we have $\tau(\mathfrak{a}'')=\mathfrak{a}''$.
			
		\end{lemma}
		
		\begin{proof}
			
			We choose $a_{1}=...=a_{(n-1)/2}=a_{(n+3)/2}=...=a_{n}=1$, and $a_{(n+1)/2}$ equals $1$ or $-1$ to make sure that $\mathrm{det}(A)=1$. Since the $\mfo_{F}$-lattice of rank $n$ equipped with a quadratic form corresponding to $A$ is unimodular in the sense of \cite{o2013introduction}, \S 92, by \S 92:1 in \emph{loc. cit.}, there exists $g_{1}\in\mrgl_{n}(\mfo_{F})$ such that $\,^{t}g_{1}^{-1}g_{1}^{-1}=A$, or equivalently $\tau(g_{1})g_{1}^{-1}=A$. Then we may use the same proof as that in \cite{zou2019supercuspidal}, Lemma 5.11 to obtain $\tau(\mfa'')=\mfa''$.   
			
		\end{proof}
		
		By Lemma \ref{lemmamaxtotalwild}, we may choose $g_{1}\in\mrgl_{n}(\mfo_{F})$ such that $\mfa''=\mfa^{g_{1}}$ is $\tau$-invariant. Let $M=\mfo_{F}^{\times}\times...\times\mfo_{F}^{\times}$ be the subgroup of $\mrgl_{n}(\mfo_{F})$ via diagonal embedding, let $M''=M^{g_{1}}$ and $U''^{1}=U^{1g_{1}}:=U^{1}(\mfa)^{g_{1}}$. Then using directly the proof of \cite{zou2019supercuspidal} Proposition 5.9, with all the Galois involution in  \emph{loc. cit.} replaced by the trivial action, there exists $x\in M''U''^{1}$ such that for $\mfa'=\mfa''^{x}=\mfa^{g_{1}x}$ and $\theta'=\theta^{g_{1}x}$, we have $\tau(\mfa')=\mfa'$ and $\theta'\circ\tau=\theta'^{-1}$. Moreover since $g_{1}x\in g_{1}M''U''^{1}=MU^{1}g_{1}\subset\mrgl_{n}(\mfo_{F})$ and $\mfa\subset\mrm_{n}(\mfo_{F})$, we get $\mfa'=\mfa^{g_{1}x}\subset\mrm_{n}(\mfo_{F})$.
		
	\end{proof}
	
	\begin{remark}
	
	Since we didn't provide a full proof here, at least we explain where the condition $E/F$ being maximal and wildly ramified is used. First when $E/F$ is totally ramified and $[E:F]=n$, the quotients $\mfa^{\times}/1+\mfp_{\mfa}$ and $1+\mfp_{\mfa}^{i}/1+\mfp_{\mfa}^{i+1}$ for $i=1,2,...$ are abelian groups, which is crucial for the corresponding argument in \cite{zou2019supercuspidal}, Lemma 5.12, Lemma 5.13 and Lemma 5.14. Secondly the condition $n$ being odd is also used, which is because $p$ is odd and $n$ is a power of $p$.
		
	\end{remark}
	
	\subsection{The maximal case}
	
	In this subsection, we further use the result proved in \S \ref{subsecmaxwildram} to consider the following special case of Theorem \ref{thmendotauselfdualstra}.
	
	\begin{proposition}\label{proptauselfdualstramax}
		
		Let $[\mfa,\beta]$ be a simple stratum in $\mrm_{n}(F)$ and let $\theta\in\mcc(\mfa,\beta)$ be a simple character with $n=d$. Then for an orthogonal involution $\tau=\tau_{\varepsilon}$ which is $G$-conjugate to $\tau_{J_{n}}$, there exists a simple stratum $[\mfa',\beta']$ and a simple character $\theta'\in\mcc(\mfa',\beta')$ such that $(\mfa',\theta')$ is $G$-conjugate to $(\mfa,\theta)$ with the property $\tau(\mfa')=\mfa',$ $\theta'\circ\tau=\theta'^{-1}$ and $\tau(\beta')=\beta'^{-1}$.
		
	\end{proposition}
	
	\begin{remark}\label{remmaximpmaxwild}
		
		If we assume $E/F$ to be totally wildly ramified, then by direct calculation and Lemma \ref{lemmaGLnOFsymhasse}, we have
		$$\mathrm{det}(I_{n})=\mathrm{det}(J_{n})\ \text{or}\ \mathrm{det}(-J_{n})\quad\text{and}\quad\mathrm{Hasse}(I_{n})=\mathrm{Hasse}(J_{n})=\mathrm{Hasse}(-J_{n})=1.$$
		Thus $I_{n}$ is $G$-conjugate to $J_{n}$ or $-J_{n}$, which means that $\tau_{I_{n}}$ is $G$-conjugate to $\tau_{J_{n}}$. Choosing $\varepsilon=I_{n}$, Proposition \ref{proptauselfdualstramax} implies Proposition \ref{proptauselfdualstramaxwild}.
		
	\end{remark}
	
	\begin{remark}
		
		Since $\tau_{J_{n}}$ represents the split orthogonal group, it satisfies the condition of Theorem \ref{thmendotauselfdualstra}, which justifies that Proposition \ref{proptauselfdualstramax} is indeed a special case of Theorem \ref{thmendotauselfdualstra}.
		
	\end{remark}
	
	\begin{proof}
		
		We write $n=t(n/t)$ with $t=[T:F]$ and $n/t$ a power of $p$ as an odd number, where $T$ is the maximal tamely ramified subextension of $E$ over $F$. We define
		$$J_{t,n/t}:=\mathrm{diag}(J_{t},...,J_{t})$$
		as a matrix in $\mrm_{n/t}(\mrm_{t}(F))=\mrm_{n}(F)$. Using Lemma \ref{lemmaGLnOFsymhasse}, we have $$\mathrm{Hasse}(J_{t,n/t})=\mathrm{Hasse}(J_{n})=\mathrm{Hasse}(-J_{n})=1.$$
		Moreover by direct calculation we have
		$$\mathrm{det}(J_{t,n/t})=\mathrm{det}(J_{n})\ \text{or}\ \mathrm{det}(-J_{n}).$$
		Thus using Proposition \ref{propSGLEorbit}, $J_{t,n/t}$ is similar to $J_{n}$ or $-J_{n}$. Thus $\tau_{J_{t,n/t}}$ is $G$-conjugate to $\tau_{J_{n}}$ and $\tau_{\varepsilon}$. By Proposition \ref{proptauGLEorbit}, we may replace $\varepsilon$ by multiplying an element in $F^{\times}$ to make sure that $\varepsilon$ is similar to $J_{t,n/t}$. Thus using Lemma \ref{lemmatauGaction}, we only need to consider the case where $\varepsilon=J_{t,n/t}$ and $\tau=\tau_{J_{t,n/t}}$. So from now on we assume $\varepsilon=J_{t,n/t}$.
		
		Using Proposition \ref{proptameembed}, we may choose
		$$\iota:T\hookrightarrow\mrm_{t}(F)$$
		to be an $F$-algebra embedding which is $J_{t}$-symmetric. By abuse of notation, we consider the following embedding
		$$\iota:\mrm_{n/t}(T)\hookrightarrow\mrm_{n/t}(\mrm_{t}(F))=\mrm_{n}(F)$$
		given by mapping each entry $T$ to the corresponding $\mrm_{t}(F)$ via the original $\iota$. If we regard $T$ as an $F$-subalgebra of $\mathrm{M}_{n/t}(T)$ given by the diagonal embedding, then $\iota(T)^{\times}$ is fixed by $\tau$. By the Skolem-Noether theorem, we may choose $g\in G$ such that $\iota(T)=T^{g}$. Thus using $[\mfa^{g},\beta^{g}]$ and $\theta^{g}$ to replace $[\mfa,\beta]$ and $\theta$, we may suppose $\iota(T)$ to be the maximal tamely ramified extension with respect to $E/F$. Thus we identify $T$ with $\iota(T)$ and omit $\iota$.
		
		Let $C=\mathrm{M}_{n/t}(T)$ denote the centralizer of $T$ in $\mathrm{M}_{n}(F)$ and let $t_{C}$ denote the transpose on $C$. For $c=(c_{ij})_{ij}\in \mrgl_{n/t}(T)$, we have
		$$\tau(c)=(J_{t,n/t}^{-1}\,^{t}cJ_{t,n/t})^{-1}=((J_{t}^{-1}\,^{t}c_{ji}J_{t})_{ij})^{-1}=((c_{ji})_{ij})^{-1}=\,^{t_{C}}c^{-1}=\tau'(c)$$
		where we use the fact that $\iota$ is $J_{t}$-symmetric and we write $\tau'(x)=\,^{t_{C}}x^{-1}$ for any $x\in C^{\times}$. Thus $\tau'$ as the restriction of $\tau$ to $C^{\times}$ is the orthogonal involution $\tau_{I_{n/t}}$ on $C^{\times}=\mathrm{GL}_{n/t}(T)$. As mentioned in \S \ref{subsectiontypetheory}, the intersection $\mathfrak{c}=\mathfrak{a}\cap C$ gives rise to a simple stratum $[\mathfrak{c},\beta]$ and the restriction of $\theta$ to $H^{1}(\mathfrak{c},\beta)$, denoted by $\theta_{T}$, is the interior $T/F$-lift of $\theta$. Since $E/T$ is totally wildly ramified, using Proposition \ref{proptauselfdualstramaxwild} with $G$, $\theta$ and $\tau$ replaced by $C^{\times}$, $\theta_{T}$ and $\tau'$ respectively, there exists $c\in C^{\times}$ such that $\tau'(\mathfrak{c}^{c})=\mathfrak{c}^{c}$ and $\theta_{T}^{c}\circ\tau'=(\theta_{T}^{c})^{-1}$. As a corollary, we also have $\tau'(H^{1}(\mathfrak{c}^{c},\beta^{c}))=H^{1}(\mathfrak{c}^{c},\beta^{c})$ and $\mcc(\mathfrak{c}^{c},\beta^{c})=\mcc(\mathfrak{c}^{c},\,^{t_{C}}(\beta^{c}))$.\footnote{It is because $(\theta_{T}^{c})^{-1}\in\mcc(\mathfrak{c}^{c},\beta^{c})$ and $\theta_{T}^{c}\circ\tau'=(\theta_{T}^{c})^{-1}\circ\,^{t_{C}}\in\mcc(\mathfrak{c}^{c},\,^{t_{C}}(\beta^{c}))$.}
		
		By the injectivity of $\mathfrak{a}\mapsto \mathfrak{a}\cap C$ between sets of hereditary orders mentioned in \S \ref{subsectiontypetheory},  $\mathfrak{a}':=\mathfrak{a}^{c}$ is $\tau$-stable. Moreover if we write $\theta'=\theta^{c}$ and $T'=T^{c}$, then from our construction of $\tau$ and the definition of $T'/F$-lift, we know that $$(\theta'\circ\tau)_{T'}=\theta'\circ\tau|_{H^{1}(\mathfrak{c}^{c},\beta^{c})}=\theta'\circ\tau'|_{H^{1}(\mathfrak{c}^{c},\beta^{c})}=\theta_{T'}'\circ\tau'$$
		and
		$$(\theta'^{-1})_{T'}=\theta_{T'}'^{-1}$$ are equal. Thus by the last paragraph of \S \ref{subsectiontypetheory}, the simple character $\theta'$ satisfies the property $\theta'\circ\tau=\theta'^{-1}$.
		
		Finally using Lemma \ref{lemmastevensbeta} with $\varepsilon=J_{t,n/t}$, we may choose $\beta'$ in the simple stratum such that $\theta'\in\mcc(\mfa',\beta')$ and $\tau(\beta')=\beta'^{-1}$, thus we finish the proof.
		
	\end{proof}
	
	Before we prove the general case, we state and prove the following important lemma which studies the set $\varepsilon E'^{\times}$ consisting of symmetric matrices, where $E'=F[\beta']$ with $\beta'$ chosen as in Proposition \ref{proptauselfdualstramax}.
	
	\begin{lemma}\label{lemmaJdE'JtT}
		
		We may choose $[\mfa',\beta']$ and $\theta'\in\mcc(\mfa',\beta')$ satisfying the conclusion of Proposition \ref{proptauselfdualstramax} and $T$ a tame parameter field of $\theta'$, and we may fix $\iota:T\hookrightarrow\mrm_{t}(F)$ a $J$-symmetric embedding given by Proposition \ref{proptameembed}, such that for any $x\in E'^{\times}$, there exists $x_{t}\in T^{\times}$ such that $\varepsilon x$ is similar to $\mathrm{diag}(J_{t}\iota(x_{t}),...,J_{t}\iota(x_{t}))$.
		
	\end{lemma}
	
	\begin{proof}
		
		First we assume $\varepsilon=J_{t,d/t}$. We recall that in the proof of Proposition \ref{proptauselfdualstramax}, first we obtain a simple stratum $[\mfa',\beta]$ and a simple character $\theta'\in\mcc(\mfa',\beta)$, such that $\tau(\mfa')=\mfa'$ and $\theta'\circ\tau=\theta'^{-1}$, then we use Lemma \ref{lemmastevensbeta} to get $\beta'$. In this case we have $\theta'\in\mcc(\mfa',\beta)\cap\mcc(\mfa',\beta')$, thus $J^{1}(\mfa',\beta)=J^{1}(\mfa',\beta')$ as the maximal pro-$p$-subgroup of the normalizer of $\theta$. Moreover from our construction of $[\mfa',\beta]$, for $T$ the maximal tamely ramified subextension of $E/F$ with $E=F[\beta]$ and for $\iota:T\hookrightarrow\mrm_{t}(F)$ the chosen $J$-symmetric embedding, we have $$T=\{\mathrm{diag}(\iota(x_{t}),...,\iota(x_{t}))\in\mrm_{d/t}(\mrm_{t}(F))=\mrm_{d}(F)|x_{t}\in T\}.$$
		Thus we get
		$$\varepsilon T^{\times}=\{\mathrm{diag}(J_{t}\iota(x_{t}),...,J_{t}\iota(x_{t}))\in\mrm_{d/t}(\mrm_{t}(F))=\mrm_{d}(F)|x_{t}\in T^{\times}\}.$$
		
		We write $T'$ for the maximal tamely ramified subextension of $E'/F$ with $E'=F[\beta']$. By Lemma \ref{lemmaE/Lextension} with $E=E'$ and $L=T'$, the embedding $T'\hookrightarrow E'$ induces an isomorphism $$T'^{\times}/T'^{\times2}\simeq E'^{\times}/E'^{\times2}.$$
		Thus for any $x\in E'^{\times}$, there exists $y\in E'^{\times}$ such that $xy^{2}\in T'^{\times}$. Thus $$\varepsilon x=\,^{t}y^{-1}\varepsilon (xy^{2})y^{-1},$$
		where we use the fact that $\varepsilon^{-1}\,^{t}y^{-1}\varepsilon=y^{-1}$. Thus every element in $\varepsilon E'^{\times}$ is similar to an element in $\varepsilon  T'^{\times}$. Thus to finish the proof, we only need to show that any element in $\varepsilon  T'^{\times}$ is similar to an element in $\varepsilon T^{\times}$.
		
		Using \cite{bushnell2014effective}, Proposition 2.6, there exists $j\in J^{1}(\mfa',\beta)=J^{1}(\mfa',\beta')$ such that $T'=T^{j}$. For any $x\in T^{\times}$, we have $j^{-1}xj\in T'^{\times}$. Thus we get $\tau(x)=x^{-1}$ and $\tau(j^{-1}xj)=(j^{-1}xj)^{-1}$, which implies that
		$$kxk^{-1}=\tau(x^{-1})=x,$$ where $k:=\tau(j)j^{-1}\in C\cap J^{1}(\mfa',\beta)=J^{1}(\mfc',\beta)\subset U^{1}(\mfc')$ with $C=Z_{\mrm_{d}(F)}(T)=\mrm_{d/t}(T)$. Moreover we have
		$$\varepsilon j^{-1}xj=(\varepsilon j^{-1}\varepsilon^{-1})\varepsilon xj=\,^{t}\tau(j)\varepsilon xj=\,^{t}j\,^{t}k\varepsilon xj,$$
		So we only need to show that $\,^{t}k\varepsilon x$ is similar to $\varepsilon x$.
		
		We denote by $\tau'$ the restriction of $\tau$ to $C^{\times}$, thus by definition $\tau'(c)=\,^{t_{C}}c^{-1}$ for any $c\in C^{\times}$, where $t_{C}$ denotes the transpose on $C$. Since $\tau(k)k=1$, we have $\tau'(k)k=1$, or equivalently $\,^{t_{C}}k=k$. Since $\mathrm{det}_{C}(k)\in 1+\mfp_{T}\subset T^{\times 2}$ and $\mathrm{Hasse}_{T}(k)=1$ by Proposition \ref{proptauselfdualstramaxwild} and Lemma \ref{lemmaGLnOFsymhasse}, by Proposition \ref{propSGLEorbit}, there exists $m\in C^{\times}$ such that
		$$\,^{t_{C}}mm=\,^{t_{C}}k\quad \text{or equivalently}\quad \tau(m)^{-1}m=\tau(k)^{-1},$$
		where we denote by $\mathrm{det}_{C}$ the determinant with respect to $C=\mrm_{d/t}(T)$ and by $\mathrm{Hasse}_{T}$ the Hasse invariant with respect to $T$.
		Thus
		$$\,^{t}k\varepsilon x=\varepsilon\tau(k)^{-1}x=\varepsilon\tau(m)^{-1}mx=\,^{t}m\varepsilon mx=\,^{t}m\varepsilon xm,$$
		which means that $\,^{t}k\varepsilon x$ is similar to $\varepsilon x$. So we finish the proof when $\varepsilon=J_{t,d/t}$.
		
		For the general case, since $\tau_{\varepsilon}$ and $\tau_{J_{t,d/t}}$ are $G$-conjugate, we may choose $\varepsilon$ up to multiplying an element in $F^{\times}$, such that $\varepsilon=\,^t gJ_{t,d/t}g$ with a certain $g\in G$. We assume that $[\mfa',\beta']$ and $\theta'$ satisfy this lemma for $\tau=\tau_{J_{t,d/t}}$. We choose $[\mfa'',\beta'']:=[\mfa'^{g},\beta'^{g}]$, $\theta''=\theta'^{g}$,
		and by Lemma \ref{lemmatauGaction} we have
		$$\tau_{\varepsilon}(\mfa'')=\mfa'',\quad \theta''\circ\tau_{\varepsilon}=\theta''^{-1}\quad \text{and}\quad \tau_{\varepsilon}(\beta'')=\beta''^{-1}.$$
		Moreover we have
		$$\varepsilon E''^{\times}=\,^{t}gJ_{t,d/t}gE'^{\times g}=\,^{t}g(J_{t,d/t} E'^{\times})g,$$
		which means that each element in $\varepsilon E''^{\times}$ is similar to an element in $J_{t,d/t} E'^{\times}$. Thus $[\mfa'',\beta'']$, $\theta''$ satisfy the condition of the lemma when $\tau=\tau_{\varepsilon}$.
		
		
	\end{proof}
	
	\begin{remark}\label{remxxtclass}
		
		From the proof we may further observe that when $\varepsilon=J_{t,d/t}$, if we identify $T$ with the maximal tamely ramified subextension of $E'$ over $F$ via an $F$-embedding, then $x$ and $x_{t}$ are in the same class of $T^{\times}/T^{\times2}\simeq E'^{\times}/E'^{\times 2}$ given by Lemma \ref{lemmaE/Lextension} for $E=E'$ and $L=T$.
		
	\end{remark}
	
	Finally we state and prove the following corollary, saying the results for calculating Hasse invariant in \S \ref{subsechilberthasse} can be generalized to certain cases where $E/F$ is not necessarily tamely ramified.
	
	\begin{corollary}\label{corcalhasse}
		
		For $\varepsilon=J_{d}$ and $[\mfa',\beta']$, $\theta'$ constructed in Lemma \ref{lemmaJdE'JtT}, the results in Lemma \ref{lemmahassecase3a}, Lemma \ref{lemmahassecase2b2}, Corollary \ref{corhassecase3b}, Lemma \ref{lemmahassecase2b1}, Lemma \ref{lemmahassecase2a1},  Lemma \ref{lemmahassecase2a2} hold for $E=E'$.
		
	\end{corollary}
	
	\begin{proof}
		
		Since all the proofs are similar, we only prove Lemma \ref{lemmahassecase2a1} as an example.
		
		First of all when $d$ is even, by direct calculation and Lemma \ref{lemmaGLnOFsymhasse} we have $\mrdet(J_{t,d/t})=\mrdet(J_{d})$ and $\mathrm{Hasse}(J_{t,d/t})=\mathrm{Hasse}(J_{d})=1$. Thus $J_{t,d/t}$ is similar to $J_{d}$. Using this fact and Remark \ref{remxxtclass}, we deduce that when $\varepsilon=J_{d}$, we may assume $x$ and $x_{t}$ in the result of Lemma \ref{lemmaJdE'JtT} to be in the same class of $E'^{\times}/E'^{\times 2}\simeq T^{\times}/T^{\times2}$, where we identify $T$ with the maximal tamely ramified subextension of $E'$ over $F$ via an embedding. In particular, when $x=\varpi_{E'}$ is a uniformizer of $E'$, we may assume $x_{t}=\varpi_{T}$ to be a uniformizer of $T$ in the same class as that of $\varpi_{E'}$, and when $x=\varpi_{E'}\epsilon_{0}$ with $\epsilon_{0}$ an element in $\mfo_{E'}^{\times }\backslash\mfo_{E'}^{\times 2}$, we may also assume $x_{t}=\varpi_{T}\epsilon_{0}'$ with $\epsilon_{0}'$ an element in $\mfo_{T}^{\times }\backslash\mfo_{T}^{\times 2}$. Thus using Lemma \ref{lemmaJdE'JtT} for $x=\varpi_{E'}$ and $x=\varpi_{E'}\epsilon_{0}$, we have
		\begin{align*}
			& \quad\ \mathrm{Hasse}(\mathrm{diag}(J_{d}\varpi_{E'},...,J_{d}\varpi_{E'},J_{d}\varpi_{E'}\epsilon_{0}))\\
			&=\mathrm{Hasse}(\mathrm{diag}(\mathrm{diag}(J_{t}\varpi_{T},...,J_{t}\varpi_{T}),...,\mathrm{diag}(J_{t}\varpi_{T},...,J_{t}\varpi_{T}),\mathrm{diag}(J_{t}\varpi_{T}\epsilon_{0}',...,J_{t}\varpi_{T}\epsilon_{0}')))\\
			&=\mathrm{Hasse}(\mathrm{diag}(\mathrm{diag}(J_{t}\varpi_{T},...,J_{t}\varpi_{T},J_{t}\varpi_{T}\epsilon_{0}'),...,\mathrm{diag}(J_{t}\varpi_{T},...,J_{t}\varpi_{T},J_{t}\varpi_{T}\epsilon_{0}')))\\
			&=\mathrm{Hasse}(\mathrm{diag}(J_{t}\varpi_{T},...,J_{t}\varpi_{T},J_{t}\varpi_{T}\epsilon_{0}'))^{n/t}\\
			&=\mathrm{Hasse}(\mathrm{diag}(J_{t}\varpi_{T},...,J_{t}\varpi_{T},J_{t}\varpi_{T}\epsilon_{0}')),
		\end{align*}
		where the matrix in the third line is the direct sum of $n/t$ copies of $\mathrm{diag}(J_{t}\varpi_{T},...,J_{t}\varpi_{T},J_{t}\varpi_{T}\epsilon_{0}')\in\mrm_{tm}(F),$ and for the fourth line we use the fact that $\mathrm{det}(\mathrm{diag}(J_{t}\varpi_{T},...,J_{t}\varpi_{T},J_{t}\varpi_{T}\epsilon_{0}'))$ is of even order in $F^{\times}$ and Corollary \ref{corhassekmatrix}, and for the final line we use the fact that $n/t$ is odd. Thus we may use the tamely ramified case to finish the proof.
		
		When $d$ is odd, if $\mrdet(J_{t,d/t})=\mrdet(J_{d})$ we can still follow the proof above verbatim. If $\mrdet(J_{t,d/t})=\mrdet(-J_{d})$, we deduce that $J_{t,d/t}$ is similar to $-J_{d}$. Thus following the above proof, when $x=\varpi_{E'}$ (resp. $\varpi_{E'}\epsilon_{0}$) we may choose $x_{t}=-\varpi_{T}$ (resp. $-\varpi_{T}\epsilon_{0}'$), where $\varpi_{E'}$, $\varpi_{T}$, $\epsilon_{0}$, $\epsilon_{0}'$ are defined as above. Thus for $\varpi_{T}'=-\varpi_{T}$ as a uniformizer of $T$ and using the same calculation, we have
		$$\mathrm{Hasse}(\mathrm{diag}(J_{d}\varpi_{E'},...,J_{d}\varpi_{E'},J_{d}\varpi_{E'}\epsilon_{0}))=\mathrm{Hasse}(\mathrm{diag}(J_{t}\varpi_{T}',...,J_{t}\varpi_{T}',J_{t}\varpi_{T}'\epsilon_{0}')).$$
		And still we use the tamely ramified case to finish the proof.
		
	\end{proof}
	
	\subsection{The general case}\label{subsectiongeneral}
	
	In this subsection, we finish the proof of Theorem \ref{thmendotauselfdualstra}. For $[\mfa,\beta]$ and $\theta\in\mcc(\mfa,\beta)$ given as in the theorem, we choose $\beta_{0}\in\mrm_{d}(F)$ such that there exists an $F$-algebra isomorphism $F[\beta_{0}]\rightarrow F[\beta]$ which maps $\beta_{0}$ to $\beta$. Let $\mfa_{0}$ be the unique hereditary order of $\mrm_{d}(F)$ normalized by $\beta_{0}$. Thus $[\mfa_{0},\beta_{0}]$ is a simple stratum of $\mrm_{d}(F)$ and we let $\theta_{0}=t_{\mfa,\mfa_{0}}^{\beta,\beta_{0}}(\theta)$ be the transfer of $\theta$ as a simple character with respect to $[\mfa_{0},\beta_{0}]$. Using Proposition \ref{proptauselfdualstramax}, for $\tau_{J_{d}}$ the involution on $\mrgl_{d}(F)$, there exist a simple stratum $[\mfa_{0}',\beta_{0}']$ and a simple character $\theta_{0}'\in\mcc(\mfa_{0}',\beta_{0}')$ such that $(\mfa_{0}',\theta_{0}')$ is $\mrgl_{d}(F)$-conjugate to $(\mfa_{0},\theta_{0})$ with the following property:
	
	(1) $\tau_{J_{d}}(\mfa_{0}')=\mfa_{0}'$ and $\tau_{J_{d}}(H^{1}(\mfa_{0}',\beta_{0}'))=H^{1}(\mfa_{0}',\beta_{0}')$;
	
	(2) $\theta_{0}'\circ\tau_{J_{d}}=\theta_{0}'^{-1}$;
	
	(3) $\tau_{J_{d}}(\beta_{0}')=\beta_{0}'^{-1}$;
	
	(4) Corollary \ref{corcalhasse} holds.
	
	Now we embed $\mathrm{M}_{d}(F)$ diagonally in  $\mathrm{M}_{n}(F)$, which gives an $F$-algebra homomorphism $\iota':F[\beta_{0}']\hookrightarrow\mathrm{M}_{n}(F)$. We write $\beta'=\iota'(\beta_{0}')=\beta_{0}'\otimes...\otimes\beta_{0}'$ and $E'=F[\beta']$. The centralizer of $E'$ in $\mathrm{M}_{n}(F)$, denoted by $B'$, is naturally identified with $\mathrm{M}_{m}(E')$. Let $\mathfrak{b}'$ be a maximal standard hereditary order in $B'$ which may be identified with $\mathrm{M}_{m}(\mathfrak{o}_{E'})$, and let $\mathfrak{a}'$ be the unique hereditary order of $\mathrm{M}_{n}(F)$ normalized by $E'^{\times}$ such that $\mathfrak{a}'\cap B'=\mathfrak{b}'$. Then we obtain a simple stratum $[\mathfrak{a}',\beta']$ in $\mathrm{M}_{n}(F)$. Let $\theta'=t_{\mfa_{0}',\mfa'}^{\beta_{0}',\beta'}(\theta_{0}')\in\mathcal{C}(\mathfrak{a}_{0}',\beta_{0}')$ be the transfer of $\theta_{0}'$.
	
	We denote by $T'$ the maximal tamely ramified subextension of $E'/F$ and we denote by $T_{m}'$ an unramified extension of degree $m$ over $T'$. We denote by $E_{m}'=T_{m}'E'$ an unramified extension of degree $m$ over $E'$. Since $E'/T'$ and $E_{m}'/T_{m}'$ are totally wildly ramified, it is easy to check that
	\begin{equation}\label{eqT'E'F}
		\mrn_{T'/F}(T'^{\times})F^{\times2}/F^{\times2}=\mrn_{E'/F}(E'^{\times})F^{\times2}/F^{\times2}
	\end{equation}
	and
	\begin{equation}\label{Tm'Em'F}
		\mrn_{T_{m}'/F}(T_{m}'^{\times})F^{\times2}/F^{\times2}=\mrn_{E_{m}'/F}(E_{m}'^{\times})F^{\times2}/F^{\times2}.
	\end{equation}
	The latter group is a subgroup of the former one, and both of them are subgroups of $F^{\times}/F^{\times2}$, which is a group of order four.
	
	We consider the following special orthogonal involutions $\tau=\tau_{\varepsilon}$ such that
	
	\textbf{Case (\romannumeral1)} If $\mrn_{T_{m}'/F}(T_{m}'^{\times})F^{\times2}/F^{\times2}=F^{\times}/F^{\times2}$, then $\varepsilon=J_{d,m}=\mathrm{diag}(J_{d},...,J_{d})\in\mrm_{m}(\mrm_{d}(F))=\mrm_{n}(F)$;
	
	\textbf{Case (\romannumeral2)} If $\mrn_{T_{m}'/F}(T_{m}'^{\times})F^{\times2}/F^{\times2}$ is a subgroup of $F^{\times}/F^{\times2}$ of order two, we consider the following two cases:
	
	\textbf{(\romannumeral2.a)} If $2|m$, then $\varepsilon$ equals $J_{d,m}$ or $\mathrm{diag}(J_{d},...J_{d},J_{d}\epsilon)$, where $\epsilon\in\mfo_{E'}^{\times}$;
	
	\textbf{(\romannumeral2.b)} If $2\nmid m$, then $\varepsilon$ equals $J_{d,m}$ or $\mathrm{diag}(J_{d}\epsilon,...,J_{d}\epsilon)$, where $\epsilon\in E'^{\times}$;
	
	\textbf{Case (\romannumeral3)} If $\mrn_{T_{m}'/F}(T_{m}'^{\times})F^{\times2}/F^{\times2}=\{1\}$, then $\varepsilon$ equals $J_{d,m}$ or $\mathrm{diag}(J_{d}\varpi_{E'},...,J_{d}\varpi_{E'}\epsilon)$, where $\epsilon\in\mfo_{E'}^{\times}$ and $\varpi_{E'}$ is a certain uniformizer of $E'$. We distinguish the following two cases:
	
	\textbf{(\romannumeral3.a)} $\mrn_{T'/F}(T'^{\times})/F^{\times2}=\{1\}$;
	
	\textbf{(\romannumeral3.b)}
	$\mrn_{T'/F}(T'^{\times})F^{\times2}/F^{\times2}$ is not trivial.
	
	We want to check that for $[\mfa',\beta']$, $\theta'$ and $\tau=\tau_{\varepsilon}$ given as above, the conditions (1), (2) and (3) in Theorem \ref{thmendotauselfdualstra} are satisfied.
	For each $\varepsilon$ above, we may write $\varepsilon=J_{d,m}a_{\varepsilon}\varepsilon_{E'}$, where $a_{\varepsilon}\in E'^{\times}$ and $\varepsilon_{E'}=\mathrm{diag}(1,...,1,\epsilon)\in\mrgl_{m}(E')$ with $\epsilon\in\mfo_{E'}^{\times}.$ Thus for $x=(x_{ij})_{ij}\in\mrgl_{m}(E')$, we have
	\begin{align}
		\tau(x)&=((J_{d,m}a_{\varepsilon}\varepsilon_{E'})^{-1}\,^{t}((x_{ij})_{ij})J_{d,m}a_{\varepsilon}\varepsilon_{E'})^{-1}=((\varepsilon_{E'}^{-1}a_{\varepsilon}^{-1}((J_{d}^{-1}\,^{t}x_{ji}J_{d})_{ij})a_{\varepsilon}\varepsilon_{E'})^{-1} \nonumber \\
		&=(\varepsilon_{E'}^{-1}a_{\varepsilon}^{-1}((x_{ji})_{ij})a_{\varepsilon}\varepsilon_{E'})^{-1}=(\varepsilon_{E'}^{-1}(\,^{t_{E'}}x)\varepsilon_{E'})^{-1}=\tau'(x),\label{eqtauxtau'x}
	\end{align}
	where we write $\,^{t_{E'}}$ for the transpose on $\mrgl_{m}(E')$ and $\tau':=\tau_{\varepsilon_{E'}}$ for the orthogonal involution defined on $\mrgl_{m}(E')$ corresponding to $\varepsilon_{E'}$, and we use the fact that the embedding $E'\hookrightarrow\mrm_{d}(F)$ is $J_{d}$-symmetric and $a_{\varepsilon}$ commutes with elements in $\mrgl_{m}(E')$. Thus we proved that the restriction of $\tau$ to $\mathrm{GL}_{m}(E')$ equals $\tau'$ as an orthogonal involution on $\mrgl_{m}(E')$. In particular, since $\epsilon$ is an element in $E'$, we know that $\varepsilon_{E'}$ commutes with elements in $E'$ and we have $\tau(\beta')=\beta'^{-1}$. Thus condition (3) is verified.
	
	Since $\mfb'$ is a maximal standard hereditary order in $B'$ which may be identified with $\mathrm{M}_{m}(\mathfrak{o}_{E'})$, it is $\tau'$-stable. Thus from our assumption of $\tau$ and construction of $\mfa'$, we deduce that $\mathfrak{a}'$ is $\tau$-stable. 
	By definition $H^{1}(\mathfrak{a}',\beta')$ is $\tau$-stable, which means that condition (1) is verified.
	
	Let $M$ be the standard Levi subgroup of $G$ isomorphic to $\mathrm{GL}_{d}(F)\times...\times\mathrm{GL}_{d}(F)$. Let $P$ be the standard parabolic subgroup of $G$ generated by $M$ and upper triangular matrices, and let $N$ be its unipotent radical. Let $N^{-}$ be the unipotent radical of the parabolic subgroup opposite to $P$ with respect to $M$. By \cite{secherre2008representations}, Th\'eor\`eme 2.17, we have
	\begin{align}\label{eqHNMN}
		H^{1}(\mathfrak{a}',\beta')&=(H^{1}(\mathfrak{a}',\beta')\cap N^{-})\cdot(H^{1}(\mathfrak{a}',\beta')\cap M)\cdot(H^{1}(\mathfrak{a}',\beta')\cap N),\\
		H^{1}(\mathfrak{a}',\beta')\cap M&=H^{1}(\mathfrak{a}_{0}',\beta'_{0})\times...\times H^{1}(\mathfrak{a}_{0}',\beta'_{0}).
	\end{align}
	By \emph{loc. cit.}, the character $\theta'$ is trivial on $H^{1}(\mathfrak{a}',\beta')\cap N^{-}$ and $H^{1}(\mathfrak{a}',\beta')\cap N$, and the restriction of $\theta'$ to $H^{1}(\mathfrak{a}',\beta')\cap M$ equals $\theta_{0}'\otimes...\otimes\theta_{0}'$. We have $$\theta'\circ\tau|_{H^{1}(\mathfrak{a}',\beta')\cap N^{-}}=\theta'\circ\tau|_{H^{1}(\mathfrak{a}',\beta')\cap N}=\theta'^{-1}|_{H^{1}(\mathfrak{a}',\beta')\cap N^{-}}=\theta'^{-1}|_{H^{1}(\mathfrak{a}',\beta')\cap N}=1.$$
	Moreover since $\tau=\tau_{\varepsilon}$ with $$\varepsilon=\mathrm{diag}(J_{d},...,J_{d})\ \text{or}\ \mathrm{diag}(J_{d}\epsilon,...,J_{d}\epsilon)\ \text{or}\  \mathrm{diag}(J_{d},...,J_{d},J_{d}\epsilon)\ \text{or}\
	\mathrm{diag}(J_{d}\varpi_{E'},...,J_{d}\varpi_{E'},J_{d}\varpi_{E'}\epsilon),$$
	and since $\epsilon$ and $\varpi_{E'}$ normalize $\theta_{0}'$,
	we have
	$$\theta'\circ\tau|_{H^{1}(\mathfrak{a}',\beta')\cap M} =\theta_{0}'\circ\tau_{J_{d}}\otimes...\otimes\theta_{0}'\circ\tau_{J_{d}}=\theta_{0}'^{-1}\otimes...\otimes\theta_{0}'^{-1}=\theta'^{-1}|_{H^{1}(\mathfrak{a}',\beta')\cap M}.$$
	Thus by equation (\ref{eqHNMN}), we have $\theta'\circ\tau=\theta'^{-1}$, which is the condition (2). Thus for those special orthogonal involutions, we finish the proof.
	
	Finally we show that for a given orthogonal involution $\tau$ and the corresponding orthogonal group $H=G^{\tau}$ satisfying the condition of Theorem \ref{thmendotauselfdualstra}, $\tau$ is conjugate to one of the orthogonal involutions mentioned in \textbf{Case (\romannumeral1), (\romannumeral2)} or \textbf{(\romannumeral3)}. We consider them separately.
	
	\textbf{Case (\romannumeral1)} By definition,
	\begin{equation}\label{eqTmF}
		\mrn_{E_{m}'/F}(E_{m}'^{\times})F^{\times2}/F^{\times2}=\mrn_{T_{m}'/F}(T_{m}'^{\times})F^{\times2}/F^{\times2}=F^{\times}/F^{\times2}
	\end{equation}
	then using Proposition \ref{propE/Fbasic} for $E=T_{m}'$, we deduce that $[T_{m}':F]$ is odd, thus $n=[E_{m}':F]$ is odd. 
	By Lemma \ref{lemmaGLnOFsymhasse}, we have
	$$\mathrm{Hasse}(J_{n})=\mathrm{Hasse}(-J_{n})=\mathrm{Hasse}(J_{d,m})=1.$$
	And moreover
	$$\mathrm{det}(J_{d,m})=\mathrm{det}(J_{n})\quad \text{or}\quad \mathrm{det}(J_{d,m})= \mathrm{det}(-J_{n}).$$
	So by Proposition \ref{propSGLEorbit}, $J_{d,m}$ is similar to $J_{n}$ or $-J_{n}$, which means that $\tau_{J_{n}}$ and $\tau_{J_{d,m}}$ are in the same $G$-orbit.
	
	\textbf{Case (\romannumeral2)} By Lemma \ref{lemmaGLnOFsymhasse}, we have
	$$\mathrm{Hasse}(J_{n})=\mathrm{Hasse}(J_{d,m})=1.$$
	
	\textbf{(\romannumeral2.a)} Since $T_{m}'/T'$ is unramified and $m$ is even, we get
	$$\mrn_{T_{m}'/F}(T_{m}'^{\times})F^{\times2}/F^{\times2}=\mrn_{T'/F}(\mfo_{T'}^{\times})F^{\times2}/F^{\times2}=\mrn_{T_{m}'/F}(\mfo_{T_{m}'}^{\times})F^{\times2}/F^{\times2}.$$
	Thus using equation (\ref{eqT'E'F}) and (\ref{Tm'Em'F}) we know that
	$$\mrn_{E_{m}'/F}(\mfo_{E_{m}'}^{\times})F^{\times2}/F^{\times2}=\mrn_{T_{m}'/F}(\mfo_{T_{m}'}^{\times})F^{\times2}/F^{\times2}=\mrn_{E'/F}(\mfo_{E'}^{\times})F^{\times2}/F^{\times2}=\mrn_{T'/F}(\mfo_{T'}^{\times})F^{\times2}/F^{\times2}$$ is a subgroup of $F^{\times}/F^{\times2}$ of order two. Thus there exists $\epsilon_{0}\in\mfo_{E'}^{\times}$ such that the image of $\mrn_{E'/F}(\epsilon_{0})$ in $\mrn_{E'/F}(\mfo_{E'}^{\times})F^{\times2}/F^{\times2}$ is nontrivial. From now on we fix one such $\epsilon_{0}$.
	
	\textbf{(\romannumeral2.a.1)} If either of the three cases is true:
	\begin{itemize}
		\item $2|d$;
		\item $2\nmid d$ and $4|m$;
		\item$2\nmid d$, $4\nmid m$ and $-1\in F^{\times2}$,
	\end{itemize}
	then by direct calculation we get
	$$1=\mathrm{disc}(J_{d,m})=\mathrm{disc}(J_{n}).$$
	Thus by Proposition \ref{propSGLEorbit}, $J_{d,m}$ is in the same $G$-orbit as $J_{n}$, representing the $G$-conjugacy class of split orthogonal group. Moreover, we have
	$$\mathrm{det}(J_{d},...,J_{d},J_{d}\epsilon_{0})=\mrn_{E'/F}(\epsilon_{0})$$
	which is non-trivial in $F^{\times}/F^{\times2}$. Thus by Proposition \ref{proptauGLEorbit} and Proposition \ref{propGconjortho}, we know that $\tau_{\varepsilon}$ corresponds to the $G$-conjugacy class of orthogonal groups mentioned in Theorem \ref{thmendotauselfdualstra}, which is quasisplit but not split.
	
	\textbf{(\romannumeral2.a.2)} If $2\nmid d$, $4\nmid m$ and $-1\notin F^{\times2}$, we get
	$$\mrn_{T_{m}'/F}(T_{m}'^{\times})F^{\times2}/F^{\times2}=\{1,-1\}.$$
	By direct calculation we get
	$$\mathrm{det}(\mathrm{diag}(J_{d},...,J_{d},-J_{d}))=\mathrm{det}(J_{n})=-1$$
	and
	$$\mathrm{det}(J_{d,m})=1.$$
	Thus if we further choose $\epsilon=-1$ and $\varepsilon=\mathrm{diag}(J_{d},...,J_{d},-J_{d})$, then by Proposition \ref{proptauGLEorbit} and Proposition \ref{propGconjortho}, $\tau_{\varepsilon}$ and $\tau_{J_{d,m}}$ correspond to the two $G$-conjugacy classes of orthogonal groups respectively mentioned in Theorem \ref{thmendotauselfdualstra}, where the former class is split, and the latter class is quasisplit but not split.
	
	\textbf{(\romannumeral2.b)} Since $m$ is odd, we deduce that
	$$\mrn_{T_{m}'/F}(T_{m}'^{\times})F^{\times2}/F^{\times2}=\mrn_{T'/F}(T'^{\times})F^{\times2}/F^{\times2}=\mrn_{E'/F}(E'^{\times})F^{\times2}/F^{\times2}$$
	and $d$ is even by Proposition \ref{propE/Fbasic} with $E=T'$. We fix $\epsilon\in\mrn_{E'/F}(E'^{\times})$ whose image in $F^{\times}/F^{\times2}$ is non-trivial. By direct calculation we get
	$$\mathrm{det}(J_{d,m})=(-1)^{md(d-1)/2}=(-1)^{md(n-1)/2}=\mathrm{det}(J_{n}).$$
	Thus by Proposition \ref{propSGLEorbit}, $J_{d,m}$ is in the same $G$-orbit as $J_{n}$, representing the $G$-conjugacy class of split orthogonal group. Moreover, we have
	$$\mathrm{det}(J_{d}\epsilon,...,J_{d}\epsilon)=\mrn_{E'/F}(\epsilon)^{m}$$
	which is non-trivial in $F^{\times}/F^{\times2}$. Thus by Proposition \ref{proptauGLEorbit} and Proposition \ref{propGconjortho}, $\tau_{\varepsilon}$ correspond to the $G$-conjugacy class of orthogonal groups mentioned in Theorem \ref{thmendotauselfdualstra}, which is quasisplit but not split.
	
	\textbf{Case (\romannumeral3)} First of all since
	$$\mrn_{E_{m}'/F}(E_{m}'^{\times})F^{\times2}/F^{\times2}=\mrn_{T_{m}'/F}(T_{m}'^{\times})F^{\times2}/F^{\times2}=\{1\},$$
	by Proposition \ref{propE/Fbasic} with $E=T_{m}'$ we know that $4|[T_{m}':F]$. Thus $4|[E_{m}':F]=n$. By direct calculation, we have
	$$\mathrm{det}(J_{d,m})=\mathrm{det}(J_{n})=1.$$
	Moreover, by Lemma \ref{lemmaGLnOFsymhasse} we get
	$$\mathrm{Hasse}(J_{d,m})=\mathrm{Hasse}(J_{n})=1.$$
	Thus by Proposition \ref{propSGLEorbit}, $J_{d,m}$ is in the same $G$-orbit as $J_{n}$, representing the $G$-conjugacy class of split orthogonal group. Thus we only need to show that for $\varepsilon=\mathrm{diag}(J_{d}\varpi_{E'},...,J_{d}\varpi_{E'}\epsilon)$ with $\varpi_{E'}$ and $\epsilon\in\mfo_{E}^{\times}$ well-chosen, $\tau=\tau_{\varepsilon}$ corresponds to the non-quasisplit orthogonal group. By direct calculation, we have
	$$\mathrm{det}(\varepsilon)=(-1)^{n(d-1)/2}\mrn_{E'/F}(\varpi_{E'})^{m}\mrn_{E'/F}(\epsilon)=\mrn_{E'/F}(\varpi_{E'})^{m}\mrn_{E'/F}(\epsilon).$$
	
	\textbf{(\romannumeral3.a)} Since
	$$\mrn_{E'/F}(E'^{\times})F^{\times2}/F^{\times2}=\mrn_{T'/F}(T'^{\times})F^{\times2}/F^{\times2}=\{1\},$$
	$\mathrm{det}(\varepsilon)$ is trivial as an element in $F^{\times}/F^{\times2}$. Thus we only need to choose $\epsilon$ such that $\mathrm{Hasse}(\varepsilon)=-1$. By Lemma \ref{lemmahassecase3a} and Corollary \ref{corcalhasse}, we may choose $\varpi_{E'}$ and $\epsilon_{0}$ such that $\mathrm{Hasse}(J_{d}\varpi_{E'})=1$ and $\mathrm{Hasse}(J_{d}\varpi_{E'}\epsilon_{0})=-1$. Then using Corollary \ref{corhassekmatrix} and the fact that $\mathrm{det}(J_{d}\varpi_{E'})$, $\mathrm{det}(J_{d}\varpi_{E'}\epsilon_{0})\in F^{\times2}$, we get
	$$\mathrm{Hasse}(\varepsilon)=\mathrm{Hasse}(J_{d}\varpi_{E'})^{m-1}\mathrm{Hasse}(J_{d}\varpi_{E'}\epsilon_{0})=-1.$$
	
	\textbf{(\romannumeral3.b)} Since $\mrn_{E'/F}(E'^{\times})F^{\times2}/F^{\times2}$ is not trivial and $\mrn_{E'_{m}/F}(E_{m}'^{\times})F^{\times2}/F^{\times2}$ is trivial, $m$ is even and there exists a uniformizer $\varpi_{F}'$ of $F$ such that
	$$\mrn_{E'/F}(E'^{\times})F^{\times2}/F^{\times2}=\{1,\varpi_{F}'\}.$$
	Thus $$\mathrm{det}(\varepsilon)=\mrn_{E'/F}(\varpi_{E'})^{m}\mrn_{E'/F}(\epsilon)\equiv \mathrm\mrn_{E'/F}(\epsilon)\quad(\text{mod}\ F^{\times2}).$$
	Since $\mathrm\mrn_{E'/F}(\epsilon)\in \mfo_{F}^{\times}\cap \mrn_{E'/F}(E'^{\times})$, its image in $F^{\times}/F^{\times2}$ is trivial, that is, $\mathrm{disc}(\varepsilon)=1$. So as in \textbf{(\rn3.a)}, we only need to show that $\mathrm{Hasse}(\varepsilon)=-1$. Fix $\epsilon_{0}\in\mfo_{E'}^{\times}\backslash\mfo_{E'}^{\times2}$, by Corollary \ref{corhassecase3b} and Corollary \ref{corcalhasse}, we may choose $\epsilon$ equals $1$ or $\epsilon_{0}$, such that
	$$\mathrm{Hasse}(\varepsilon)=\mathrm{Hasse}(\mathrm{diag}(J_{d}\varpi_{E'},...,J_{d}\varpi_{E'},J_{d}\varpi_{E'}\epsilon))=-1.$$
	So we finish the discussion for \textbf{(\rn3.b)}.
	
	Thus for $H=G^{\tau}$ given as an orthogonal group in Theorem \ref{thmendotauselfdualstra} with $\tau=\tau_{\varepsilon}$, we have shown that $\tau$ is $G$-conjugate to one of the special orthogonal involutions mentioned in \textbf{Case (\romannumeral1)}, \textbf{(\romannumeral2)} or \textbf{(\romannumeral3)}. Furthermore, we may change $\varepsilon$ up to multiplying by an element in $E'^{\times}$ such that $\varepsilon$ is similar to one of the special symmetric matrices mentioned in \textbf{Case (\romannumeral1)}, \textbf{(\romannumeral2)} or \textbf{(\romannumeral3)}. Using Lemma \ref{lemmatauGaction} and the special cases proved, we end the proof of Theorem \ref{thmendotauselfdualstra}.
	
	\begin{remark}\label{remendothm}
		
		In the proof of Theorem \ref{thmendotauselfdualstra}, we actually showed that for $\tau$ an involution in \textbf{Case (\romannumeral1)}, \textbf{(\romannumeral2)} or \textbf{(\romannumeral3)}, the choices of $[\mfa',\beta']$ and $\theta'$ are the same. Moreover, $E=E'=F[\beta']$ satisfies Corollary \ref{corcalhasse}, which follows from $E'=F[\beta']\simeq F[\beta_{0}']$ and our choice of $[\mfa_{0}',\beta_{0}']$.
		
	\end{remark}
	
	\section{Distinguished type theorem and the orbits of distinguished type}\label{sectiondistypethm}
	
	Let $\pi$ be a supercuspidal representation of $G$, let $T$ be a tame parameter field of $\pi$ and let $T_{m}$ be the unramified extension of degree $m$ over $T$, where $n=md$ is determined by $\pi$ as before. From Theorem \ref{thmtauselfdualstra}, Theorem \ref{thmtauselfdualtype}, Theorem \ref{thmendotauselfdualstra} and Remark \ref{remendothm}, there exist a simple stratum $[\mfa,\beta]$, a simple character $\theta\in\mcc(\mfa,\beta)$ attached to $\pi$ and a simple type $(\bs{J},\Lambda)$ containing $\theta$ and compactly inducing $\pi$ such that
	
	(1) $\tau_{0}(\mfa)=\mfa$ and $\tau_{0}(H^{1}(\mfa,\beta))=H^{1}(\mfa,\beta)$;
	
	(2) $\theta\circ\tau_{0}=\theta^{-1}$;
	
	(3) $\tau_{0}(\beta)=\beta^{-1}$;
	
	(4) $\tau_{0}(\bs{J})=\bs{J}$ and $\Lambda^{\tau_{0}}=\Lambda^{\vee}$;
	
	(5) Lemma \ref{lemmahassecase3a}, Lemma \ref{lemmahassecase2b2}, Corollary \ref{corhassecase3b}, Lemma \ref{lemmahassecase2b1}, Lemma \ref{lemmahassecase2a1},  Lemma \ref{lemmahassecase2a2} hold for $E=F[\beta]$.
	
	Here we assume $\tau_{0}=\tau_{\varepsilon_{0}}$, where $\varepsilon_{0}$ is a symmetric matrix in $G$ as follows:
	
	\textbf{Case (\romannumeral1)} If $\mrn_{T_{m}/F}(T_{m}^{\times})F^{\times2}/F^{\times2}=F^{\times}/F^{\times2}$, then $\varepsilon_{0}=J_{d,m}$.
	
	\textbf{Case (\romannumeral2)} If $\mrn_{T_{m}/F}(T_{m}^{\times})F^{\times2}/F^{\times2}$ is a subgroup of $F^{\times}/F^{\times2}$ of order two, we consider the following two cases:
	
	\textbf{(\romannumeral2.a)} If $2|m$, then $\varepsilon_{0}$ equals $J_{d,m}$ or $\mathrm{diag}(J_{d},...J_{d},J_{d}\epsilon_{0})$, where $\epsilon_{0}\in\mfo_{E}^{\times}\backslash\mfo_{E}^{\times2}$;
	
	\textbf{(\romannumeral2.b)} If $2\nmid m$, then $\varepsilon_{0}$ equals $J_{d,m}$ or $\mathrm{diag}(J_{d}\epsilon,...,J_{d}\epsilon)$, where $\epsilon$ is chosen to be either a uniformizer in $E$ or an element in $\mfo_{E}^{\times}\backslash\mfo_{E}^{\times2}$, such that $\mrn_{E/F}(\epsilon)\in\mrn_{T_{m}/F}(T_{m}^{\times})-F^{\times2}$.
	
	\textbf{Case (\romannumeral3)} If $\mrn_{T_{m}/F}(T_{m}^{\times})F^{\times2}/F^{\times2}=\{1\}$, we consider the following two cases:
	
	\textbf{(\romannumeral3.a)} If $\mrn_{T/F}(T^{\times})F^{\times2}/F^{\times2}=\{1\}$, then $\varepsilon_{0}$ equals $J_{d,m}$ or $\mathrm{diag}(J_{d}\varpi_{E},...,J_{d}\varpi_{E},J_{d}\varpi_{E}\epsilon_{0})$, where $\epsilon_{0}\in\mfo_{E}^{\times}\backslash\mfo_{E}^{\times2}$ and $\varpi_{E}$ is a uniformizer of $E$ chosen by Lemma \ref{lemmahassecase3a}, such that $\mathrm{Hasse}(J_{d}\varpi_{E})=1$ and $\mathrm{Hasse}(J_{d}\varpi_{E}\epsilon_{0})=-1$;
	
	\textbf{(\romannumeral3.b)} If $\mrn_{T/F}(T^{\times})F^{\times2}/F^{\times2}$ is not trivial, then $\varepsilon_{0}$ equals $J_{d,m}$ or $\mathrm{diag}(J_{d}\varpi_{E},...,J_{d}\varpi_{E},J_{d}\varpi_{E}\epsilon_{0})$ where $\epsilon_{0}\in\mfo_{E}^{\times}$ and $\varpi_{E}$ is a certain uniformizer of $E$, such that $\mathrm{Hasse}(\mathrm{diag}(J_{d}\varpi_{E},...,J_{d}\varpi_{E},J_{d}\varpi_{E}\epsilon_{0}))$ $=-1$.
	
	Thus in different cases, $G^{\tau_{0}}$ represents all possible $G$-conjugacy classes of orthogonal groups mentioned in Theorem \ref{thmendotauselfdualstra} respectively.
	
	From now on until the end of this section, we fix $\varepsilon_{0}$, $[\mfa,\beta]$, $\theta$ and $(\bs{J},\Lambda)$ as above. By (\ref{eqtauxtau'x}) if we restrict $\tau_{0}$ to $B^{\times}=\mrgl_{m}(E)$, it becomes an orthogonal involution $\tau_{\varepsilon_{0E}}$ with respect to $E$, where $\varepsilon_{0E}$ equals $I_{m}$ or $\mathrm{diag}(1,...,1,\epsilon_{0})$ with $\epsilon_{0}\in\mfo_{E}^{\times}\backslash\mfo_{E}^{\times2}$. We fix $\varepsilon$ a symmetric matrix in $G$ and $\tau=\tau_{\varepsilon}$ an orthogonal involution on $G$. We write $u=\varepsilon_{0}^{-1}\varepsilon$, then by direct calculation we get
	$$\tau(x)=u^{-1}\tau_{0}(x)u\quad \text{for any}\ x\in G$$
	and
	\begin{equation}\label{equtau0u=1}
		u\tau_{0}(u)=\varepsilon_{0}^{-1}\varepsilon\varepsilon_{0}^{-1}\,^{t}\varepsilon_{0}\,^{t}\varepsilon^{-1}\varepsilon_{0}=1. \end{equation}
	We write $\gamma=u\tau(g)g^{-1}$. We first state the following main theorem of this section:
	
	\begin{theorem}\label{thmdisttype}
		
		For $\pi$ a supercuspidal representation and $G^{\tau}$ an orthogonal group of $G$, the representation $\pi$ is distinguished by $G^{\tau}$ if and only if there exists a $\tau$-selfdual simple type $(\bs{J},\Lambda)$ of $\pi$ such that $\mrhom_{\bs{J}\cap G^{\tau}}(\Lambda,1)\neq 0$.
		
	\end{theorem}
	
	The ``if" part of this theorem is obvious, so we only need to proof the ``only if" part of this theorem. We assume $\pi$ to be distinguished by $G^{\tau}$ and we choose $(\bs{J},\Lambda)$ to be $\tau_{0}$-selfdual as above. By direct calculation, we get
	\begin{equation}\label{eqtauHtautheta}
		\tau(H^{1})=\tau_{0}(H^{1})^{u}=H^{1u}, \quad\theta^{\tau}\simeq(\theta^{\tau_{0}})^{u}\simeq(\theta^{-1})^{u}\quad \text{and} \quad \tau(\beta)=(\beta^{-1})^{u},
	\end{equation}
	and
	\begin{equation}\label{eqtauJtauLambda}
		\tau(\bs{J})=\tau_{0}(\bs{J})^{u}=\bs{J}^{u}\quad \text{and} \quad\Lambda^{\tau}\simeq(\Lambda^{\tau_{0}})^{u}\simeq\Lambda^{\vee u}.
	\end{equation}
	Using the Mackey formula and Frobenius reciprocity, we have
	
	$$0\neq\mrhom_{G^{\tau}}(\pi,1)\simeq\prod_{g\in\bs{J}\backslash G/G^{\tau}}\mrhom_{\bs{J}^{g}\cap G^{\tau}}(\Lambda^{g},1).$$
	The main step is to prove the following important theorem:
	
	\begin{theorem}\label{thmtaugginJ}
		
		For $g\in G$ such that $\mrhom_{\bs{J}^{g}\cap G^{\tau}}(\Lambda^{g},1)\neq 0$, we have
		$\gamma=u\tau(g)g^{-1}\in\bs{J}$.
		
	\end{theorem}
	
	Thus for the simple type $(\bs{J}^{g},\Lambda^{g})$, we get
	$$\tau(\bs{J}^{g})=\tau_{0}(\bs{J})^{u\tau(g)}=\bs{J}^{\gamma g}=\bs{J}^{g}\quad \text{and} \quad(\Lambda^{g})^{\tau}\simeq(\Lambda^{\tau_{0}})^{u\tau(g)}\simeq(\Lambda^{\vee})^{\gamma g}\simeq(\Lambda^{g})^{\vee},$$
	where we use the fact that $\gamma\in\bs{J}$ normalizes $\bs{J}$ and $\Lambda$. Thus $(\bs{J}^{g},\Lambda^{g})$ is what we want, which finishes the ``only if" part of Theorem \ref{thmdisttype}. So from now on, we focus on the proof of Theorem \ref{thmtaugginJ}.
	
	\subsection{Double cosets contributing to the distinction of $\theta$}
	
	In this subsection, we prove the following proposition:
	
	\begin{proposition}\label{propthetadisc}
		
		For $g\in G$, the character $\theta^{g}$ is trivial on $H^{1g}\cap G^{\tau}$ if and only if $\gamma\in JB^{\times}J$.
		
	\end{proposition}
	
	\begin{proof}
		
		We follow the proof of \cite{secherre2019supercuspidal}, Lemma 6.5. We choose $\tau$, $\chi$ and $H$ in \emph{loc. cit.} to be our $\tau$, $\theta$ and $H^{1}$ respectively. We use the assumptions $\tau(H^{1})=H^{1u}$ and $\theta\circ\tau=\theta^{-1u}$ to replace the original assumptions $\tau(H)=H$ and $\chi\circ\tau=\chi^{-1}$ respectively. And we use $\gamma=u\tau(g)g^{-1}$ to replace $\tau(g)g^{-1}$ in \emph{loc. cit.} Finally we notice that $\gamma$ intertwines $\theta$ if and only if $\gamma\in JB^{\times}J$. With the replacements and remarks mentioned above, the original proof can be used directly.
		
	\end{proof}
	
	As a result, for $g\in G$ such that $\mrhom_{\bs{J}^{g}\cap G^{\tau}}(\Lambda^{g},1)\neq 0$, restricting to $H^{1g}$ we get $\theta^{g}|_{H^{1g}\cap G^{\tau}}=1$, or equivalently $\gamma\in JB^{\times}J$.
	
	\subsection{The double coset lemma}
	
	In this section we prove the following double coset lemma:
	
	\begin{lemma}\label{lemmadoublecoset}
		
		Let $g\in G$ and let $\gamma=u\tau(g)g^{-1}\in JB^{\times}J$. Then changing $g$ with another representative in $Jg G^{\tau}$, we may assume $\gamma\in B^{\times}$.
		
	\end{lemma}
	
	\begin{remark}
		
		By direct calculation, we get
		\begin{equation}\label{eqgamma}
			\gamma=u\tau(g)g^{-1}=\varepsilon_{0}^{-1}\,^{t}g^{-1}\varepsilon g^{-1}=\tau_{0}(g)ug^{-1}, 
		\end{equation}
		and
		\begin{equation}\label{eqtau0gammagamma=1}
			\tau_{0}(\gamma)\gamma=g\tau_{0}(u)\tau_{0}(g)^{-1}\tau_{0}(g)ug^{-1}=g\tau_{0}(u)ug^{-1}=1.
		\end{equation}
		Since $\tau_{0}(J)=J$, if we change $g$ with a new representative of $JgG^{\tau}$, the new $\gamma$ belongs to the same $J$-$J$ double coset represented by the original $\gamma$, that is, the property $\gamma\in JB^{\times}J$ does not depend on the choice of $g$ in the $J$-$G^{\tau}$ double coset.
		
	\end{remark}
	
	\begin{proof}
		
		First of all, we need the following lemma:
		
		\begin{lemma}\label{lemmataub=b}
			
			There exists $b\in B^{\times}$ such that $\gamma\in JbJ$ and $\tau_{0}(b)b=1$.
			
		\end{lemma}
		
		\begin{proof}
			
			Since $\mfb^{\times}$ is a maximal order of $B^{\times}$, using the Cartan decomposition for $B^{\times}\simeq\mrgl_{m}(E)$, we may assume $\gamma=xcy$ such that $x,y\in J$ and \begin{equation}\label{eqcdiag}
				c=\mathrm{diag}(\varpi_{E}^{a_{1}}I_{m_{1}},...,\varpi_{E}^{a_{r}}I_{m_{r}}),
			\end{equation}
			where $a_{1}>...>a_{r}$ as integers and $m_{1}+...+m_{r}=m$. By definition of $\varepsilon_{0}$, the restriction of $\tau_{0}$ to $B^{\times}$ is also an orthogonal involution $\tau_{0}'$ defined by $$\tau_{0}'(z)=\varepsilon_{0E}^{-1}\,^{t_{E}}z^{-1}\varepsilon_{0E}\quad \text{for any}\ z\in B^{\times},$$
			where $\,^{t_{E}}$ represents the transpose on $\mrgl_{m}(E)$. If we write $b=c\varepsilon_{0E}$, then by definition we get
			$$\tau_{0}(b)b=\tau_{0}'(c\varepsilon_{0E})c\varepsilon_{0E}=\varepsilon_{0E}^{-1}\,^{t_{E}}c^{-1}\,^{t_{E}}\varepsilon_{0E}^{-1}\varepsilon_{0E}c\varepsilon_{0E}=\varepsilon_{0E}^{-1}\,^{t_{E}}c^{-1}c\varepsilon_{0E}=\varepsilon_{0E}^{-1}\varepsilon_{0E}=1.$$
			So the choice of $b$ satisfies our conditions.
			
		\end{proof}
		
		Now we write $\gamma=x'bx$ with $x,x'\in J$, $b=c\varepsilon_{0E}\in B^{\times}$ and $c$ as in (\ref{eqcdiag}). Replacing $g$ by $\tau_{0}(x')^{-1}g$ does not change the double coset $JgG^{\tau}$ but changes $\gamma$ into $bx\tau_{0}(x')$. So we may and will assume that $\gamma=bx$
		with $x\in J$.
		
		Write $K$ for the group $J\cap J^{b}$. Since $\tau_{0}(b)=b^{-1}$ and $\tau_{0}(J)=J$, using (\ref{eqtau0gammagamma=1}) we have $x\in J$ and $bxb^{-1}=\gamma b^{-1}=\tau_{0}(\gamma^{-1})\tau_{0}(b)=\tau_{0}(x^{-1})\in J$, thus $x\in K$. Moreover, we have the following corollary of Lemma \ref{lemmataub=b}.
		
		\begin{corollary}
			
			The map $\delta_{b}:k\mapsto b^{-1}\tau_{0}(k)b$ is an involution on $K$.
			
		\end{corollary}
		
		For $a_{1}>...>a_{r}$ and $m_{1}+...+m_{r}=m$ as in (\ref{eqcdiag}), and $M=\mathrm{GL}_{m_{1}d}(F)\times...\times\mathrm{GL}_{m_{r}d}(F)\subseteq G$, let $P$ be the standard parabolic subgroup of $G$ generated by $M$ and upper triangular matrices. Let $N$ and $N^{-}$ be the unipotent radicals of $P$ and its opposite parabolic subgroup with respect to $M$. By definition, $b$ normalizes $M$ and we have
		\begin{align*}
			K=(K\cap N ^{-})\cdot(K\cap M)\cdot(K\cap N).
		\end{align*}
		We have similar properties for the subgroup $V=K\cap B^{\times}=U(\mathfrak{b})\cap b^{-1}U(\mathfrak{b})b$ of $B^{\times}$:
		\begin{align*}
			V=(V\cap N ^{-})\cdot(V\cap M)\cdot(V\cap N).
		\end{align*}
	    By definition, $V$ is also fixed by $\delta_{b}$.
		
		\begin{lemma}\label{lemmaK1}
			
			The subset
			$$K^{1}=(K\cap N ^{-})\cdot(J^{1}\cap M)\cdot(K\cap N)$$
			is a $\delta_{b}$-stable normal pro-$p$-subgroup of $K$, and we have $K=VK^{1}$.
			
		\end{lemma}
		
		\begin{proof}
			
			The proof is the same as that in \cite{secherre2019supercuspidal}, Lemma 6.10.
			
		\end{proof}
		
		\begin{lemma}\label{lemmaredpgroup}
			
			For $x\in K$ satisfying $x\delta_{b}(x)=1$, there exist $k\in K$ and $v\in V$ such that
			
			(1) the element $v$ is in $\mathrm{GL}_{m_{1}}(\mathfrak{o}_{E})\times...\times\mathrm{GL}_{m_{r}}(\mathfrak{o}_{E})\subseteq B^{\times}$ satisfying $v\delta_{b}(v)=1$;
			
			(2) $\delta_{b}(k)xk^{-1}\in vK^{1}$.
			
		\end{lemma}
		
		\begin{proof}
			
			We may follow the same proof as \cite{zou2019supercuspidal}, Lemma 6.9, by replacing $\sigma$ and $^{*}$ in \emph{loc. cit.} with trivial map and $\,^{t}$. Noting that in instead of considering the three cases separately by using Lemma 6.10, Lemma 6.11 and Lemma 6.12 in \emph{loc. cit.}, there is only one case to consider in our lemma and we only need to use Lemma 6.11 in \emph{loc. cit.}, which we restate here for completeness.
			
			\begin{lemma}[\cite{kleidman1990subgroup}, Proposition 2.5.4]\label{lemmaklerdman}
			
			For $x$ a symmetric matrix in $\mrgl_{m}(\bs{l})$ with $\bs{l}$ a finite field, there exists $A\in \mrgl_{m}(\bs{l})$ such that $\,^t AxA=I_{m}$ or $\mathrm{diag}(1,...,1,\epsilon)$, where $\epsilon\in\bs{l}^{\times}-\bs{l}^{\times2}$ is fixed. In particular the $\mrgl_{m}(\bs{l})$-similar class of $x$ is determined by $\mrdet(x)\in\bs{l}^{\times}/\bs{l}^{\times2}$.
				
			\end{lemma}
			
		\end{proof}
		
		We finish the proof of Lemma \ref{lemmadoublecoset}. Applying Lemma \ref{lemmaredpgroup} gives us $k\in K$ and $v\in V$ such that $bv\tau_{0}(bv)=1$ and $\delta_{b}(k)xk^{-1}\in vK^{1}$. Thus we have $\tau_{0}(k)\gamma k^{-1}\in bv K^{1}$. Therefore, replacing $g$ by $kg$ and $b$ by $bv$, we may assume that $\gamma$ is written as
		\begin{equation}\label{eqgammabx}
			\gamma=bx,\quad b\tau_{0}(b)=1, \quad x\in K^{1},\quad b\in\varpi_{E}^{a_{1}}\mathrm{GL}_{m_{1}}(\mathfrak{o}_{E})\times...\times\varpi_{E}^{a_{r}}\mathrm{GL}_{m_{r}}(\mathfrak{o}_{E}).
		\end{equation}
		Furthermore, we have $\delta_{b}(x)x=1$.
		
		Since $K^{1}$ is a $\delta_{b}$-stable pro-$p$-group and $p$ is odd, the first cohomology set of $\delta_{b}$ in $K^{1}$ is trivial. Thus $x=\delta_{b}(y)y^{-1}$ for some $y\in K^{1}$, hence using (\ref{eqgamma}) we have $\gamma=\tau_{0}(g)ug^{-1}=\tau_{0}(y)by^{-1}$. As a result, if we further use $y^{-1}g$ to replace $g$, we get $\gamma=b\in B^{\times}$, which finishes the proof of Lemma \ref{lemmadoublecoset}.
		
	\end{proof}
	
	\begin{remark}
		
		Noting that in \cite{secherre2019supercuspidal} and \cite{zou2019supercuspidal}, the corresponding double coset lemma says that $\gamma\in JB^{\times}J$ if and only if $g\in JB^{\times}G^{\tau}$. However in our case if we assume $\varepsilon=\varepsilon_{0}$ and $\gamma=\tau(g)g^{-1}\in JB^{\times}J$, then it is possible that $g$ is not in $JB^{\times}G^{\tau}$. We will discuss this new phenomenon and calculate all the possible $J$-$G^{\tau}$ cosets in \S \ref{subsecdoublecosetdisc} .
		
	\end{remark}
	
	\subsection{Distinction of the Heisenberg representation}
	
	Let $\eta$ be the Heisenberg representation of $J^{1}$ associated to $\theta$, we have the following result as in \cite{secherre2019supercuspidal}, Proposition 6.12 and \cite{zou2019supercuspidal}, Proposition 6.13:
	
	\begin{proposition}\label{propheisdist}
		
		Given $g\in G$, we have
		
		$$\mathrm{dim}_{\mbc}\mrhom_{G^{\tau}}(\eta^{g},1)=\begin{cases}
			1 \quad &\text{if}\ \gamma=u\tau(g)g^{-1}\in JB^{\times}J,\\
			0 \quad &\text{otherwise.}
		\end{cases}$$
		
	\end{proposition}
	
	\begin{proof}
		
		First we restrict $\eta^{g}$ to $H^{1g}$ which is isomorphic to $\theta^{g(J^{1}:H^{1})^{1/2}}$. Using Proposition \ref{propthetadisc} when $\gamma\notin JB^{\times}J$, the dimension equals 0.
		
		When $\gamma\in JB^{\times}J$, by Lemma \ref{lemmadoublecoset} we may further assume $\gamma\in B^{\times}$. We write
		$$\delta(x):=(\tau(g)g^{-1})^{-1}\tau(x)\tau(g)g^{-1}\quad \text{for}\ x\in G$$
		as an involution on $G$, then by definition and (\ref{eqtau0gammagamma=1}) we have
		$$\mrhom_{G^{\tau}}(\eta^{g},1)\simeq\mrhom_{G^{\delta}}(\eta,1),$$
		and
		\begin{equation}\label{eqgammadeltagamma=1}
			\gamma\delta(\gamma)=\gamma\gamma^{-1}\tau_{0}(\gamma)\gamma=1.
		\end{equation}
		Moreover, using (\ref{eqtauHtautheta}) we have
		\begin{equation}\label{eqdeltaHdeltatheta}
			\delta(H^{1})=(\tau(g)g^{-1})^{-1}H^{1u}\tau(g)g^{-1}=H^{1\gamma}\quad\text{and}\quad\theta\circ\delta=(\theta^{-1})^{u\tau(g)g^{-1}}=(\theta^{-1})^{\gamma}.
		\end{equation}
		So using \cite{zou2019supercuspidal}, Proposition 6.14, we finish the proof.
		
	\end{proof}

	\subsection{Distinction of the extension of a Heisenberg representation}
	
	Let $\bs{\kappa}$ be an irreducible representation of $\bs{J}$ extending $\eta$, then there exists a unique representation $\bs{\rho}$ of $\bs{J}$ trivial on $J^{1}$ up to isomorphism, such that $\Lambda=\bs{\kappa}\otimes\bs{\rho}$. First of all we have the following proposition:
	
	\begin{proposition}\label{propkappa}
		
		Let $g\in G$ such that $\gamma\in JB^{\times}J$.
		
		(1) There is a unique character $\chi$ of $\boldsymbol{J}^{g}\cap G^{\tau}$ trivial on $J^{1g}\cap G^{\tau}$ such that
		$$\mathrm{Hom}_{J^{1g}\cap G^{\tau}}(\eta^{g},1)=\mathrm{Hom}_{\boldsymbol{J}^{g}\cap G^{\tau}}(\boldsymbol{\kappa}^{g},\chi^{-1}).$$
		
		(2) The canonical linear map
		$$\mathrm{Hom}_{J^{1g}\cap G^{\tau}}(\eta^{g},1)\otimes\mathrm{Hom}_{\boldsymbol{J}^{g}\cap G^{\tau}}(\boldsymbol{\rho}^{g},\chi)\rightarrow\mathrm{Hom}_{\boldsymbol{J}^{g}\cap G^{\tau}}(\Lambda^{g},1).$$
		is an isomorphism.
		
	\end{proposition}
	
	\begin{proof}
		
		With the aid of Proposition \ref{propheisdist}, the proof is the same as that in \cite{secherre2019supercuspidal}, Lemma 6.20.
		
	\end{proof}
	
	For $g\in G$ such that $\gamma=u\tau(g)g^{-1}=\tau_{0}(g)ug^{-1}\in JB^{\times}J$, using $u\tau(g)=\tau_{0}(g)u$ to replace $g$, we have
	$$\tau_{0}(\tau_{0}(g)u)u(\tau_{0}(g)u)^{-1}=gu^{-1}\tau_{0}(g)^{-1}=(\tau_{0}(g)ug^{-1})^{-1}\in JB^{\times}J,$$
	which means that we may consider $u\tau(g)$ instead of $g$ in Proposition \ref{propkappa}. Thus there exists a unique character $\chi'$ of $\boldsymbol{J}^{u\tau(g)}\cap G^{\tau}$ trivial on $J^{1u\tau(g)}\cap G^{\tau}$ such that
	$$\mathrm{Hom}_{J^{1u\tau(g)}\cap G^{\tau}}(\eta^{u\tau(g)},1)\simeq\mathrm{Hom}_{\boldsymbol{J}^{u\tau(g)}\cap G^{\tau}}(\boldsymbol{\kappa}^{u\tau(g)},\chi'^{-1}).$$
	Moreover, we know that $\tau(\boldsymbol{J})=\boldsymbol{J}^{u}$, $\tau(J)=J^{u}$, $\tau(J^{1})=J^{1u}$ and $\tau(H^{1})=H^{1u}$, thus as in \cite{zou2019supercuspidal}, Lemma 4.2 and Lemma 6.15, it is easy to show that \begin{equation}\label{eqjtau=j0tau}
		\boldsymbol{J}^{g}\cap G^{\tau}=\boldsymbol{J}^{u\tau(g)}\cap G^{\tau}=J^{g}\cap G^{\tau}=J^{u\tau(g)}\cap G^{\tau}
	\end{equation}
	As a result, $\chi$ and $\chi'$ are characters defined on the same group $\boldsymbol{J}^{g}\cap G^{\tau}=\boldsymbol{J}^{u\tau(g)}\cap G^{\tau}$. 
	
	\begin{proposition}\label{propchichi'}
		
		For $\chi$ and $\chi'$ characters of $\boldsymbol{J}^{g}\cap G^{\tau}=\boldsymbol{J}^{u\tau(g)}\cap G^{\tau}$ defined above, we have $\chi=\chi'$.
		
	\end{proposition}
	
	\begin{proof}	
		
		We write $\delta(x)=(\tau(g)g^{-1})^{-1}\tau(x)\tau(g)g^{-1}$ for any $x\in G$. Using the basic results in the simple type theory (see \cite{zou2019supercuspidal}, \S 3.2 for example), we have $\gamma=u\tau(g)g^{-1}\in I_{G}(\eta)=I_{G}(\kappa^{0})$, where $\kappa^{0}=\boldsymbol{\kappa}|_{J}$ and $I_{G}(\eta)$ (resp. $I_{G}(\kappa^{0})$) denotes the intertwining set of $\eta$ (resp. $\kappa^{0}$). Moreover we have
		$$\mathrm{dim}_{\mbc}(\mathrm{Hom}_{J\cap J^{\gamma}}(\kappa^{0\gamma},\kappa^{0}))=\mathrm{dim}_{\mbc}(\mathrm{Hom}_{J^{1}\cap J^{1\gamma}}(\eta^{\gamma},\eta))=1.$$
		By direct calculation, we have $J^{1}\cap G^{\delta}=J^{1\gamma}\cap G^{\delta}$ as a subgroup of $J^{1}\cap J^{1\gamma}$ and $H^{1}\cap G^{\delta}=H^{1\gamma}\cap G^{\delta}$.
		Using \cite{zou2019supercuspidal}, Proposition 6.20 for our $\gamma$ and $\delta$, we have:
		
		\begin{proposition}\label{propvarphibijec}
			
			For a non-zero homomorphism $\varphi\in\mathrm{Hom}_{J^{1}\cap J^{1\gamma}}(\eta^{\gamma},\eta)=\mathrm{Hom}_{J\cap J^{\gamma}}(\kappa^{0\gamma},\kappa^{0})$, it naturally induces a $\mbc$-vector space isomorphism
			\begin{align*}
				f_{\varphi}:\mathrm{Hom}_{J^{1}\cap G^{\delta}}(\eta,1)&\rightarrow\mathrm{Hom}_{J^{1\gamma}\cap G^{\delta}}(\eta^{\gamma},1),\\
				\lambda\quad&\mapsto\quad \lambda\circ\varphi.
			\end{align*}
			
		\end{proposition}
		
		Now we use Proposition \ref{propvarphibijec} to finish the proof of Proposition \ref{propchichi'}. Using Proposition \ref{propheisdist} for $g$ and $u\tau(g)$ respectively, we have $$\mathrm{dim}_{\mbc}\mathrm{Hom}_{J^{1g}\cap G^{\tau}}(\eta^{g},1)=\mathrm{dim}_{\mbc}\mathrm{Hom}_{J^{1u\tau(g)}\cap G^{\tau}}(\eta^{u\tau(g)},1)=1.$$ By Proposition \ref{propvarphibijec}, for $0\neq\varphi\in\mathrm{Hom}_{J^{1}\cap J^{1\gamma}}(\eta^{\gamma},\eta)=\mathrm{Hom}_{J^{1g}\cap J^{1u\tau(g)}}(\eta^{g},\eta^{u\tau(g)})$,
		\begin{align*}
			f_{\varphi}:\mathrm{Hom}_{J^{1g}\cap G^{\tau}}(\eta^{g},1)&\rightarrow\mathrm{Hom}_{J^{1u\tau(g)}\cap G^{\tau}}(\eta^{u\tau(g)},1),\\
			\lambda\quad&\mapsto\quad \lambda\circ\varphi,
		\end{align*}
		is bijective. 
		If we choose
		$$0\neq \lambda\in\mathrm{Hom}_{J^{1g}\cap G^{\tau}}(\eta^{g},1)\quad \text{and}\quad 0\neq \lambda':=f_{\varphi}(\lambda)=\lambda\circ\varphi\in\mathrm{Hom}_{J^{1u\tau(g)}\cap G^{\tau}}(\eta^{u\tau(g)},1),$$
		then for any $v$ in the representation space of $\eta$ and any $x\in J^{g}\cap G^{\tau}=J^{u\tau(g)}\cap G^{\tau}$, we have
		\begin{align*}
			\chi'(x)^{-1}\lambda'(v)&=\lambda'(\kappa^{0u\tau(g)}(x)v)\qquad(\text{by Proposition \ref{propkappa}.(1)})\\
			&=\lambda(\varphi(\kappa^{0u\tau(g)}(x)v))\qquad(\text{by definition of}\ \lambda')\\
			&=\lambda(\kappa^{0g}(x)\varphi(v))\qquad(\text{since}\ \varphi\in\mathrm{Hom}_{J^{g}\cap J^{u\tau(g)}}(\kappa^{0u\tau(g)},\kappa^{0g}))\\
			&=\chi(x)^{-1}\lambda(\varphi(v))\qquad(\text{by Proposition \ref{propkappa}.(1)})\\
			&=\chi(x)^{-1}\lambda'(v)\qquad(\text{by definition of}\ \lambda').
		\end{align*}
		Since $v$ and $x\in J^{g}\cap G^{\tau}=J^{u\tau(g)}\cap G^{\tau}$ are arbitrary, we have $\chi'|_{J^{u\tau(g)}\cap G^{\tau}}=\chi|_{J^{g}\cap G^{\tau}}$, which finishes the proof with the aid of (\ref{eqjtau=j0tau}).
		
		\subsection{Existence of a $\tau$-selfdual extension of $\eta$}
		
		\begin{proposition}\label{propkappaselfdual}
			
			There is $\boldsymbol{\kappa}$ as an extension of $\eta$ such that $\boldsymbol{\kappa}^{\tau_{0}\vee}\simeq\boldsymbol{\kappa}$.
			
		\end{proposition}
		
		\begin{proof}
			
			We refer to \cite{zou2019supercuspidal}, \S 6.5, especially Proposition 6.24 for a proof. Noting that the restriction of $\tau_{0}$ to $\mrgl_{m}(\boldsymbol{l})$ becomes an orthogonal involution with respect to the symmetric matrix $\overline{\varepsilon_{0E}}\in \mrgl_{m}(\boldsymbol{l})$, where $\overline{\varepsilon_{0E}}$ represents the image of $\varepsilon_{0E}$ in $\mrgl_{m}(\boldsymbol{l})\simeq\mrgl_{m}(\mfo_{E})/(1+\mrm_{m}(\mfp_{E}))$, thus if we replace $\sigma$ and $\tau$ in the \emph{loc. cit.} by the trivial action and $\tau_{0}$, then the same proof in the case where $E/E_{0}$ is ramified in \emph{loc. cit.}
			works for our proposition.
			
		\end{proof}
		
		From now on until the end of this section we fix $\bs{\kappa}$ as in Proposition \ref{propkappaselfdual}. We have the following corollary:
		
		\begin{corollary}\label{corchiquad}
			
			The character $\chi$ defined by Lemma \ref{propkappa}.(1) is quadratic, that is, $\chi^{2}=1$.
			
		\end{corollary}
		
		\begin{proof}
			
			We have the following isomorphisms
			\begin{align*}
				\mathrm{Hom}_{J^{1u\tau(g)}\cap G^{\tau}}(\eta^{u\tau(g)},1)&\simeq\mathrm{Hom}_{J^{1g}\cap G^{\tau}}(\eta^{g},1)\\
				&\simeq\mathrm{Hom}_{\boldsymbol{J}^{g}\cap  G^{\tau}}(\boldsymbol{\kappa}^{g},\chi^{-1})\\
				&\simeq\mathrm{Hom}_{\boldsymbol{J}^{g}\cap G^{\tau}}(\chi,\boldsymbol{\kappa}^{g\vee})\quad (\text{by the duality of contragredient})\\
				&\simeq\mathrm{Hom}_{\boldsymbol{J}^{g}\cap G^{\tau}}(\boldsymbol{\kappa}^{g\vee},\chi)\\
				&\simeq\mathrm{Hom}_{\boldsymbol{J}^{g}\cap G^{\tau}}(\boldsymbol{\kappa}^{g\vee}\circ\tau,\chi\circ\tau)\\
				&\simeq\mathrm{Hom}_{\boldsymbol{J}^{g}\cap G^{\tau}}((\boldsymbol{\kappa}^{\tau_{0}\vee})^{u\tau(g)},\chi\circ\tau)\\
				&\simeq\mathrm{Hom}_{\boldsymbol{J}^{u\tau(g)}\cap G^{\tau}}(\boldsymbol{\kappa}^{u\tau(g)},\chi\circ\tau)\quad(\text{since}\ \boldsymbol{\kappa}\ \text{is}\  \tau_{0}\text{-selfdual}).
			\end{align*}
			Using Proposition \ref{propchichi'} and the uniqueness of $\chi'$, we have $\chi\circ\tau=\chi^{-1}$. Since $\chi$ is defined on $\boldsymbol{J}^{g}\cap G^{\tau}=J^{g}\cap G^{\tau}$ which is $\tau$-invariant, we have $\chi\circ\tau=\chi$, thus $\chi^{2}=\chi(\chi\circ\tau)=1$.
			
		\end{proof}
		
		\subsection{Proof of Theorem \ref{thmtaugginJ}}\label{subsectionpfdisttype}
		
		In this subsection, we finish the proof of Theorem \ref{thmtaugginJ}. For $g\in G$ given as in \emph{loc. cit.}, by Lemma \ref{lemmadoublecoset} and the Cartan decomposition, we may replace $g$ by another representative in the same $J$-$G^{\tau}$ double coset, such that
		\begin{equation}\label{eqgammabxcond}
			\gamma:=u\tau(g)g^{-1}\in\varpi_{E}^{a_{1}}\mathrm{GL}_{m_{1}}(\mathfrak{o}_{E})\times...\times\varpi_{E}^{a_{r}}\mathrm{GL}_{m_{r}}(\mathfrak{o}_{E}),
		\end{equation}
		where $a_{i}$, $m_{i}$ are defined as in Lemma \ref{lemmataub=b}. Thus
		there exists a unique standard hereditary order $\mathfrak{b}_{m}\subseteq\mathfrak{b}$ such that
		$$U^{1}(\mathfrak{b}_{m})=(U\cap\delta(U^{1}))U^{1}=(U\cap U^{1\gamma})U^{1},$$
		where we define $U=U(\mfb)$, $U^{1}=U^{1}(\mfb)$ and  $\delta(x)=(\tau(g)g^{-1})^{-1}\tau(x)\tau(g)g^{-1}$ for any $x\in G$ as an involution on $G$.
		First we have the following lemma whose proof is the same as that in \cite{secherre2019supercuspidal}, Lemma 6.22, inspired by \cite{hakim2008distinguished}, Proposition 5.20:
		
		\begin{lemma}\label{lemU1bm}
			
			We have $U^{1}(\mathfrak{b}_{m})=(U^{1}(\mathfrak{b}_{m})\cap G^{\delta})U^{1}$.
			
		\end{lemma}
		
		To finish the proof, it is enough to show that $r=1$ in (\ref{eqgammabxcond}). If not, we know that $\mathfrak{b}_{m}$ by definition is a proper suborder of $\mathfrak{b}$. Furthermore, $\overline{U^{1}(\mathfrak{b}_{m})}:=U^{1}(\mathfrak{b}_{m})/U^{1}$ is a non-trivial unipotent subgroup of $U/U^{1}\simeq\mathrm{GL}_{m}(\boldsymbol{l})$. Using Proposition \ref{propkappa}.(2), we have
		$$\mathrm{Hom}_{\boldsymbol{J}\cap G^{\delta}}(\boldsymbol{\rho},\chi^{g^{-1}})\simeq\mathrm{Hom}_{\boldsymbol{J}^{g}\cap G^{\tau}}(\boldsymbol{\rho}^{g},\chi)\neq 0.$$
		Restricting to $U^{1}(\mathfrak{b}_{m})\cap G^{\delta}$, we have
		\begin{equation}\label{eqrho}
			\mathrm{Hom}_{U^{1}(\mathfrak{b}_{m})\cap G^{\delta}}(\boldsymbol{\rho},\chi^{g^{-1}})\neq 0.
		\end{equation}
		Using Lemma \ref{lemU1bm}, we have the isomorphism
		$$(U^{1}(\mathfrak{b}_{m})\cap G^{\delta})U^{1}/U^{1}\simeq U^{1}(\mathfrak{b}_{m})/U^{1}.$$
		We denote by $\overline{\rho}$ the cuspidal representation of $U^{0}/U^{1}\simeq\mrgl_{m}(\bs{l})$ whose inflation is $\boldsymbol{\rho}|_{U^{0}}$, and $\overline{\chi^{g^{-1}}}$ the character of $\overline{U^{1}(\mfb_{m})}$ whose inflation is $\chi^{g^{-1}}$.
		We consider the equation (\ref{eqrho}) modulo $U^{1}$ and we have
		$$\mathrm{Hom}_{\overline{U^{1}(\mathfrak{b}_{m})}}(\overline{\rho},\overline{\chi^{g^{-1}}})\neq 0.$$
		Since $\chi^{g^{-1}}|_{J\cap G^{\delta}}$ is quadratic and $\overline{U^{1}(\mathfrak{b}_{m})}$ is a $p$-group with $p\neq 2$, we get $\overline{\chi^{g^{-1}}}|_{\overline{U^{1}(\mfb_{m})}}=1$, thus
		$$\mathrm{Hom}_{\overline{U^{1}(\mathfrak{b}_{m})}}(\overline{\rho},1)\neq 0$$
		which contradicts to the fact that $\overline{\rho}$ is supercuspidal. So we finish the proof.
		
	\end{proof}
	
	\subsection{Double cosets contributing to the distinction of $\pi$}\label{subsecdoublecosetdisc}
	
	In this subsection, we assume $\varepsilon=\varepsilon_{0}$ and $\tau=\tau_{0}$. We want to study all the possible $\bs{J}$-$G^{\tau}$ double cosets contributing to the distinction of $\pi$. Precisely, we want to study those $g\in G$ such that
	$$\mrhom_{\bs{J}^{g}\cap G^{\tau}}(\Lambda^{g},1)\neq 0.$$
	By Lemma \ref{lemmadoublecoset}, we may change $g$ with another representative in $JgG^{\tau}$ to assume that $\gamma=\tau(g)g^{-1}\in B^{\times}$. Moreover, by Theorem \ref{thmtaugginJ} we get $\gamma\in\bs{J}$. As a result, we have
	\begin{equation}\label{eqgammainEb}
		\gamma\in \bs{J}\cap B^{\times}=E^{\times}\mfb^{\times}.
	\end{equation}
	First by changing $g$ up to multiplying an element in $E^{\times}$ on the left, which does not change the double coset $\bs{J}gG^{\tau}$, we may assume $\gamma\in\mfb^{\times}$ or $\varpi_{E}\mfb^{\times}$. 
	Since $J\cap B^{\times}=\mfb^{\times}=\mrgl_{m}(\mfo_{E})$, using Proposition \ref{propSGLoEorbit} we may change $g$ up to multiplying an element in $\mfb^{\times}$ on the left, which does not change the double coset $\bs{J}gG^{\tau}$, such that
	\begin{equation}\label{eqgamma4value}
		\gamma= I_{m}\ \text{or}\ \mathrm{diag}(1,...,1,\epsilon_{0})\ \text{or}\ \mathrm{diag}(\varpi_{E},...,\varpi_{E},\varpi_{E})\ \text{or}\ \mathrm{diag}(\varpi_{E},...,\varpi_{E},\varpi_{E}\epsilon_{0}).
	\end{equation}
	By definition, we have
	\begin{equation}\label{eqnorndetgamma}
		\mrn_{E/F}(\mrdet_{B}(\gamma))=\mrdet(\gamma)\in F^{\times2},
	\end{equation}
	where $\mrdet_{B}$ denotes the determinant on $B^{\times}=\mrgl_{m}(E)$. By studying different cases separately, we will give out all the possible double cosets of $g$ satisfying the condition (\ref{eqgamma4value}).
	
	\textbf{Case (\romannumeral1)} If $\mrn_{T_{m}/F}(T_{m}^{\times})F^{\times2}/F^{\times2}=F^{\times}/F^{\times2}$, then
	$$\mrn_{E/F}:E^{\times}/E^{\times2}\longrightarrow F^{\times}/F^{\times 2}$$
	is bijective. Thus (\ref{eqnorndetgamma}) shows that $\mrdet_{B}(\gamma)\equiv 1$ (mod $E^{\times2}$). Thus from (\ref{eqgamma4value}) and the fact that $m$ is odd, we get $\gamma=1$, which means that $g\in G^{\tau}$. Thus in this case there is only one double coset $\bs{J}G^{\tau}$.
	
	\textbf{Case (\romannumeral2)} If $\mrn_{T_{m}/F}(T_{m}^{\times})F^{\times2}/F^{\times2}$ is a subgroup of $F^{\times}/F^{\times2}$ of order two, we consider the following two cases:
	
	\textbf{(\romannumeral2.a)} If
	$2|m$, then from the same argument in \S \ref{subsectiongeneral} we have $\mrn_{E_{m}/E}(E_{m}^{\times})E^{\times2}/E^{\times2}=\{1,\epsilon_{0}\}$, where $\epsilon_{0}\in\mfo_{E}^{\times}\backslash\mfo_{E}^{\times2}$ as above. And moreover the ramification index of $E/F$ is odd and $\mrn_{E/F}(\epsilon_{0})\notin F^{\times2}$. Using (\ref{eqgamma4value}) and (\ref{eqnorndetgamma}), $\gamma$ equals $I_{m}$ or $\mathrm{diag}(\varpi_{E},...,\varpi_{E})$.
	
	\textbf{(\romannumeral2.a.1)} We assume one of the three cases is true:
	\begin{itemize}
		\item$2|d$;
		\item$2\nmid d$ and $4|m$;
		\item$2\nmid d$, $4\nmid m$ and $-1\in F^{\times2}$.
	\end{itemize}
	If $\varepsilon=J_{d,m}$ and $\varepsilon_{0E}=I_{m}$, then in the case where $\gamma=\tau(g)g^{-1}=I_{m}$, we have $g\in G^{\tau}$. In the case where $\gamma=\mathrm{diag}(\varpi_{E},...,\varpi_{E})$, using Proposition \ref{propSGLEorbit} and the fact that $$\mathrm{det}_{B}(\mathrm{diag}(\varpi_{E},...,\varpi_{E}))=\varpi_{E}^{m}\in E^{\times2}\quad \text{and}\quad \mathrm{Hasse}_{E}(\mathrm{diag}(\varpi_{E},...,\varpi_{E}))=1,$$ there exists $g_{1}\in B^{\times}$ such that $\tau(g_{1})g_{1}^{-1}=\mathrm{diag}(\varpi_{E},...,\varpi_{E})$, where we denote by $\mathrm{Hasse}_{E}$ the Hasse invariant for the symmetric matrices in $B^{\times}=\mrgl_{m}(E)$ and we use Lemma \ref{lemmahilbertresidue} to calculate the Hasse invariant. Thus we have $g\in \bs{J}g_{1}G^{\tau}$. So there are two possible double cosets $\bs{J}G^{\tau}$ and $\bs{J}g_{1}G^{\tau}$.
	
	If $\varepsilon=\mathrm{diag}(J_{d},...,J_{d},J_{d}\epsilon_{0})$ and $\varepsilon_{0E}=\mathrm{diag}(1,...,1,\epsilon_{0})$ with $\epsilon_{0}\in\mfo_{E}^{\times}\backslash\mfo_{E}^{\times2}$, then in the case where $\gamma=I_{m}$, we have $g\in G^{\tau}$. In the case where $\gamma=\mathrm{diag}(\varpi_{E},...,\varpi_{E})$, by direct calculation we get
	$$\,^{t}g^{-1}\mathrm{diag}(J_{d},...,J_{d},J_{d}\epsilon_{0})g^{-1}=\mathrm{diag}(J_{d}\varpi_{E},...,J_{d}\varpi_{E},J_{d}\varpi_{E}\epsilon_{0}).$$
	Using Lemma \ref{lemmaGLnOFsymhasse} we obtain $\mathrm{Hasse}(\,^{t}g^{-1}\mathrm{diag}(J_{d},...,J_{d},J_{d}\epsilon_{0})g^{-1})=1$. However by Lemma \ref{lemmahassecase2a1} and Corollary \ref{corcalhasse}, we have $\mathrm{Hasse}(\mathrm{diag}(J_{d}\varpi_{E},...,J_{d}\varpi_{E},J_{d}\varpi_{E}\epsilon_{0}))=-1$, thus there does not exist any $g\in G$ such that $\gamma=\mathrm{diag}(\varpi_{E},...,\varpi_{E},\varpi_{E})$, so there is only one possible double coset $\bs{J}G^{\tau}$.
	
	\textbf{(\romannumeral2.a.2)} If $2\nmid d$, $4\nmid m$ and $-1\notin F^{\times2}$, then we may choose $\epsilon_{0}=-1\in\mfo_{E}^{\times}\backslash\mfo_{E}^{\times2}$.
	
	If $\varepsilon=\mathrm{diag}(J_{d},...,J_{d},-J_{d})$ and $\varepsilon_{0E}=\mathrm{diag}(1,...,1,-1)$, then in the case where $\gamma=\tau(g)g^{-1}=I_{m}$, we have $g\in G^{\tau}$. In the case where $\gamma=\mathrm{diag}(\varpi_{E},...,\varpi_{E})$, using Proposition \ref{propSGLEorbit} and the fact that (by Lemma \ref{lemmahilbertresidue} for example) $$\mathrm{det}_{B}(\mathrm{diag}(\varpi_{E},...,-\varpi_{E}))=-\varpi_{E}^{m}\in \epsilon_{0}E^{\times2}\quad \text{and}\quad \mathrm{Hasse}_{E}(\mathrm{diag}(\varpi_{E},...,-\varpi_{E}))=1,$$ there exists $g_{1}\in B^{\times}$ such that $\,^{t}g_{1}^{-1}\varepsilon_{0E}g_{1}^{-1}=\mathrm{diag}(\varpi_{E},...,-\varpi_{E})$, or equivalently $\tau(g_{1})g_{1}^{-1}=\mathrm{diag}(\varpi_{E},...,\varpi_{E},\varpi_{E})$. Thus we have $g\in \bs{J}g_{1}G^{\tau}$. So there are two possible double cosets $\bs{J}G^{\tau}$ and $\bs{J}g_{1}G^{\tau}$.
	
	If $\varepsilon=J_{d,m}$ and $\varepsilon_{0E}=I_{m}$, then in the case where $\gamma=I_{m}$, we have $g\in G^{\tau}$. In the case where $\gamma=\mathrm{diag}(\varpi_{E},...,\varpi_{E})$, by direct calculation we get
	$$\,^{t}g^{-1}\mathrm{diag}(J_{d},...,J_{d},J_{d})g^{-1}=\mathrm{diag}(J_{d}\varpi_{E},...,J_{d}\varpi_{E},J_{d}\varpi_{E}).$$
	Using Lemma \ref{lemmaGLnOFsymhasse} we get $\mathrm{Hasse}(\,^{t}g^{-1}\mathrm{diag}(J_{d},...,J_{d},J_{d})g^{-1})=1$. Using Lemma \ref{lemmahassecase2a2} and Corollary \ref{corcalhasse} we have $\mathrm{Hasse}(\mathrm{diag}(J_{d}\varpi_{E},...,J_{d}\varpi_{E},J_{d}\varpi_{E}))=-1$, thus there does not exist any $g$ as above such that $\gamma=\mathrm{diag}(\varpi_{E},...,\varpi_{E},\varpi_{E})$, so there is only one possible double coset $\bs{J}G^{\tau}$.
	
	\textbf{(\romannumeral2.b)} If $2\nmid m$, then $\varepsilon$ equals $J_{d,m}$ or $\mathrm{diag}(J_{d}\epsilon,...,J_{d}\epsilon)$, where $\epsilon\in E^{\times}$. In this case we have $\mrn_{E_{m}/F}(E_{m}^{\times})F^{\times2}/F^{\times2}=\mrn_{E/F}(E^{\times})F^{\times2}/F^{\times2}$
	and $2|d$. Furthermore by Proposition \ref{propE/Fbasic}, either the ramification index or the inertia degree of $E/F$ is odd. We further consider the following two cases:
	
	\textbf{(\romannumeral2.b.1)} If the ramification index of $E/F$ is odd, then $\epsilon=\epsilon_{0}\in\mfo_{E}^{\times}\backslash\mfo_{E}^{\times2}$ such that $\mrn_{E/F}(\epsilon_{0})\notin F^{\times2}$. By (\ref{eqgamma4value}) and (\ref{eqnorndetgamma}), we deduce that $\gamma$ equals $I_{m}$ or $\mathrm{diag}(\varpi_{E},...,\varpi_{E},\varpi_{E}')$, where $\varpi_{E}'$ equals $\varpi_{E}$ or $\varpi_{E}\epsilon_{0}$ such that $\mrn_{E/F}(\varpi_{E}')\in F^{\times 2}$.
	
	If $\varepsilon=J_{d,m}$, we have $g\in G^{\tau}$ when $\gamma=I_{m}$. When $\gamma=\mathrm{diag}(\varpi_{E},...,\varpi_{E},\varpi_{E}')$, using Lemma \ref{lemmaGLnOFsymhasse}, Lemma \ref{lemmahassecase2b1} and Corollary \ref{corcalhasse}, we have
	$$\mathrm{det}(J_{d,m}\mathrm{diag}(\varpi_{E},...,\varpi_{E},\varpi_{E}'))\in \mathrm{det}(J_{d,m})F^{\times2}\quad \text{and}\quad \mathrm{Hasse}(J_{d,m}\mathrm{diag}(\varpi_{E},...,\varpi_{E},\varpi_{E}'))=1,$$
	thus by Proposition \ref{propSGLEorbit}, there exists $g_{1}\in G$ such that $$\,^{t}g_{1}^{-1}J_{d,m}g_{1}^{-1}=J_{d,m}\mathrm{diag}(\varpi_{E},...,\varpi_{E},\varpi_{E}')=J_{d,m}\gamma,$$ or in other words $\tau(g_{1})g_{1}^{-1}=\gamma$. Thus $g\in g_{1}G^{\tau}$. So we get two double cosets $\bs{J}G^{\tau}$ and $\bs{J}g_{1}G^{\tau}$.
	
	\begin{remark}
		
		Since $\mathrm{det}_{B}(\mathrm{diag}(\varpi_{E},...,\varpi_{E},\varpi_{E}'))=\varpi_{E}^{m-1}\varpi_{E}'\notin E^{\times2}$, it is impossible to choose $g_{1}\in B^{\times}$ such that $\tau(g_{1})g_{1}=\gamma$. Thus $\bs{J}g_{1}G^{\tau}$ is disjoint with $\bs{J}B^{\times}G^{\tau}$. Similar phenomena also occur in \textbf{(\romannumeral2.b.2)} and \textbf{(\romannumeral3)} below.
		
	\end{remark}
	
	If $\varepsilon=\mathrm{diag}(J_{d}\epsilon_{0},...,J_{d}\epsilon_{0})$, we get $g\in G^{\tau}$ in the case where $\gamma=I_{m}$. In the case where $\gamma=\mathrm{diag}(\varpi_{E},...,\varpi_{E},\varpi_{E}')$, by direct calculation we have
	\begin{equation}\label{eq2b1cond}
		\,^{t}g^{-1}\mathrm{diag}(J_{d}\epsilon_{0},...J_{d}\epsilon_{0},J_{d}\epsilon_{0})g^{-1}=\mathrm{diag}(J_{d}\varpi_{E}\epsilon_{0},...J_{d}\varpi_{E}\epsilon_{0},J_{d}\varpi_{E}'\epsilon_{0}).
	\end{equation}
	
	By Lemma \ref{lemmaGLnOFsymhasse}, we get
	$$\mathrm{Hasse}(\,^{t}g^{-1}\mathrm{diag}(J_{d}\epsilon_{0},...J_{d}\epsilon_{0},J_{d}\epsilon_{0})g^{-1})=1.$$
	And by Lemma \ref{lemmahassecase2b1} and Corollary \ref{corcalhasse}, we obtain $\mathrm{Hasse}(\mathrm{diag}(J_{d}\varpi_{E}\epsilon_{0},...J_{d}\varpi_{E}\epsilon_{0},J_{d}\varpi_{E}'\epsilon_{0}))=-1$, thus the condition (\ref{eq2b1cond}) is never satisfied. Thus there is only one possible double coset $\bs{J}G^{\tau}$.
	
	\textbf{(\romannumeral2.b.2)} If the inertia degree of $E/F$ is odd, then $\epsilon=\varpi_{E}$ as a uniformizer of $E$ such that $\mrn_{E/F}(\varpi_{E})\notin F^{\times2}$. By (\ref{eqgamma4value}) and (\ref{eqnorndetgamma}), we get $\mrn_{E/F}(\mrdet_{B}(\gamma))\in F^{\times2}$, thus $\mrdet_{B}(\gamma)$ equals $1$ or $\epsilon_{0}$, which means that $\gamma$ equals $I_{m}$ or $\mathrm{diag}(1,...,1,\epsilon_{0})$.
	
	If $\varepsilon=J_{d,m}$, we have $g\in G^{\tau}$ in the case where $\gamma=I_{m}$. In the case where $\gamma=\mathrm{diag}(1,...,1,\epsilon_{0})$, using Lemma \ref{lemmaGLnOFsymhasse} we have
	$$\mathrm{det}(J_{d,m}\mathrm{diag}(1,...,1,\epsilon_{0}))\in \mathrm{det}(J_{d,m})F^{\times2}\quad \text{and}\quad \mathrm{Hasse}(J_{d,m}\mathrm{diag}(1,...,1,\epsilon_{0}))=\mathrm{Hasse}(J_{d,m})=1,$$
	thus by Proposition \ref{propSGLEorbit}, there exists $g_{1}\in G$ such that $$\,^{t}g_{1}^{-1}J_{d,m}g_{1}^{-1}=J_{d,m}\mathrm{diag}(1,...,1,\epsilon_{0}),$$ or equivalently $\tau(g_{1})g_{1}^{-1}=\gamma$. Thus $g\in g_{1}G^{\tau}$. So we get two double cosets $\bs{J}G^{\tau}$ and $\bs{J}g_{1}G^{\tau}$.
	
	If $\varepsilon=\mathrm{diag}(J_{d}\varpi_{E},...,J_{d}\varpi_{E})$, we get $g\in G^{\tau}$ in the case where $\gamma=I_{m}$. In the case where $\gamma=\mathrm{diag}(1,...,1,\epsilon_{0})$, by direct calculation we have
	\begin{equation}\label{eq2b2cond}
		\,^{t}g^{-1}\mathrm{diag}(J_{d}\varpi_{E},...J_{d}\varpi_{E},J_{d}\varpi_{E})g^{-1}=\mathrm{diag}(J_{d}\varpi_{E},...J_{d}\varpi_{E},J_{d}\varpi_{E}\epsilon_{0}),
	\end{equation}
	However by Corollary $\ref{corhassecase3b}$ and Corollary \ref{corcalhasse}, this condition is never satisfied. Thus there is only one possible double coset $\bs{J}G^{\tau}$.
	
	\textbf{Case (\romannumeral3)} If $\mrn_{T_{m}/F}(T_{m}^{\times})F^{\times2}/F^{\times2}=\{1\}$, we consider the following two cases:
	
	\textbf{(\romannumeral3.a)} If $\mrn_{E/F}(E^{\times})F^{\times2}/F^{\times2}=\{1\}$, then $\varepsilon$ equals $J_{d,m}$ or $\mathrm{diag}(J_{d}\varpi_{E},...,J_{d}\varpi_{E}\epsilon_{0})$, where $\epsilon_{0}\in\mfo_{E}^{\times}\backslash\mfo_{E}^{\times2}$ and $\varpi_{E}$ is a uniformizer of $E$ satisfying Lemma \ref{lemmahassecase3a} with $E'=E$.
	
	If $\varepsilon=J_{d,m}$, by (\ref{eqgamma4value}) we have
	\begin{align}\label{eq3acond1}
		\,^{t}g^{-1}J_{d,m}g^{-1}=J_{d,m}\ \text{or}\ \mathrm{diag}(J_{d},...,J_{d},J_{d}\epsilon_{0})\ &\text{or}\ \mathrm{diag}(J_{d}\varpi_{E},...,J_{d}\varpi_{E},J_{d}\varpi_{E}) \nonumber \\  &\text{or}\ \mathrm{diag}(J_{d}\varpi_{E},...,J_{d}\varpi_{E},J_{d}\varpi_{E}\epsilon_{0})
	\end{align}
	
	Since the determinants of both sides of (\ref{eq3acond1}) are in $F^{\times2}$, and by Lemma \ref{lemmaGLnOFsymhasse}, Lemma \ref{lemmahassecase3a} and Corollary \ref{corcalhasse}, we have
	$$\mathrm{Hasse}(J_{d,m})=\mathrm{Hasse}(\mathrm{diag}(J_{d},...,J_{d},J_{d}\epsilon_{0}))=\mathrm{Hasse}(\mathrm{diag}(J_{d}\varpi_{E},...,J_{d}\varpi_{E},J_{d}\varpi_{E}))=1,$$
	and
	$$\mathrm{Hasse}(\mathrm{diag}(J_{d}\varpi_{E},...,J_{d}\varpi_{E},J_{d}\varpi_{E}\epsilon_{0}))=-1,$$
	then by Proposition \ref{propSGLEorbit} there exist $g_{0}=1$, $g_{1}$ and $g_{2}$ which satisfy equation (\ref{eq3acond1}) with the first three terms on the right separately. Furthermore, equation (\ref{eq3acond1}) with the last term on the right is never satisfied. Thus there are exactly three double cosets $\bs{J}G^{\tau}$, $\bs{J}g_{1}G^{\tau}$ and $\bs{J}g_{2}G^{\tau}$.
	
	If $\varepsilon=\mathrm{diag}(J_{d}\varpi_{E},...,J_{d}\varpi_{E}\epsilon_{0})$, then by (\ref{eqgamma4value}) we have
	\begin{align}\label{eq3acond2}
		\,^{t}g^{-1}\varepsilon g^{-1}=\varepsilon I_{m} \ \text{or}\ \varepsilon\mathrm{diag}(1,...,1,\epsilon_{0})\ &\text{or}\ \varepsilon\mathrm{diag}(\varpi_{E},...,\varpi_{E},\varpi_{E}) \nonumber \\  &\text{or}\ \varepsilon\mathrm{diag}(\varpi_{E},...,\varpi_{E},\varpi_{E}\epsilon_{0})
	\end{align}
	
	Since the determinants of both sides of (\ref{eq3acond2}) are in $F^{\times2}$, and by Lemma \ref{lemmaGLnOFsymhasse}, Lemma \ref{lemmahassecase3a} and Corollary \ref{corcalhasse}, we have
	\begin{align*}\mathrm{Hasse}(\varepsilon\mathrm{diag}(1,...,1,\epsilon_{0}))&=\mathrm{Hasse}(\varepsilon\mathrm{diag}(\varpi_{E},...,\varpi_{E},\varpi_{E}))\\
		&=\mathrm{Hasse}(\varepsilon\mathrm{diag}(\varpi_{E},...,\varpi_{E},\varpi_{E}\epsilon_{0}))=1,
	\end{align*}
	and
	$$\mathrm{Hasse}(\varepsilon)=-1.$$
	Then equation (\ref{eq3acond2}) is never satisfied with the last three terms on the right, and $g_{0}=1$ satisfies (\ref{eq3acond2}) with the first term on the right. Thus there is only one double coset $\bs{J}G^{\tau}$.
	
	\textbf{(\romannumeral3.b)} If $\mrn_{E/F}(E^{\times})F^{\times2}/F^{\times2}$ is not trivial, then $\varepsilon$ equals $J_{d,m}$ or $\mathrm{diag}(J_{d}\varpi_{E},...,J_{d}\varpi_{E}\epsilon_{0})$, where $\epsilon_{0}\in\mfo_{E}^{\times}$ and $\varpi_{E}$ is a uniformizer of $E$. Using the similar proof as \textbf{(\rn3.a)}, with Lemma \ref{lemmahassecase3a} replaced by Corollary $\ref{corhassecase3b}$, we can show that if $\varepsilon=J_{d,m}$, there are three double cosets $\bs{J}g_{0}G^{\tau}$, $\bs{J}g_{1}G^{\tau}$ and $\bs{J}g_{2}G^{\tau}$, where $g_{0}=1$, $g_{1}$ and $g_{2}$ are defined such that $\tau(g_{i})g_{i}^{-1}$ equal three of the four terms on the right side of equation (\ref{eqgamma4value}). If $\varepsilon=\mathrm{diag}(J_{d}\varpi_{E},...,J_{d}\varpi_{E}\epsilon_{0})$, then there is only one double coset $\bs{J}G^{\tau}$.
	
	We sum up the main result of this subsection as the following proposition:
	
	\begin{proposition}\label{propdiscorbit}
		
		\textbf{Case (\rn1)} When $\mrn_{T_{m}/F}(T_{m}^{\times})F^{\times2}/F^{\times2}=F^{\times}/F^{\times2}$, the only double coset contributing to the distinction  is $\bs{J}g_{0}G^{\tau_{0}}$, where we write $g_{0}=1$ here and after to normalize the notation;
		
		\textbf{Case (\rn2)} When $\mrn_{T_{m}/F}(T_{m}^{\times})F^{\times2}/F^{\times2}$ is a subgroup of $F^{\times}/F^{\times2}$ of order 2, if $G^{\tau_{0}}$ is quasisplit but not split, then the only double coset contributing to the distinction is $\bs{J}g_{0}G^{\tau_{0}}$; If $G^{\tau_{0}}$ is split, then there are two different double cosets $\bs{J}g_{0}G^{\tau_{0}}$ and $\bs{J}g_{1}G^{\tau_{0}}$ contributing to the distinction, where $\tau_{0}(g_{1})g_{1}^{-1}\in B^{\times}$;
		
		\textbf{Case (\rn3)} When $\mrn_{T_{m}/F}(T_{m}^{\times})F^{\times2}/F^{\times2}=\{1\}$, if $G^{\tau_{0}}$ is not quasisplit, then the only double coset contributing to the distinction is $\bs{J}g_{0}G^{\tau_{0}}$; If $G^{\tau_{0}}$ is split, then there are three different double cosets $\bs{J}g_{0}G^{\tau_{0}}$, $\bs{J}g_{1}G^{\tau_{0}}$ and $\bs{J}g_{2}G^{\tau_{0}}$ contributing to the distinction, where $\tau_{0}(g_{1})g_{1}^{-1},\tau_{0}(g_{2})g_{2}^{-1}\in B^{\times}$.
		
	\end{proposition}
	
	\begin{remark}
		
		The above proposition does not guarantee that every double coset as above corresponds to a distinguished space, and it says nothing about the dimension. However in the next section we will find out that each double coset indeed contributes to the distinction and the corresponding dimension is one respectively.
		
	\end{remark}
	
	\begin{remark}
		
		We may also give out all the maximal simple characters contained in $\pi$ that are $\tau_{0}$-selfdual. Let $\theta$ be a fixed maximal simple character such that $\theta\circ\tau_{0}=\theta^{-1}$. Any other maximal simple characters contained in $\pi$ can be written as $\theta^{g}$ with $g\in G$. Thus $\theta^{g}$ is $\tau_{0}$-selfdual if and only if $\gamma=\tau_{0}(g)g^{-1}$ normalizes $\theta$, that is, $\gamma\in\bs{J}$. Thus from the above argument, $g$ is in the same $\bs{J}$-$G^{\tau_{0}}$ double coset as one of the $g_{i}$ in Proposition \ref{propdiscorbit}. Thus one has a one-to-one correspondence between $\bs{J}$-$G^{\tau_{0}}$ double cosets in \emph{loc. cit.} and $G^{\tau_{0}}$-orbits of $\tau_{0}$-selfdual maximal characters contained in $\pi$.
		
	\end{remark}

	\section{Proofs of the main theorems}
	
	In this section, we finish the proof of our main theorems. Let $\pi$ be a given supercuspidal representation of $G$ and let $\tau$ be a given orthogonal involution on $G$. First of all, if $\pi$ is distinguished by $G^{\tau}$, then we restrict $\pi$ to $F^{\times}\cap G^{\tau}=\{1,-1\}$ which is contained in the centre of $G$ and we get $\omega_{\pi}(-1)=1$. So $\omega_{\pi}(-1)=1$ is indeed a necessary condition for $\pi$ to be distinguished by $G^{\tau}$. So from now on we assume further that $\pi$
	satisfies this condition.

	\subsection{Orthogonal groups contributing to the distinction of $\pi$}\label{subsectionorthgroupdisc}
	
	In this subsection, we first assume that $H=G^{\tau}$ satisfies the condition of Theorem \ref{thmdistcri}. From the proof of Theorem \ref{thmendotauselfdualstra}, $G^{\tau}$ is conjugate to $G^{\tau_{0}}$ with $\tau_{0}=\tau_{\varepsilon_{0}}$ defined as in the beginning of section \ref{sectiondistypethm}. Since the property of distinction does not depend on the choice of the representative of a $G$-conjugacy class, we may suppose $\tau=\tau_{0}$.
	
	We choose a $\tau$-selfdual simple type $(\bs{J},\Lambda)$ of $\pi$ as in section 5, then using the Mackey formula and Frobenius reciprocity we get
	$$\mrhom_{G^{\tau}}(\pi,1)\simeq\prod_{g\in\bs{J}\backslash G/G^{\tau}}\mrhom_{\bs{J}^{g}\cap G^{\tau}}(\Lambda^{g},1).$$
	In \S \ref{subsecdoublecosetdisc}, we studied all the possible double cosets that contribute to the distinction. By Proposition \ref{propdiscorbit}, we have
	$$\mrhom_{G^{\tau}}(\pi,1)\simeq\bigoplus_{g_{i}}\mrhom_{\bs{J}^{g_{i}}\cap G^{\tau}}(\Lambda^{g_{i}},1),$$
	where $g_{i}$ runs over a finite set of representatives, depending on \textbf{Case (\rn1), Case (\rn2)} or \textbf{Case (\rn3)} of \emph{loc. cit.}
	
	Moreover, we may write
	$$\Lambda\simeq\bs{\kappa}\otimes\bs{\rho},$$
	where by Proposition \ref{propkappaselfdual} we assume $\bs{\kappa}^{\tau\vee}\simeq\bs{\kappa}$, thus we also have $\bs{\rho}^{\tau\vee}\simeq\bs{\rho}$. By Proposition \ref{propkappa}, we get
	\begin{equation}\label{eqkappadisc}
		\mrdim_{\mbc}\mrhom_{\bs{J}^{g_{i}}\cap G^{\tau}}(\bs{\kappa}^{g_{i}},\chi_{i}^{-1})=1
	\end{equation}
	and
	$$\mrhom_{\bs{J}^{g_{i}}\cap G^{\tau}}(\Lambda^{g_{i}},1)\simeq\mrhom_{\bs{J}^{g_{i}}\cap G^{\tau}}(\bs{\kappa}^{g_{i}},\chi_{i}^{-1})\otimes\mrhom_{\bs{J}^{g_{i}}\cap G^{\tau}}(\bs{\rho}^{g_{i}},\chi_{i}),$$
	where $\chi_{i}$ is a quadratic character of $\bs{J}^{g_{i}}\cap G^{\tau}$.
	Thus to finish the proof for $\tau=\tau_{0}$, we only need to calculate
	$$\mrdim_{\mbc}\mrhom_{\bs{J}^{g_{i}}\cap G^{\tau}}(\bs{\rho}^{g_{i}},\chi_{i}).$$
	We define $\delta_{i}(x)=\gamma_{i}^{-1}\tau(x)\gamma_{i}$ for any $x\in G$ with $\gamma_{i}=\tau(g_{i})g_{i}^{-1}$, then by the exact definition of $\tau$ and $\delta_{i}$, the restriction of $\delta_{i}$ to $\mrgl_{m}(\bs{l})\simeq J/J^{1}$ is an orthogonal involution, and we denote by
	$\mrgl_{m}(\bs{l})^{\delta_{i}}$ the corresponding orthogonal group. So we have
	$$\mrhom_{\bs{J}^{g_{i}}\cap G^{\tau}}(\bs{\rho}^{g_{i}},\chi_{i})\simeq\mrhom_{J\cap G^{\delta_{i}}}(\rho,\chi_{i}^{g_{i}^{-1}})\simeq\mrhom_{\mrgl_{m}(\bs{l})^{\delta_{i}}}(\overline{\rho},\overline{\chi_{i}^{g_{i}^{-1}}}),$$
	where $\overline{\rho}$ and $\overline{\chi_{i}^{g_{i}^{-1}}}$ denote the representations of $J/J^{1}$ and $J\cap G^{\delta_{i}}/J^{1}\cap G^{\delta_{i}}$ whose inflations equal $\rho:=\bs{\rho}|_{J}$ and $\chi_{i}^{g_{i}^{-1}}$ respectively. Using (\ref{eqkappadisc}) we get $\omega_{\bs{\kappa}}(-1)=\chi_{i}^{g_{i}^{-1}}(-1)^{-1}$, where $\omega_{\bs{\kappa}}$ denotes the central character of $\bs{\kappa}$. By \cite{hakim2012distinguished}, Proposition 6.7,  $\mrhom_{\mrgl_{m}(\bs{l})^{\delta_{i}}}(\overline{\rho},\overline{\chi_{i}^{g_{i}^{-1}}})$ is non-zero and of dimension 1 if and only if $\omega_{\overline{\rho}}(-1)=\overline{\chi_{i}^{g_{i}^{-1}}}(-1)$, or equivalently
	\begin{equation}\label{eqomegachi}
		\omega_{\bs{\rho}}(-1)=\chi_{i}^{g_{i}^{-1}}(-1),
	\end{equation}
	where $\omega_{\overline{\rho}}$ and $\omega_{\bs{\rho}}$ denote the central character of $\overline{\rho}$ and $\bs{\rho}$ respectively. If we denote by $\omega_{\Lambda}$ and $\omega_{\pi}$ the central character of $\Lambda$ and $\pi$ respectively, then we get
	\begin{equation}\label{eqomegapi}
		\omega_{\pi}(-1)=\omega_{\Lambda}(-1)=\omega_{\bs{\kappa}}(-1)\omega_{\bs{\rho}}(-1)=\chi_{i}^{g_{i}^{-1}}(-1)^{-1}\omega_{\bs{\rho}}(-1),
	\end{equation}
	Combining (\ref{eqomegachi}) with (\ref{eqomegapi}), $\mrhom_{\mrgl_{m}(\bs{l})^{\delta_{i}}}(\overline{\rho},\overline{\chi_{i}^{g_{i}^{-1}}})$ is non-zero and of dimension 1 if and only if $\omega_{\pi}(-1)=1$. Thus we proved the
	``if" part of Theorem \ref{thmdistcri} and Theorem \ref{thmdistdim}.
	
	\subsection{Other orthogonal groups}\label{subsectionotherortho}
	
	In this subsection, we finish the proof of Theorem \ref{thmdistcri}, by showing that if $\pi$ is distinguished, then the corresponding orthogonal group must satisfy the condition of \emph{loc. cit.}
	
	Let $\tau(x)=\varepsilon^{-1}\,^{t}x^{-1}\varepsilon$ for $x\in G$ as an orthogonal involution and let $G^{\tau}$ be the corresponding orthogonal group. We assume $\varepsilon_{0}=J_{d,m}$ and we write $\tau_{0}=\tau_{\varepsilon_{0}}$. We choose $[\mfa,\beta]$, $\theta$ and $(\bs{J},\Lambda)$ as in section \ref{sectiondistypethm}.
	
	If $\pi$ is distinguished by $G^{\tau}$, then by 
	Theorem \ref{thmtaugginJ} and Lemma \ref{lemmadoublecoset}, there exists $g\in G$ such that $\gamma=u\tau(g)g^{-1}\in E^{\times}\mfb^{\times}$. 
	We define $$\delta(x)=(\tau(g)g^{-1})^{-1}\tau(x)\tau(g)g^{-1}=\gamma^{-1}\varepsilon_{0}^{-1}\,^{t}x^{-1}\varepsilon_{0}\gamma\quad\text{for any}\ x\in G$$ as an orthogonal involution of $G$, then we have
	$$J=\delta(J),\ J^{1}=\delta(J^{1}),\ J^{g}\cap G^{\tau}=(J\cap G^{\delta})^{g}\ \text{and}\ J^{1g}\cap G^{\tau}=(J^{1}\cap G^{\delta})^{g}.$$ By definition for $x\in B^{\times}$, we have $\delta(x)=\gamma^{-1}\,^{t_{E}}x^{-1}\gamma$, where $\,^{t_{E}}$ denotes the transpose with respect to $B\simeq\mrm_{m}(E)$.
	Since $\gamma\in E^{\times}\mfb^{\times}$, the restriction of $\delta$ induces an orthogonal involution on $\mrgl_{m}(\bs{l})\simeq J/J^{1}$. 
	
	\begin{lemma}
	
	For $\bs{l}_{m}/\bs{l}$ an extension of finite fields of odd characteristic of degree $m$ and $\delta$ an orthogonal involution on $\mrgl_{m}(\bs{l})$, there exists a $\delta$-split embedding $\bs{l}_{m}^{\times}\hookrightarrow\mrgl_{m}(\bs{l})$.
	
	\end{lemma}
	
	\begin{proof}
	
	Write $\delta(x)=\overline{\varepsilon}^{-1}\,^{t}x^{-1}\overline{\varepsilon}$ for any $x\in \mrgl_{m}(\bs{l})$ with $\overline{\varepsilon}$ symmetric.
	When $\overline{\varepsilon}=J_{m}$, it is exactly \cite{hakim2013tame}, Lemma 4.7. In general we fix a $\tau_{J_{m}}$-split embedding $\iota_{0}:\bs{l}_{m}^{\times}\hookrightarrow\mrgl_{m}(\bs{l})$. By Proposition \ref{propsymembedepi0} (or more precisely its proof, since right now we consider finite fields instead of non-archimedean local fields), we have that $\bs{l}_{m}^{\times}$ is $\delta$-split for  $\overline{\varepsilon}\in J_{m}\iota_{0}(\bs{l}_{m}^{\times})$. Calculating the determinant and using Lemma \ref{lemmaklerdman}, such $\delta$ ranges over both $\mrgl_{m}(\bs{l})$-classes of orthogonal involutions, which finishes the proof.
	
	\end{proof}
	
	Using the above lemma, there exists an embedding $\bs{l}_{m}^{\times}\hookrightarrow\mrgl_{m}(\bs{l})$ such that $\bs{l}_{m}^{\times}$ is $\delta$-split. Thus $\bs{l}_{m}^{\times}$ can be regarded as a $\delta$-split subgroup of $J$ via such an embedding. We denote by $E_{m}=E[\bs{l}_{m}^{\times}]$ the maximal unramified extension of degree $m$ over $E$ which is $\delta$-split, thus $E_{m}^{g}$ is $\tau$-split which is $F$-isomorphic to $E_{m}$. In other words, there exists an $F$-embedding $\iota:E_{m}\hookrightarrow\mrm_{n}(F)$ which is $\tau$-split. We have proved Proposition \ref{propintermid}.
	
	Using the results in \S \ref{subsectionorthgroupdisc} we know that $\pi$ is distinguished by $G^{\tau_{J_{n}}}$, thus we may in particular consider the argument above for $\varepsilon=J_{n}$ and we deduce that $E_{m}$ is $\tau_{J_{n}}$-split, that is, the condition of the following lemma is satisfied.
	
	\begin{lemma}[\cite{hakim2013tame}, Lemma 6.4]\label{lemmaTorbit}
		
		Assume that there exists a $J$-symmetric embedding $E_{m}\hookrightarrow\mrm_{n}(F)$. Then for $Y_{E_{m}/F}=E_{m}^{\times}/(E_{m}^{\times2}F^{\times})$ and $\mathcal{O}^{E_{m}}$ the set of $E_{m}^{\times}$-orbits of orthogonal involutions $\tau$ such that $E_{m}$ is $\tau$-split, the map
		$$\mu_{E_{m}/F}:Y_{E_{m}/F}\rightarrow\mathcal{O}^{E_{m}}$$
		which sends the coset of $x\in E_{m}^{\times}$ to the orbit of $\tau_{J_{n}x}$ is a bijection.
		
	\end{lemma}
	
	In particular we have $\tau_{J_{n}},\tau\in\mathcal{O}^{E_{m}}$. Since $E_{m}/T_{m}$ is totally wildly ramified, as in Lemma \ref{lemmaE/Lextension} it is easy to see that $E_{m}^{\times}/E_{m}^{\times2}F^{\times}\simeq T_{m}^{\times}/T_{m}^{\times2}F^{\times}$, and we denote by $y_{T_{m}/F}$ the corresponding cardinality. Thus by \cite{hakim2013tame}, Lemma 6.2, $y_{T_{m}/F}-1$ equals the number of quadratic extensions of $F$ contained in $T_{m}$. Furthermore by \cite{hakim2013tame}, Lemma 3.8 we have
	$$y_{T_{m}/F}=\begin{cases}
		1 & \quad \text{\textbf{Case (\rn1)}},\\
		2 & \quad \text{\textbf{Case (\rn2)}},\\
		4 & \quad \text{\textbf{Case (\rn3)}}.
	\end{cases}$$
	Thus in \textbf{Case (\rn1)}, we have $|\mathcal{O}^{E_{m}}|=1$, which means that $\mathcal{O}^{E_{m}}$ consists of the  $E_{m}^{\times}$-orbit represented by the split involution $\tau_{J_{n}}$, thus $G^{\tau}$ is split.
	In \textbf{Case (\rn2)}, we have $|\mathcal{O}^{E_{m}}|=2$. And by direct calculation,  $$\mathrm{det}(J_{n}E_{m}^{\times})F^{\times2}/F^{\times2}=(-1)^{n(n-1)/2}\mathrm{N}_{E_{m}/F}(E_{m}^{\times})F^{\times2}/F^{\times2}=(-1)^{n(n-1)/2}\mathrm{N}_{T_{m}/F}(T_{m}^{\times})F^{\times2}/F^{\times2},$$ which is of order 2. Thus $\mathcal{O}^{E_{m}}$ consists of two  $E_{m}^{\times}$-orbits, one of which is split, and the other is quasisplit but not split with the determinants of its corresponding symmetric matrices contained in  $(-1)^{n(n-1)/2}\mrn_{T_{m}/F}(T_{m}^{\times})F^{\times2}\backslash(-1)^{n(n-1)/2}F^{\times2}$. Thus $G^{\tau}$ is either split or quasisplit that satisfies the condition of Theorem \ref{thmdistcri}. In \textbf{Case (\rn3)}, $$\mathrm{det}(J_{n}E_{m}^{\times})F^{\times2}/F^{\times2}=(-1)^{n(n-1)/2}\mathrm{N}_{E_{m}/F}(E_{m}^{\times})F^{\times2}/F^{\times2}=\{1\}.$$
	Thus by Proposition \ref{propGconjortho}, $G^{\tau}$ is either split or non-quasisplit. Combining these three cases together, we have shown that $G^{\tau}$ must satisfy the condition of Theorem \ref{thmdistcri}, which finishes the ``only if" part of Theorem \ref{thmdistcri}.
	
\begin{remark}

Indeed the argument in this subsection is based on the existence of a certain $\gamma\in\bs{J}$ instead of the stronger condition that $\pi$ is $G^{\tau}$-distinguished. As a corollary if $G^{\tau}$ does not satisfy the condition of Theorem \ref{thmdistcri}, then any $\gamma=u\tau(g)g^{-1}$ is not contained in $\bs{J}$.

\end{remark}

\subsection{A variant of the main theorems}

Finally, the following variant of Theorem \ref{thmdistcri} and Theorem \ref{thmdistdim} is true.

\begin{theorem}\label{thmmaintwist}

For $\pi$ and $T_{m}$ as in Theorem \ref{thmdistcri}, $G^{\tau_{\varepsilon}}$ an orthogonal subgroup of $G$ and $\mu$ a character of $G^{\tau_{\varepsilon}}$ whose order is relatively prime to $p$, the distinguished space $\mrhom_{G^{\tau_{\varepsilon}}}(\pi,\mu)\neq 0$ if and only if the following two conditions are satisfied:

(1) $\omega_{\pi}(-1)=\mu(-1)$;

(2) Precisely one of the following conditions holds:

\begin{enumerate}
	
	\item $\mrn_{T_{m}/F}(T_{m}^{\times})F^{\times2}/F^{\times2}=F^{\times}/F^{\times2}$ and $G^{\tau_{\varepsilon}}$ is split;
	\item $\mrn_{T_{m}/F}(T_{m}^{\times})F^{\times2}/F^{\times2}$ is a subgroup of $F^{\times}/F^{\times2}$ of order $2$ and $G^{\tau_{\varepsilon}}$ is either split or quasisplit but not split such that $(-1)^{n(n-1)/2}\mrdet(\varepsilon)\in\mrn_{T_{m}/F}(T_{m}^{\times})-F^{\times2}$;
	\item $\mrn_{T_{m}/F}(T_{m}^{\times})F^{\times2}/F^{\times2}=\{1\}$ and $G^{\tau_{\varepsilon}}$ is either split or not quasisplit.
	
\end{enumerate}

Moreover, the dimension $\mrdim_{\mbc}\mrhom_{G^{\tau_{\varepsilon}}}(\pi,\mu)$ in the three cases is
\begin{enumerate}
	
	\item $1$;
	
	\item $1$ if $G^{\tau_{\varepsilon}}$ is not split and $2$ if $G^{\tau_{\varepsilon}}$ is split;
	
	\item $1$ if $G^{\tau_{\varepsilon}}$ is not split and $3$ if $G^{\tau_{\varepsilon}}$ is split.
	
\end{enumerate}

\end{theorem}	

\begin{proof}

We explain how the previous proofs can be used here. We let $\tau_{0}=\tau_{\varepsilon_{0}}$, $[\mfa,\beta]$, $\theta$, $\eta$, $(\bs{J},\Lambda)$, and $E=F[\beta]$ be defined exactly as in the beginning of section \ref{sectiondistypethm}. For $\tau=\tau_{\varepsilon}$, we write $u=\varepsilon_{0}^{-1}\varepsilon$ and $\gamma=u\tau(g)g^{-1}$. Using the Mackey formula and Frobenius reciprocity, we have
$$\mrhom_{G^{\tau}}(\pi,\mu)\simeq\prod_{g\in\bs{J}\backslash G/G^{\tau}}\mrhom_{\bs{J}^{g}\cap G^{\tau}}(\Lambda^{g},\mu).$$

If $\mrhom_{\bs{J}^{g}\cap G^{\tau}}(\Lambda^{g},\mu)\neq 0$, then restricting to $H^{1g}\cap G^{\tau}$ and noting that $\mu$ is trivial on $H^{1g}\cap G^{\tau}$, we have $\gamma\in JB^{\times}J$ by Proposition \ref{propthetadisc}. If we choose $\bs{\kappa}$ to be a $\tau_{0}$-selfdual extension of $\eta$ by Proposition \ref{propkappaselfdual}, then using Proposition \ref{propheisdist}, Proposition \ref{propkappa} and Proposition \ref{propchichi'}, there exists a quadratic character $\chi$ of $\bs{J}^{g}\cap G^{\tau}$, such that $$\mrhom_{\bs{J}^{g}\cap G^{\tau}}(\bs{\kappa},\chi^{-1})\simeq\mbc$$ 
and
$$\mathrm{Hom}_{\boldsymbol{J}^{g}\cap G^{\tau}}(\Lambda^{g},\mu)\simeq\mathrm{Hom}_{\bs{J}^{g}\cap G^{\tau}}(\bs{\kappa}^{g},\chi^{-1})\otimes\mathrm{Hom}_{\boldsymbol{J}^{g}\cap G^{\tau}}(\boldsymbol{\rho}^{g},\chi\mu).$$
Since the order of $\mu$ is relatively prime to $p$, using a similar argument to that in \S \ref{subsectionpfdisttype}, the space $\mathrm{Hom}_{\boldsymbol{J}^{g}\cap G^{\tau}}(\boldsymbol{\rho}^{g},\chi\mu)\neq 0$ implies that $\gamma\in \bs{J}$. Changing $g$ by another representative in the same $\bs{J}$-$G^{\tau}$ double coset, we may assume $\gamma\in E^{\times}\mfb^{\times}$ (\emph{cf.} Lemma \ref{lemmadoublecoset}).

When $\varepsilon_{0}=\varepsilon$ and $\tau=\tau_{0}$, using the classification result in \S \ref{subsecdoublecosetdisc} we have
$$\mrhom_{G^{\tau}}(\pi,\mu)\simeq\bigoplus_{g_{i}}\mrhom_{\bs{J}^{g_{i}}\cap G^{\tau}}(\bs{\kappa}^{g_{i}},\chi_{i}^{-1})\otimes\mrhom_{\bs{J}^{g_{i}}\cap G^{\tau}}(\bs{\rho}^{g_{i}},\chi_{i}\mu),$$
where $g_{i}$ runs over a finite set of representatives depending on \textbf{Case (\rn1), Case (\rn2)} or \textbf{Case (\rn3)} in Proposition \ref{propdiscorbit}, and $\chi_{i}$ is the quadratic character of $\bs{J}^{g_{i}}\cap G^{\tau}$ such that $\mrhom_{\bs{J}^{g_{i}}\cap G^{\tau}}(\bs{\kappa}^{g_{i}},\chi_{i}^{-1})\simeq\mbc$. Using a similar proof $\mrhom_{\bs{J}^{g_{i}}\cap G^{\tau}}(\bs{\rho}^{g_{i}},\chi_{i}\mu)$ is non-zero and of dimension 1 if and only if $\omega_{\bs{\rho}}(-1)=\chi(-1)\mu(-1)$, or equivalently $\omega_{\pi}(-1)=\mu(-1)$, which proves the ``if" part and verifies the corresponding dimension. The ``only if'' part follows exactly from \S \ref{subsectionotherortho}.

\end{proof}

\newcommand{\etalchar}[1]{$^{#1}$}

\end{document}